\documentclass[12pt,a4paper]{amsart}
\calclayout
\usepackage[utf8]{inputenc}
\usepackage[T1]{fontenc}
\usepackage{mathtools}
\mathtoolsset{showonlyrefs}
\usepackage{stackrel}
\usepackage{mathrsfs}
\usepackage{hyperref}
\usepackage{comment}
\usepackage{amsthm}
\theoremstyle{plain}
\newtheorem{thm}{Theorem}[section]
\newtheorem{lem}[thm]{Lemma}
\newtheorem{prop}[thm]{Proposition}
\newtheorem{cor}[thm]{Corollary}

\theoremstyle{definition}

\newtheorem{rem}[thm]{Remark}

\newtheorem{asp}{Assumption}
\numberwithin{equation}{section}
\allowdisplaybreaks
\begin{document}
\title
[stable CLT for super-OU processes]
{Stable Central Limit Theorems for Super Ornstein-Uhlenbeck Processes}
\author
[Y.-X. Ren, R. Song, Z. Sun and J. Zhao]
{Yan-Xia Ren, Renming Song, Zhenyao Sun and Jianjie Zhao}
\address{
  Yan-Xia Ren \\
  LMAM School of Mathematical Sciences \& Center for Statistical Science \\
  Peking University \\
  Beijing, P. R. China, 100871}
\email{yxren@math.pku.edu.cn}
\thanks{The research of Yan-Xia Ren is supported in part by NSFC (Grant Nos. 11671017  and 11731009) and LMEQF.}
\address{
  Renming Song \\
  Department of Mathematics \\
  University of Illinois at Urbana-Champaign \\
  Urbana, IL, USA, 61801}
\email{rsong@illinois.edu}
\thanks{The Research of Renming Song is support in part by a grant from the Simons Foundation (\#429343, Renming Song)}
\address{
  Zhenyao Sun \\
  School of Mathematics and Statistics\\
  Wuhan University \\
  Hubei, P. R. China, 100871}
\email{zhenyao.sun@gmail.com}
\address{
  Jianjie Zhao \\
  School of Mathematical Sciences \\
  Peking University \\
  Beijing, P. R. China, 100871}
\email{zhaojianjie@pku.edu.cn}
\thanks{Jianjie Zhao is the corresponding author}
\begin{abstract}
  In this paper, we study the asymptotic behavior of a supercritical $(\xi,\psi)$-superprocess $(X_t)_{t\geq 0}$
  whose underlying spatial motion $\xi$ is an Ornstein-Uhlenbeck process on $\mathbb R^d$ with generator $L = \frac{1}{2}\sigma^2\Delta - b x \cdot \nabla$ where $\sigma, b >0$;
  and whose branching mechanism $\psi$ satisfies Grey's condition and some perturbation condition which guarantees that,
 when $z\to 0$, $\psi(z)=-\alpha z + \eta z^{1+\beta} (1+o(1))$ with $\alpha > 0$, $\eta>0$ and $\beta\in (0, 1)$.
  Some law of large numbers and $(1+\beta)$-stable central limit theorems are established for
  $(X_t(f) )_{t\geq 0}$, where the function $f$ is
  assumed to be of polynomial growth.
A phase transition arises  for the central limit theorems in the sense 
that the forms of the central limit theorem are different in three different regimes corresponding the branching rate 
being relatively small, large or critical at a balanced value.
\end{abstract}
\maketitle
\section{Introduction}
\subsection{Motivation}
\label{subsec:M}
Let $d \in \mathbb N:= \{1,2,\dots\}$ and $\mathbb R_+:= [0,\infty)$.
Let $\xi=\{(\xi_t)_{t\geq 0}; (\Pi_x)_{x\in \mathbb R^d}\}$ be an $\mathbb R^d$-valued Ornstein-Uhlenbeck process (OU process) with generator
\begin{align}
  Lf(x)
  = \frac{1}{2}\sigma^2\Delta f(x)-b x \cdot \nabla f(x)
  , \quad  x\in \mathbb R^d, f \in C^2(\mathbb R^d),
\end{align}
where $\sigma > 0$ and $b > 0$ are constants.
Let $\psi$ be a function on $\mathbb R_+$ of the form
\begin{align}
  \label{eq: honogeneou branching mechanism}
  \psi(z)
  =- \alpha z + \rho z^2 + \int_{(0,\infty)} (e^{-zy} - 1 + zy)~\pi(dy)
  , \quad  z \in \mathbb R_+,
\end{align}
where $\alpha > 0 $, $\rho \geq0$ and $\pi$ is a measure on $(0,\infty)$ with $\int_{(0,\infty)}(y\wedge y^2)~\pi(dy)< \infty$.
$\psi$ is referred to as a branching mechanism and $\pi$ is referred to as the L\'evy measure of $\psi$.
Denote by $\mathcal M(\mathbb R^d)$ the space of all finite Borel measures on $\mathbb R^d$.
For $f,g\in \mathcal B(\mathbb R^d, \mathbb R)$ and $\mu \in \mathcal M(\mathbb R^d)$,
 write $\mu(f)= \int f(x)\mu(dx)$
and $\langle f, g\rangle = \int f(x)g(x) dx$ whenever the integrals make sense.
We say a real-valued Borel function $f:(t,x)\mapsto f(t,x)$ on $\mathbb R_+\times \mathbb R^d$ is \emph{locally bounded} if, for each $t\in \mathbb R_+$, we have $ \sup_{s\in [0,t],x\in \mathbb R^d} |f(s,x)|<\infty. $
We say that an $\mathcal M(\mathbb R^d)$-valued Hunt process $X = \{(X_t)_{t\geq 0}; (\mathbb{P}_{\mu})_{\mu \in \mathcal M(\mathbb R^d)}\}$
on  $(\Omega, \mathscr{F})$
is a \emph{super Ornstein-Uhlenbeck process (super-OU process)} with branching mechanism $\psi$, or a $(\xi, \psi)$-superprocess, if for each non-negative bounded Borel function $f$ on $\mathbb R^d$, we have
\begin{align}
  \label{eq: def of V_t}
  \mathbb{P}_{\mu}[e^{-X_t(f)}]
  = e^{-\mu(V_tf)}
  , \quad t\geq 0, \mu \in \mathcal M(\mathbb R^d),
\end{align}
where $(t,x) \mapsto V_tf(x)$ is the unique locally bounded non-negative solution to the equation
\begin{align}
  V_tf(x) + \Pi_x \Big[ \int_0^t\psi (V_{t-s}f(\xi_s) )~ds\Big]
	= \Pi_x [f(\xi_t)]
  , \quad x\in \mathbb R^d, t\geq 0.
\end{align}	
The existence of such super-OU process $X$ is well known, see \cite{Dynkin1993Superprocesses} for instance.

Recently, there have been quite a few papers on laws of large numbers for superdiffusions.
In \cite{Englander2009Law, EnglanderWinter2006Law, EnglanderTuraev2002A-scaling}, some weak laws of large numbers (convergence in law or in probability) were established.
The strong law of large numbers for superprocesses was first studied in \cite{ChenRenWang2008An-almost}, followed by \cite{ChenRenSongZhang2015Strong-law, ChenRenYang2019Skeleton, EckhoffKyprianouWinkel2015Spines, KouritzinRen2014A-strong, LiuRenSong2013Strong, Wang2010An-almost} under different settings.
For a good survey on recent developments in laws of large numbers for branching Markov processes and superprocesses, see \cite{EckhoffKyprianouWinkel2015Spines}.

The strong law of large numbers for the super-OU process $X$ above can be stated as follows:
Under some conditions on $\psi$ (these conditions are satisfied under our Assumptions 1 and 2 below),
there exists an $\Omega_0$ of $\mathbb{P}_\mu$-full probability for every $\mu\in\mathcal M(\mathbb R^d)$ such that on $\Omega_0$, for every Lebesgue-a.\/e. continuous bounded non-negative function $f$ on $\mathbb R^d$, we have
 $\lim_{t\to\infty} e^{-\alpha t} X_t(f) =H_\infty\langle f, \varphi\rangle $,
where $H_\infty$ is the limit of the martingale $e^{-\alpha t}X_t(1)$
and $\varphi$ is the invariant density of the OU process $\xi$ defined in \eqref{invariantdensity} below.
See \cite[Theorem 2.13 \& Example 8.1]{ChenRenYang2019Skeleton} and \cite[Theorem 1.2 \& Example 4.1]{EckhoffKyprianouWinkel2015Spines}.

In this paper, we will establish some spatial central limit theorems (CLTs) for the super-OU process $X$ above.
Our key assumption is that $\psi$ satisfies Grey's condition and some perturbation condition which guarantees that,
when  $z\to 0$, $\psi(z)=-\alpha z + \eta z^{1+\beta} (1+o(1))$ with $\alpha > 0$, $\eta>0$ and $\beta\in (0, 1)$.
Our  goal is to find $(F_t)_{t\geq 0}$ and $(G_t)_{t\geq 0}$ so that
$ (X_t(f) -G_t)/F_t $ converges weakly to some non-degenerate random variable as $t\rightarrow\infty$, for a large class of functions $f$.
Note that, in the setting of this paper, $X_t(f)$ typically has infinite second moment.

There are many papers on CLTs for branching processes, branching diffusions and superprocesses under the second moment condition.
See \cite{Heyde1970A-rate, HeydeBrown1871An-invariance, HeydeLeslie1971Improved} for supercritical Galton-Watson processes (GW processes),
\cite{KestenStigum1966Additional,KestenStigum1966A-limit} for supercritical multi-type GW processes, \cite{Athreya1969Limit,Athreya1969LimitB,Athreya1971Some}
for supercritical multi-type continuous time branching processes and \cite{AsmussenHering1983Branching} for general supercritical branching Markov processes under certain conditions.
Some spatial CLTs for supercritical branching OU processes with binary branching mechanism were proved in \cite{AdamczakMilos2015CLT} and some
spatial CLTs for supercritical super-OU processes with branching mechanisms satisfying a fourth moment condition were proved in \cite{Milos2012Spatial}.
These two papers made connections between CLTs and branching rate regimes.
Some spatial CLTs  for supercritical super-OU  processes with branching mechanisms satisfying only a second moment condition were established in \cite{RenSongZhang2014Central}.
Moreover, compared with the results of \cite{AdamczakMilos2015CLT,Milos2012Spatial}, the limit distributions in \cite{RenSongZhang2014Central} are non-degenerate.
Since then, a series of spatial CLTs for a large class of general supercritical branching Markov processes and superprocesses with spatially dependent branching mechanisms were proved in \cite{RenSongZhang2014CentralB,RenSongZhang2015Central,RenSongZhang2017Central}.
The functional version of the CLTs were established in \cite{Janson2004Functional} for supercritical multitype branching processes, and in \cite{RenSongZhang2017Functional} for supercritical superprocesses.

There are also many limit theorems for supercritical branching processes and branching Markov processes with branching mechanisms of infinite second moment.
Heyde \cite{Heyde1971Some} established some  CLTs for supercritical GW processes when the offspring distribution belongs to the domain of attraction of a stable law of index $\alpha\in (1, 2]$, and proved that the limit laws are stable laws.
Similar results  for supercritical multi-type GW processes and supercritical continuous time branching processes,
under some $p$-th ($p\in(1,2]$) moment condition on the offspring distribution, were given in Asmussen \cite{Asmussen76Convergence}.
 Recently, Marks and Milo\'s \cite{MarksMilos2018CLT} considered the limit behavior of supercritical branching OU processes with a special stable offspring distribution.
They established some spatial CLTs in the small and critical branching rate regimes, but they did not prove any CLT type result in the large branching rate regime.
We also mention here that very recently \cite{IksanovKoleskoMeiners2018Stable-like} considered stable fluctuations of Biggins' martingales in the context of branching random walks and \cite{RenSongSun2018Limit} considered the asymptotic behavior
of a class of critical superprocesses with spatially dependent stable branching mechanism.

As far as we know, this paper is the first to study spatial CLTs for supercritical superprocesses without the second moment condition.

\subsection{Main results}
\label{sec:I:R}
We will always assume that the following assumption holds.
\begin{asp}
  \label{asp: Greys condition}
  The branching mechanism $\psi$ satisfies Grey's condition, i.e., there exists $z' > 0$ such that $\psi(z) > 0$ for all $z>z'$ and  $\int_{z'}^\infty \psi(z)^{-1}dz < \infty$.
\end{asp}
For $\mu \in \mathcal M(\mathbb R^d)$, write $\|\mu\| = \mu(1)$.
It is known (see \cite[Theorems 12.5 \& 12.7]{Kyprianou2014Fluctuations} for example) that, under Assumption \ref{asp: Greys condition}, the \emph{extinction event} $D :=\{\exists t\geq 0,~\text{s.t.}~ \|X_t\| =0 \}$ has positive probability with respect to $\mathbb P_\mu$ for each  $\mu \in \mathcal M(\mathbb R^d)$.
In fact, $ \mathbb{P}_{\mu} (D) = e^{-\bar v \|\mu\|}$ where $ \bar v := \sup\{\lambda \geq 0: \psi(\lambda) = 0\} \in (0,\infty) $ is the largest root of $\psi$.

Denote by $\Gamma$ the gamma function.
For any $\sigma$-finite signed measure $\mu$, we use $|\mu|$ to denote the total variation measure of $\mu$.
In this paper, we will also assume the following:
\begin{asp}
  \label{asp: branching mechanism}
  There exist constants $\eta > 0$ and $\beta \in (0,1)$ such that
  \begin{align}
    \label{eq: asp of branching mechanism}
    \int_{(1,\infty)}y^{1+\beta +\delta}~\Big|\pi(dy)-\frac{\eta~dy}{\Gamma(-1-\beta)y^{2+\beta}}\Big| <\infty
  \end{align}
	for some $\delta > 0$.
\end{asp}
We will show in Subsection \ref{sec: branching mechanism} that if Assumption \ref{asp: branching mechanism} holds, then $\eta$ and $\beta$ are uniquely determined by the L\'evy measure $\pi$.
In the reminder of the paper, we will always use $\eta$ and $\beta$ to denote the constants in Assumption  \ref{asp: branching mechanism}.
	Note that $\delta$ is not uniquely determined by $\pi$.
	In fact, if $\delta>0$ is a constant such that \eqref{eq: asp of branching mechanism} holds, then replacing $\delta$ by any smaller positive number, \eqref{eq: asp of branching mechanism} still holds.
	Therefore, Assumption \ref{asp: branching mechanism} is equivalent to the following statement:
	There exist constants $\eta > 0$ and $\beta \in (0,1)$ such that, for all small enough $\delta>0$, \eqref{eq: asp of branching mechanism} holds.

\begin{rem}
  \label{rem:SP}
Roughly speaking, Assumption \ref{asp: branching mechanism} says that $\psi$ is ``not too far away'' from $\widetilde \psi(z) := - \alpha z + \eta z^{1+\beta}$ near $0$.
In fact, if we consider their difference
\begin{align}
  \label{eq:PB}
  & \psi_1(z)
  := \psi(z) - \widetilde \psi(z)
  \\ &= \rho z^2+ \int_{(0,\infty)}(e^{-yz}-1+yz) \Big(\pi(dy) - \frac{\eta~dy}{\Gamma(-1-\beta) y^{2+\beta}}\Big),
  \quad z\geq 0,
\end{align}
then it can be verified that (see Lemma \ref{lem:CEP} below) $\psi_1(z)/z^{1+\beta} \xrightarrow[z\to 0]{} 0$.
Therefore, we can write $ \psi(z)  = - \alpha z + z^{1+\beta}(\eta + o(1))$ as $z\to 0$.
One can further write that $\psi(z) = - \alpha z + z^{1+\beta} l(z)$ where $l$ is a function on $[0,\infty)$ which is slowly varying at $0$.
\end{rem}

\begin{rem}
It will be proved in Lemma \ref{lem: LlogL criterion} that, under Assumption \ref{asp: branching mechanism},
 $\psi$ satisfies the $L \log L$ condition, i.e., $ \int_{(1,\infty)} y\log y~\pi(dy) < \infty. $
This guarantees that $H_\infty$, the limit of the non-negative martingale $(e^{-\alpha t} \|X_t\|)_{t\geq 0}$, is non-degenerate.
\end{rem}
Let us introduce some notation in order to give the precise formulation of our main result.
Denote by $\mathcal B(\mathbb R^d, \mathbb R)$ the space of all $\mathbb R$-valued Borel functions on $\mathbb R^d$.
Denote by $\mathcal B(\mathbb R^d, \mathbb R_+)$ the space of all $\mathbb R_+$-valued Borel functions on $\mathbb R^d$.
We use  $(P_t)_{t\geq 0}$ to denote the transition semigroup of $\xi$.	
Define
\(
P^{\alpha}_t f(x)
  := e^{\alpha t} P_t f(x)
  = \Pi_x [e^{\alpha t}f(\xi_t)]
\)
for each $x\in \mathbb R^d$, $t\geq 0$ and $f\in \mathcal B(\mathbb R^d, \mathbb R_+)$.
It is known that, see \cite[Proposition 2.27]{Li2011Measure-valued} for example, $(P^\alpha_t)_{t\geq 0}$ is the \emph{mean semigroup} of $X$ in the sense that
\(
  \mathbb{P}_{\mu}[X_t (f)]  = \mu( P^\alpha_t f)
\)
for all $\mu\in \mathcal M(\mathbb R^d)$, $t\geq 0$ and $f\in \mathcal B(\mathbb R^d, \mathbb R_+)$.

The limit behavior of $X$  is closely related to the spectral property of the OU semigroup $(P_t)_{t\geq 0}$ which we now recall (See \cite{MetafunePallaraPriola2002Spectrum} for more details).
It is known that the OU process $\xi$ has an invariant probability on $\mathbb R^d$
\begin{align}
  \label{invariantdensity}
  \varphi(x)dx
  :=\Big (\frac{b}{\pi \sigma^2}\Big )^{d/2}\exp \Big(-\frac{b}{\sigma^2}|x|^2 \Big)dx
\end{align}
which is a   symmetric multivariate Gaussian distribution.
Let $L^2(\varphi)$ be the Hilbert space with inner product
\begin{align}
  \langle f_1, f_2 \rangle_{\varphi}
  := \int_{\mathbb R^d}f_1(x)f_2(x)\varphi(x) dx, \quad f_1,f_2 \in L^2(\varphi).
\end{align}
Let $\mathbb Z_+ := \mathbb N\cup\{0\}$.
For each $p = (p_k)_{k = 1}^d \in \mathbb{Z}_+^{d}$, write $|p|:=\sum_{k=1}^d p_k$, $p!:= \prod_{k= 1}^d p_k!$ and $\partial_p:= \prod_{k = 1}^d(\partial^{p_k}/\partial x_k^{p_k})$.
The \emph{Hermite polynomials} are defined by
\begin{align}
  H_p(x)
  :=(-1)^{|p|}\exp(|x|^2) \partial_p \exp(-|x|^2)
  , \quad x\in \mathbb R^d, p \in \mathbb{Z}_+^{d}.
\end{align}
It is known that $(P_t)_{t\geq 0}$ is a strongly continuous semigroup in $L^2(\varphi)$ and its generator $L$ has discrete spectrum $\sigma(L)= \{-bk: k \in \mathbb Z_+\}$.
For $k \in \mathbb Z_+$, denote by $\mathcal{A}_k$ the eigenspace corresponding to the eigenvalue $-bk$, then $ \mathcal{A}_k = \operatorname{Span} \{\phi_p : p\in \mathbb Z_+^d, |p|=k\}$ where
\begin{align}
  \label{eigenfunction}
  \phi_p(x)
  := \frac{1}{\sqrt{ p! 2^{|p|} }} H_p \Big(\frac{ \sqrt{b} }{\sigma}x \Big)
  , \quad x\in \mathbb R^d, p\in \mathbb Z_+^d.
\end{align}
In other words,
\(
  P_t\phi_p(x)
  = e^{-b|p|t}\phi_p(x)
\)
for all $t\geq 0$, $x\in \mathbb R^d$ and $p\in \mathbb Z_+^d$.
Moreover, $\{\phi_p: p \in \mathbb Z_+^d\}$ forms a complete orthonormal basis of $L^2(\varphi)$.
Thus for each $f\in L^2(\varphi)$, we have
\begin{align}
  \label{semicomp1}
  f
  = \sum_{k=0}^{\infty}\sum_{p\in \mathbb Z_+^d:|p|=k}\langle f, \phi_p \rangle_{\varphi} \phi_p
  , \quad \text{in~} L^2(\varphi).
\end{align}
For each function $f\in L^2(\varphi)$, define the order of $f$ as
\[
  \kappa_f
  := \inf \left \{k\geq 0: \exists ~ p\in \mathbb Z_+^d , {\rm ~s.t.~} |p|=k {\rm ~and~}  \langle f, \phi_p \rangle_{\varphi}\neq 0\right \}
\]
which is the lowest non-trivial frequency in the eigen-expansion \eqref{semicomp1}.
Note that $ \kappa_f\geq 0$ and that, if $f\in L^2(\varphi)$ is non-trivial, then $\kappa_f<\infty$.
In particular, the order of any constant non-zero function is zero.

Denote by $\mathcal M_c(\mathbb R^d)$ the space of all finite Borel measures of compact support on $\mathbb R^d$.
For $p\in \mathbb{Z}_+^d$, define
\(
  H_t^p
  := e^{-(\alpha-|p|b)t}X_t(\phi_p)
\)
for all $t\geq 0$.
If $\alpha \tilde \beta>|p|b, \tilde \beta := \beta/(1+\beta)$, then for all $\gamma\in (0, \beta)$ and $\mu\in \mathcal M_c(\mathbb R^d)$, we will prove in Lemma \ref{lem:M:L:ML} that $(H_t^p)_{t\geq 0}$ is a $\mathbb{P}_{\mu}$-martingale bounded in $L^{1+\gamma}(\mathbb{P}_{\mu})$.
Thus the limit $H^p_{\infty}:=\lim_{t\rightarrow \infty}H_t^p$ exists $\mathbb{P}_{\mu}$-almost surely and in $L^{1+\gamma}(\mathbb{P}_{\mu})$.

We first present a law of large numbers for our model which extends the strong laws of large numbers of \cite{ChenRenYang2019Skeleton, EckhoffKyprianouWinkel2015Spines} in which the first order asymptotic ($\kappa_f=0$) was identified.
Denote by $\mathcal P$ the class of functions of polynomial growth on $\mathbb R^d$, i.e.,
\begin{align}
  \label{eq: polynomial growth function}
  \mathcal{P}
  := \{f\in \mathcal B(\mathbb R^d, \mathbb R):\exists C>0, n \in \mathbb Z_+ \text{~s.t.~} \forall x\in \mathbb R^d, |f(x)|\leq C(1+|x|)^n \}.
\end{align}
It is clear that $\mathcal{P} \subset L^2(\varphi)$.
\begin{thm}
  \label{thm: law of large number}
  If $f \in \mathcal{P}$ satisfies $\alpha\tilde \beta>\kappa_f b$, then for all $\gamma\in (0, \beta)$ and  $\mu\in \mathcal M_c(\mathbb R^d)$,
  \[
    e^{-(\alpha-\kappa_fb)t}X_t(f)
    \xrightarrow[t\to \infty]{}\sum_{p\in \mathbb Z_+^d:|p|=\kappa_f}\langle f, \phi_p\rangle_{\varphi} H_{\infty}^p
    \quad in~ L^{1+\gamma}(\mathbb{P}_{\mu}).
  \]
  Moreover, if $f$ is twice differentiable and all its second order partial derivatives are in $\mathcal{P}$, then we also have almost sure convergence.
\end{thm}
If $f\in \mathcal B(\mathbb R^d, \mathbb R_+)$ is non-trivial and  bounded, then $\kappa_f=0$.
Hence, Theorem \ref{thm: law of large number} says that for any $\gamma\in (0, \beta)$ and  $\mu\in \mathcal M_c(\mathbb R^d)$, as $t\rightarrow \infty$,
\(
  e^{-\alpha t}X_t(f)
  \rightarrow \langle f, \varphi\rangle H_{\infty}
\)
in $L^{1+\gamma}(\mathbb{P}_{\mu})$.
Moreover, if $f$ is twice differentiable and all its second order partial derivatives are in $\mathcal{P}$, then we also have a.s.\ convergence.
However, to get a.s.\ convergence for bounded non-negative
Lebesgue-a.e.\ continuous functions $f$, we do not need $f$ to be twice differentiable.
See \cite[Theorem 2.13 \& Example 8.1]{ChenRenYang2019Skeleton} and \cite[Theorem 1.2 \& Example 4.1]{EckhoffKyprianouWinkel2015Spines}.

For the rest of this subsection, we focus on the CLTs of $X_t(f)$ for a large collection of $f\in \mathcal P\setminus \{0\}$.
Write $\tilde u = \frac{u}{ 1+ u}$ for each $u \neq -1$.
It turns out that there is a phase transition in the sense that the results are different in the following three cases:
\begin{enumerate}
\item
  the small branching rate case where
$f$ satisfies $\alpha \tilde \beta < \kappa_f b$;
\item
  the critical branching rate case where
$f$ satisfies $\alpha \tilde \beta = \kappa_f b$; and
\item
  the large branching rate case  where
$f$ satisfies $\alpha \tilde \beta > \kappa_f b$.
\end{enumerate}
Here, the small (resp. large) branching rate case means that the branching rate $\alpha$ is small (resp. large) compared to $\kappa_f$;
 and the critical branching rate means that the branching rate $\alpha$ is at a critical balanced value compared to $\kappa_f$.
To present our result, we define a family of operators $(T_t)_{t\geq 0}$ on $\mathcal P$ by
\begin{align}
  \label{eq:I:R:1}
  T_t f
  := \sum_{p \in \mathbb Z_+^d} e^{-| |p|b - \alpha \tilde \beta |t} \langle f, \phi_p \rangle_{\varphi} \phi_p
  ,\quad t\geq 0, f\in \mathcal P,
\end{align}
and a family of $\mathbb C$-valued functionals $(m_t)_{0 \leq t < \infty}$ on $\mathcal P$ by
\begin{align}
  \label{eq:I:R:2}
  m_t[f]
  := \eta \int_0^t ~du \int_{\mathbb R^d} (-iT_u f(x))^{1+\beta} \varphi(x) ~dx
  , \quad 0 \leq t< \infty, f\in \mathcal P.
\end{align}
Define $ \mathcal C_s := \mathcal P \cap \overline{\operatorname{Span}} \{ \phi_p: \alpha \tilde \beta < |p| b \}$, $\mathcal C_c   := \mathcal P \cap \operatorname{Span} \{ \phi_p : \alpha \tilde \beta = |p| b \} $
and $ \mathcal C_l   := \mathcal P \cap \operatorname{Span} \{ \phi_p: \alpha \tilde \beta > |p| b \}$. 
Note that $\mathcal C_s$ is an infinite dimensional space, $ \mathcal C_l$ and $\mathcal C_c$
are finite dimensional spaces, and $\mathcal C_c$ might be empty.
For $f\in \mathcal P\setminus \{0\}$, in Lemma \ref{lem:m} and Proposition \ref{prop:PL:S} below, we will show that
\begin{align}
  \label{eq:I:R:3}
  m[f]
  := \begin{cases}
    \lim_{t\to \infty} m_t[f], &
    f \in \mathcal C_s \oplus \mathcal C_l, \\
    \lim_{t\to \infty} \frac{1}{t} m_t[f], & f\in \mathcal P \setminus \mathcal C_s \oplus \mathcal C_l,
  \end{cases}
\end{align}
is well defined, and moreover, there exists a $(1+\beta)$-stable random variable $\zeta^f$ with characteristic function $\theta \mapsto e^{m[\theta f]}$.
The main result of this paper is as follows.

\begin{thm}
  \label{thm:M}
 If $\mu\in \mathcal M_c(\mathbb R^d)\setminus \{0\}$, then under $\mathbb{P}_{\mu}(\cdot|D^c)$, the following hold:
\begin{enumerate}
\item
  \label{thm:M:1}
  if $f\in \mathcal C_s\setminus\{0\}$, then $\|X_t\|^{- \frac{1}{1+\beta}} X_t(f)  \xrightarrow[t\to \infty]{d} \zeta^f$;
\item
  \label{thm:M:2}
  if $f\in \mathcal C_c\setminus\{0\}$, then $ \|t X_t\|^{-\frac{1}{1+\beta}} X_t(f) \xrightarrow[t\to \infty]{d} \zeta^f$;
\item
  \label{thm:M:3}
  if $f\in \mathcal C_l\setminus\{0\}$, then
  \[
    \|X_t\|^{-\frac{1}{1+\beta}} \Big( X_t(f) - \sum_{p\in \mathbb Z^d_+:\alpha \tilde \beta>|p|b}\langle f,\phi_p\rangle_\varphi e^{(\alpha-|p|b)t}H^p_{\infty}\Big)
    \xrightarrow[t\to \infty]{d}
    \zeta^{-f}.
  \]
\end{enumerate}
\end{thm}

At this point, we should mention that the theorem above does not cover all $f\in \mathcal P$.
Theorem \ref{thm:M}.(1) can be rephrased as if $f\in \mathcal P\setminus\{0\}$ satisfies   $\alpha \tilde \beta < \kappa_f b$, then  under $\mathbb{P}_{\mu}(\cdot|D^c)$, $\|X_t\|^{- \frac{1}{1+\beta}} X_t(f)  \xrightarrow[t\to \infty]{d} \zeta^f$.
Combining the first two parts of Theorem \ref{thm:M}, one can easily get that  if $f\in \mathcal P$ satisfies   $\alpha \tilde \beta = \kappa_f b$, then  under $\mathbb{P}_{\mu}(\cdot|D^c)$,
$ \|t X_t\|^{-\frac{1}{1+\beta}} X_t(f) \xrightarrow[t\to \infty]{d} \zeta^f$.
A general  $f \in \mathcal P$ can be decomposed as $f_s + f_c + f_l$ with $f_s \in \mathcal C_s$, $f_c \in \mathcal C_c$ and $f_l \in \mathcal C_l$. 
For $f\in  \mathcal P$ satisfying $\alpha \tilde \beta > \kappa_f b$, $f_s$ and $f_c$ maybe non-trivial.
In this case, we do not have a CLT yet.
We conjecture that the limit random variables in Theorem \ref{thm:M} for $ f\in \mathcal C_s$, $f\in \mathcal C_c$ and $ f\in \mathcal C_l$ are independent. If this is valid, we can
get a CLT for $ X_t(f)$ for all $f\in  \mathcal P$.
This independence is valid  under the second moment condition, see \cite{RenSongZhang2015Central}.
We leave the question of independence of the limit stable random variables to a future project.

 We now give some intuitive explanation of the branching rate regimes and the  phase transition.
Similar explanation has been given in the context of branching-OU processes, see \cite{MarksMilos2018CLT}.
First we mention that a super-OU process arises as the ``high density'' limit of a sequence of branching-OU processes, see \cite{Li2011Measure-valued} for example. A superprocess can be thought of as a cloud of infinitesimal branching ``particles'' moving in space.
The phase transition is due to an interplay of two competing effects in the system: coarsening and smoothing.
The coarsening effect corresponds to the increase of the spatial inequality and is a consequence of the branching: simply an area with more particles will produce more offspring.
The smoothing effect corresponds to the decrease of the spatial inequality and is a consequence of the mixing property of the OU processes: each OU ``particle'' will ``forget'' its initial position exponentially fast.

Let us consider $X_t(\phi_p)$ as an example and discuss how the parameters $\alpha, \beta, b$ and $|p|$
influence those two effects:
\begin{itemize}
\item
  The branching rate $\alpha$ captures the mean intensity of the branching in the system.
  Therefore, the lager the branching rate $\alpha$, the stronger the coarsening effect.
\item
  The tail index $\beta$ describes the heaviness of the tail of the offspring distribution which belongs to the domain of attraction of some $(1+\beta)$-stable random variable.
When $\beta$ is smaller i.e. the tail is heavier, then it is more likely that 
one particle can suddenly have a large amount of offspring.
In other words, the larger the tail index $\beta$, the smaller the fluctuation of offspring number, and then the stronger the coarsening effect.
\item
 The drift parameter $b$ is related to the level of the mixing property of the OU particles.
  The larger the drift parameter $b$, the faster the OU-particles forgetting their initial position, and therefore the stronger the smoothing effect.
\item
    The order $|p|$ is related to the capability of $\phi_p$ capturing the mixing property of the OU particles.
  In particular, in the case that $|p| = 0$, no mixing property can be captured by $\phi_p \equiv 1$ since we are only considering the total mass $\|X_t\|$.
  In general, the higher the order $|p|$, the more mixing property can be captured by $\phi_p$, and therefore the stronger the smoothing effect.
\end{itemize}
Here we discuss the role of the other parameters $\rho, \eta$ and $\sigma$ in our model:
\begin{itemize}
\item
  The coefficient $\rho$ dose not influence the result since $\rho z^2$ in the branching mechanism $\psi$ is a part of the small perturbation $\psi_1$
  (see Remark \ref{rem:SP}).
\item
  The coefficients $\eta$ and $\sigma$ are hidden in the definition of the functional $m[f]$, and therefore influence the actual distribution of the limiting $(1+\beta)$-stable random variable $\xi^f$.
  Their role in the coarsening and smoothing effects are negligible compared to the four parameters $\alpha, \beta, b$ and $|p|$ mentioned above.
\end{itemize}

\subsection{An outline of the methodology}
Let us give some intuitive explanation of the methodology used in this paper.
For any $\mu\in \mathcal M_c(\mathbb R^d)$ and any random variable $Y$ with finite mean, we define
$
  \mathcal I_s^t Y
  := \mathcal I_s^t [Y, \mu]
  := \mathbb P_\mu[Y|\mathscr F_t] - \mathbb P_\mu[Y|\mathscr F_s]
$
  where $0 \leq s \leq t <\infty.$
We will use the shorter notation $\mathcal I_s^t Y$ when there is no danger of confusion.
For $f\in \mathcal{P}$, consider the following decomposition over the time interval $[0,t]$:
\begin{align}
  X_t(f)
  := \sum_{k=0}^{\lfloor t \rfloor-1} \mathcal I_{t-k-1}^{t-k} X_t (f)+\mathcal I_0^{t-\lfloor t \rfloor} X_t(f) + X_0( P^\alpha_tf),
  \quad t\geq 0.
\end{align}
To find the fluctuation of $X_t(f)$, we will investigate the fluctuation of each term on the right hand side above.
The second term and third term are negligible after the rescaling, and for the first term we will establish
a multi-variate unit interval CLT  which says that
\[
  \Big( \|X_t\|^{-\frac{1}{1+\beta}}\mathcal I^{t-k}_{t-k-1} X_t(f) \Big)_{k=0}^n
  \xrightarrow [t\to \infty]{d} (\zeta^f_k)_{k=0}^n,
\]
where $(\zeta^f_k)_{k \in \mathbb N}$ are independent $(1+\beta)$-stable random variables.
If $f \in \mathcal C_s\setminus\{0\}$, then it can be argued that $\sum_{k=0}^{\lfloor t \rfloor} \zeta^f_k \xrightarrow[t\to \infty]{d} \zeta^f$ and then intuitively we have
\(
  \|X_t\|^{-\frac{1}{1+\beta}}  X_t(f)
  \xrightarrow[t\to \infty]{d} \zeta^f.
  \)
If $f \in \mathcal C_c \setminus \{0\}$, then it can be argued that
\(
t^{-\frac{1}{1+\beta}} \sum_{k=0}^{\lfloor t\rfloor} \zeta_k \xrightarrow[t\to \infty]{ d} \zeta^f
\)
and then intuitively we have
\(
\|tX_t\|^{-\frac{1}{1+\beta}}  X_t(f)
\xrightarrow[t\to \infty]{d} \zeta^f.
\)
If $f\in \mathcal C_l$, the general idea is almost the same, except that we need to consider the decomposition over the time interval $[t,\infty)$.

This paper is our first attempt on stable CLTs for superprocesses.
There are still many open questions.
Ren, Song and Zhang have established some spatial  CLTs in \cite{RenSongZhang2015Central} for a class of superprocesses with general spatial motions under
the assumption that the branching mechanisms satisfy a second moment condition.
We hope to prove spatial CLTs for superprocesses with general motions without the second moment assumption on the branching mechanism in a future project.

Recall that our Assumption \ref{asp: branching mechanism} says that the branching mechanism $\psi$ is $-\alpha z +\eta z^{1+\beta}$ plus a small perturbation
$\psi_1(z)$
which satisfies \eqref{eq: asp of branching mechanism} with some $\delta>0$.
It would be interesting to consider more general branching mechanisms.

The following correspondence between (sub)critical branching mechanisms and Bernstein functions is well known, see, for instance,
\cite[Theorem VII.4(ii)]{Bertoin} and \cite[Proposition 7]{BRY}. Suppose that $f, g:(0, \infty)\to [0, \infty)$ are related by $f(x)=xg(x)$.
Then $f$ is a (sub)critical branching mechanism with $\lim_{x\to 0}f(x)=0$ iff $g$ is a Bernstein function with a decreasing L\'evy density.
We now use this correspondence to give some examples of branching mechanisms satisfying Assumptions \ref{asp: Greys condition} and \ref{asp: branching mechanism}.
If $h$ is a complete Bernstein function which is regularly varying at 0 with index $\beta_1\in (\beta, 1)$, then
\[
  \psi(z)
  := -\alpha z + \rho z^2+\eta z^{1+\beta}+zh(z)
  , \qquad z>0,
\]
satisfies Assumptions \ref{asp: Greys condition} and \ref{asp: branching mechanism}.
If $\beta_1\in (\beta, 1)$, $c_1\in (0, \eta/\Gamma(-1-\beta))$ and $c_2\ge 1$, then
\[
  \psi(z)
  :=-\alpha z + \rho z^2+\eta z^{1+\beta}-\int^\infty_{c_2} (e^{-yz}-1+yz)\frac{c_1dy}{y^{1+\beta_1}}
  , \qquad z\in \mathbb R_+,
\]
satisfies Assumptions \ref{asp: Greys condition} and \ref{asp: branching mechanism}.

The rest of the paper is organized as follows:
In Subsection \ref{sec: branching mechanism} we will give some preliminary results for the branching mechanism $\psi$.
In Subsections \ref{sec: controller} and \ref{sec: h-controller} we will give some estimates for some operators related to the super-OU process $X$.
In Subsection \ref{sec: stable distributions} we will give the definitions of the $(1+\beta)$-stable random variables involved in this paper.
In Subsection \ref{sc:refined} we will give some refined estimate for the OU semigroup.
In Subsection \ref{sec: Small value probability} we will give some estimates for the small value probability of continuous state branching processes.
In Subsection \ref{sec: Moments for super-OU processes} we will give upper bounds for the $(1+\gamma)$-moments for our superprocesses.
These estimates and upper bounds will be crucial in the proofs of our main results.
In Subsection \ref{sec: large rate lln}, we will give the proof of Theorem \ref{thm: law of large number}.
In Subsections \ref{sec:critical}--\ref{sec: large rate clt}, we will give the proof of Theorem \ref{thm:M}.
In the Appendix, we consider a general superprocess $(X_t)_{t\geq 0}$ and we prove that the characteristic exponent of $X_t(f)$ satisfies a complex-valued non-linear integral equation.
This fact will be used at several places in this paper, and we think it is of independent interest.

\section{Preliminaries}
\subsection{Branching mechanism}
\label{sec: branching mechanism}
Let $\psi$ be the branching mechanism given in \eqref{eq: honogeneou branching mechanism}.
Suppose that Assumptions \ref{asp: Greys condition} and \ref{asp: branching mechanism} hold.
In this subsection, we give some preliminary results on $\psi$.
Recall that $\eta$ and $\beta$ are the constants in Assumption \ref{asp: branching mechanism}.
Let $\mathbb C_+:= \{x+iy: x\in \mathbb R_+, y \in \mathbb R\}$ and $\mathbb C^0_+:= \{x+iy: x\in (0,\infty), y \in \mathbb R\}$.
\begin{lem}
  \label{lem:CEP}
	The function $\psi_1$ given by \eqref{eq:PB} can be uniquely extended as a complex-valued continuous function on $\mathbb C_+$ which is holomorphic on $\mathbb C^0_+$.
  Moreover, for all $\delta > 0$ small enough, there exists $C>0$ such that for all $z\in \mathbb C_+$, we have $|\psi_1(z)| \leq C |z|^{1+\beta+\delta} + C|z|^2.$
\end{lem}
\begin{proof}
  According to Lemma \ref{lem: extension lemma for branching mechanism} below and the uniqueness of holomorphic extensions, we know that $\psi_1$ can be uniquely extended as a complex-valued continuous function on $\mathbb C_+$ which is holomorphic on $\mathbb C^0_+$.
	The extended $\psi_1$ has the following form:
  \[
    \psi_1(z)
    = \rho z^2 + \int_{(0,\infty)}(e^{-yz}-1+yz) \Big(\pi(dy) - \frac {\eta~dy} {\Gamma(-1-\beta)y^{2+\beta}} \Big)
    , \quad z\in \mathbb C_+.
  \]
	Now, according to  Assumption \ref{asp: branching mechanism}, for all small enough $\delta > 0$, we have
  \begin{align}
    |\psi_1(z)|
    & \leq \rho |z|^2 + \int_{(0,\infty)} (|yz|\wedge |yz|^2) \Big|\pi(dy) - \frac{\eta~dy}{\Gamma(-1-\beta)y^{2+\beta}}\Big| \\
    & \leq  |z|^2 \Big(\rho + \int_{(0,1)} y^2 \Big|\pi(dy) - \frac{\eta~dy}{\Gamma(-1-\beta)y^{2+\beta}}\Big|\Big) \\
    & \quad + |z|^{1+\beta +\delta}\int_{(1,\infty)} y^{1+\beta + \delta} \Big|\pi(dy) - \frac{\eta~dy}{\Gamma(-1-\beta)y^{2+\beta}}\Big|,
      \quad z \in \mathbb C_+,
  \end{align}
	as desired.
\end{proof}
The following lemma says that the constants $\eta, \beta$ in Assumption \ref{asp: branching mechanism} are uniquely determined by the L\'evy measure $\pi$.
\begin{lem}
  \label{lem: unique of beta and eta}
  Suppose Assumption  \ref{asp: branching mechanism} holds. Suppose that there are $\eta', \delta'>0$ and $\beta'\in (0,1)$ such that
  \[
    \int_{(1,\infty)} y^{ 1 + \beta'  + \delta' }~ \Big| \pi(dy) - \frac {\eta' ~dy} {\Gamma (- 1 - \beta ) y^{2 + \beta'}} \Big|
    < \infty.
  \]
	Then $\eta'= \eta$ and $\beta ' = \beta$.
\end{lem}
\begin{proof}
	Without loss of generality, we assume that $\beta+\delta \leq \beta'+ \delta'$.
	Using  the fact that $y^{1+\beta+ \delta} \leq y^{1+\beta'+\delta'}$ for $y \geq 1$, we get
  \[
    \int_{(1, \infty)} y^{1 + \beta + \delta}   \Big| \pi(dy) - \frac {\eta' ~dy} {\Gamma( - 1 - \beta)y^{2 + \beta'}} \Big|
    < \infty .
  \]
	Comparing this with Assumption \ref{asp: branching mechanism}, we get
  \[
    \int_{(1,\infty)} y^{ 1 + \beta + \delta} \Big| \frac { \eta ~dy} {\Gamma (- 1 - \beta) y^{2 + \beta}} - \frac {\eta' ~dy} {\Gamma (- 1 - \beta) y^{2 + \beta'}} \Big| < \infty.
  \]
	In other words, if we denote by $\widetilde \pi(dy)$ the measure $\eta' \Gamma(-1-\beta)^{-1} y^{-2-\beta'} dy$, then $\widetilde \pi$ is a L\'evy measure which satisfies Assumption \ref{asp: branching mechanism}.
	Applying Lemma \ref{lem:CEP} to $\widetilde \pi$, we have that there exists $c>0$ such that
  \[
    | \eta z^{ 1 + \beta } - \eta' z^{ 1 + \beta' } |
    \leq c z^{ 1 + \beta + \delta } + c z^2
    , \quad z \in \mathbb R_+.
  \]
  Dividing both sides by $z^{1+\beta}$ we have
   $
   | \eta - \eta' z^{ \beta' - \beta } |
    \leq cz^{\delta}+cz^{1-\beta}
    ,	z \in \mathbb R_+.
 $
	This implies that $ \eta' z^{\beta' - \beta} \xrightarrow[\mathbb R^+\ni z\to 0]{} \eta >0. $
	So we must have $\beta'= \beta$ and $\eta'= \eta$.
\end{proof}
\begin{lem}
  \label{lem: LlogL criterion}
  If $\psi$ satisfies Assumption  \ref{asp: branching mechanism}, then $\psi$ satisfies the $L \log L$ condition, i.e.,
   $
    \int_{(1,\infty)} y \log y~\pi(dy)
    < \infty.
 $
\end{lem}
\begin{proof}
	Using  Assumption \ref{asp: branching mechanism} and the fact that $y\log y \leq y^{1+\beta+\delta}$ for $y$ large enough, we get
  \[
    \int_{(1,\infty)} y \log y ~\Big| \pi(dy) - \frac { \eta ~dy } { \Gamma ( - 1 - \beta ) y^{ 2 + \beta } } \Big|
    < \infty.
  \]
	Therefore we have
  \[
    \int_{ ( 1, \infty ) } y \log y ~\Big( \pi(dy) - \frac { \eta ~dy } { \Gamma ( - 1 - \beta ) y^{ 2 + \beta } } \Big)
    < \infty.
  \]
  Combining this with
  $
  \int_{ ( 1, \infty ) } \frac { \eta \log y ~dy } { \Gamma ( - 1 - \beta ) y^{ 1 + \beta } }
    < \infty,
 $
  we immediately get the desired result.
\end{proof}

\subsection{Definition of controller}
\label{sec: controller}
Denote by $\mathcal B(\mathbb R^d, \mathbb C)$ the space of all $\mathbb C$-valued Borel functions on $\mathbb R^d$.
Recall that $\mathcal P$ is given in \eqref{eq: polynomial growth function}.
Define $\mathcal P^+:= \mathcal P \cap \mathcal B(\mathbb R^d, \mathbb R_+)$ and $\mathcal P^*:= \{f\in \mathcal B(\mathbb R^d, \mathbb C): |f|\in \mathcal P^+\}$.

In this paper, we say $R$ is a \emph{monotone} operator on $\mathcal P^+$ if $R:\mathcal P^+ \to \mathcal P^+$ satisfies that $Rf\leq Rg$ for all $f\leq g$ in $\mathcal P^+$.
For a function $h: [0,\infty) \to [0,\infty)$, we say $R$ is an \emph{$h$-controller} if $R$ is a monotone operator on $\mathcal P^+$ and that $R(\theta f)\leq h(\theta) Rf$ for all $f\in \mathcal P^+$ and $\theta \in [0,\infty)$.
For subsets $\mathcal D, \mathcal I\subset \mathcal P^*$ and an operator $R$ on $\mathcal P^+$, we say an operator $A$ is \emph{controlled by $R$ from $\mathcal D$ to $\mathcal I$} if $A:\mathcal D \to \mathcal I$ and that $|Af| \leq R|f|$ for all $f\in \mathcal D$;
we say a family of operators $\mathscr O$ is \emph{uniformly controlled by $R$ from $\mathcal D$ to $\mathcal I$} if
each operator $A\in \mathscr O$  is controlled by $R$ from $\mathcal D$ to $\mathcal I$.
For subsets $\mathcal D, \mathcal I\subset \mathcal P^*$ and a function $h:[0,\infty) \to [0,\infty)$, we say an operator $A$ (resp. a family of operators $\mathscr O$) is \emph{$h$-controllable} (resp. \emph{uniformly $h$-controllable}) from $\mathcal D$ to $\mathcal I$ if there exists an $h$-controller $R$ such that $A$ (resp. $\mathscr O$) is controlled (resp. uniformly controlled) by $R$ from $\mathcal D$ to $\mathcal I$.

For two operators $A: \mathcal D_A \subset \mathcal P^*\to \mathcal P^*$ and $B: \mathcal D_B \subset \mathcal P^*\to \mathcal P^*$, define $(A \times B)f (x):= Af(x) \times Bf(x)$ for all $f\in \mathcal D_A \cap \mathcal D_B$ and $x\in \mathbb{R}^d$.
For any $a \in \mathbb R$ and any operator $A :\mathcal D_A \to \mathcal B(\mathbb R^d, \mathbb C\setminus (-\infty, 0])$, define $A^{\times a}f(x):= (Af(x))^a$ for all $f\in \mathcal D_A$ and $x\in \mathbb R^d$.

The following lemma is easy to verify.
\begin{lem}
  \label{lem: property of controllable operators}
  For each $i \in \{0,1\}$, let $\mathscr O_i$ be a family of operators which is  uniformly controlled by an $h_i$-controller $R_i$ from $\mathcal D_i \subset \mathcal P^*$ to $ \mathcal I_i \subset \mathcal P^*$.
  Then the followings hold:
  \begin{enumerate}
  \item
    If $\mathcal I_0 \subset \mathcal D_1$, then $\{A_1A_0: A_i \in \mathscr O_i, i = 0,1\}$ is uniformly controlled by the $(h_1 \circ h_0)$-controller $R_1R_0$ from $\mathcal D_0$ to $\mathcal I_1$.
  \item
    $\{ A_1 \times A_0: A_i \in \mathscr O_i, i = 0,1\}$ is uniformly controlled by the $(h_1\times h_0)$-controller $R_1 \times R_0$ from $\mathcal D_0 \cap \mathcal D_1$ to $\mathcal P^*$.
  \item
    $\{ A_1 + A_0: A_i \in \mathscr O_i, i = 0,1\}$ is uniformly controlled by the $(h_1 \vee h_0)$-controller $R_1 + R_0$ from $\mathcal D_0 \cap \mathcal D_1$ to $\mathcal P^*$.
  \item
    If $\mathcal I_0 \subset \mathcal B(\mathbb R^d, \mathbb C \setminus (\infty, 0])$ and $a>0$, then $\{A^{\times a} : A \in \mathscr O_0\}$ is uniformly controlled by the $(h_0^a)$-controller $R_0^{\times a}$ from $\mathcal D_0$ to $\mathcal P^*$.
  \item
    Suppose that $\mathscr O_0 = \{A_\theta: \theta \in \Theta \}$ where $\Theta$ is an index set.
    Further suppose that $(\Theta, \mathcal J )$ is a measurable space and that $(\theta,x) \mapsto A_\theta f(x)$ is $\mathcal J \otimes \mathcal B(\mathbb R^d)$-measurable for each $f\in \mathcal D$.
    Then the following space of operators
    \[
      \Big\{ f \mapsto \int_{\Theta} A_\theta f~\nu(d\theta) : \nu \text{ is a probability measure on } (\Theta, \mathcal J) \Big\}
    \]
    is uniformly controlled by $R_0$ from $\mathcal D_0$ to $\mathcal P^*$.
  \end{enumerate}
\end{lem}

\subsection{Controllers for the super-OU processes}
\label{sec: h-controller}
Let $X$ be our super-OU process with branching mechanism $\psi$ satisfying
Assumptions \ref{asp: Greys condition} and \ref{asp: branching mechanism}.
In this subsection, we will define several operators and study some of their properties that will be used in this paper.

Define $\psi_0(z) = \psi(z) + \alpha z$ for $z\in \mathbb{R}_+$.
According to Lemma \ref{lem:CEP}, $\psi, \psi_1$ and $\psi_0$ can all be uniquely extended as complex-valued continuous functions on $\mathbb C_+$ which are also holomorphic on $\mathbb C^0_+$.
For all $f\in \mathcal B(\mathbb R^d, \mathbb C_+)$ and $x\in \mathbb R^d$, define $\Psi f (x) = \psi\circ f(x)$, $\Psi_0 f(x)= \psi_0 \circ f(x)$ and $\Psi_1 f(x)= \psi_1 \circ f(x)$.

For all $t\in [0,\infty), x\in \mathbb R^d $ and $f \in \mathcal{P}$, let $ U_tf(x) := \operatorname{Log} \mathbb P_{\delta_x}[e^{i\theta X_t(f)}]|_{\theta = 1} $ be the value of the characteristic exponent of the infinitely divisible random variable $X_t(f)$ (See the paragraph after Lemma \ref{lem: Lip of power function}).
It follows from \eqref{eq: -v has positive real part} that $-U_tf(x)$ takes values in $\mathbb C_+$. Furthermore, we know from Proposition \ref{prop: complex FKPP-equation} that
\begin{align}
  \label{eq:chareq2}
  U_tf(x) - \int_0^t P^\alpha_{t-s} \Psi_0(-U_sf)(x)ds
  = i P^{\alpha}_t f(x)
  , \quad t\in [0,\infty), x\in \mathbb{R}^d, f\in \mathcal P.
\end{align}

For all $t\geq 0$ and $f\in \mathcal P$, we define
\begin{align}
  \label{eq: def of Zf}
  Z_t f
  := \int_0^t P^\alpha_{t-s}\big( \eta (-i P^\alpha_sf)^{1+\beta}\big)ds,
  & \qquad Z'_t f
    := \int_0^t P^\alpha_{t-s}\big( \eta (-U_s f)^{1+\beta}\big)ds,
  \\ Z''_t f
  := \int_0^t P^\alpha_{t-s}\Psi_1(-U_s f)ds,
  & \qquad\  Z'''_t f
    := (Z'_t - Z_t+ Z''_t)f.
\end{align}
Then we have that
\begin{align}
  \label{eq: key equality}
  U_t - i P^\alpha_t
  = Z'_t + Z''_t
  = Z_t + Z'''_t
  , \quad t\geq 0.
\end{align}
For all $\kappa \in \mathbb Z_+$ and $f\in \mathcal P$, define
\begin{align}
  \label{eq:Q}
  Q_\kappa f
  := \sup_{t\geq 0} e^{\kappa b t}|P_t f|,
  \qquad  Q f
  := Q_{\kappa_f}f.
\end{align}
Then according to \cite[Fact 1.2]{MarksMilos2018CLT}, $Q$ is an operator from $\mathcal P$ to $\mathcal P$.

\begin{lem}
  \label{lem: upper bound for usgx}
  Under Assumptions \ref{asp: Greys condition} and \ref{asp: branching mechanism}, the following statements are true:
  \begin{enumerate}
  \item
    $(-U_t)_{0\leq t\leq 1}$ is uniformly $\theta$-controllable from $\mathcal P$ to $\mathcal P^*\cap \mathcal B(\mathbb R^d, \mathbb C_+)$.
  \item
    $(P^\alpha_t)_{0\leq t\leq 1}$ is uniformly $\theta$-controllable on $\mathcal P^*$.
  \item
    $\Psi_0$ is $(\theta^2\vee \theta^{1+\beta})$-controllable from $\mathcal P^* \cap \mathcal B(\mathbb R^d, \mathbb C_+)$ to $\mathcal P^*$.
  \item
    $(U_t- iP_t^{\alpha})_{0\leq t\leq 1}$ is uniformly $(\theta^2\vee \theta^{1+\beta})$-controllable from $\mathcal P$ to $\mathcal P^*$.
  \item
    $(Z'_t-Z_t)_{0\leq t\leq 1}$ is uniformly $(\theta^{2+\beta}\vee \theta^{1+2\beta})$-controllable from $\mathcal P$ to $\mathcal P^*$.
  \item
    For all $\delta > 0$ small enough, we have that $(Z''_t)_{0\leq t\leq 1}$ is uniformly $(\theta^2\vee \theta^{1+\beta+\delta})$-controllable from $\mathcal P$ to $\mathcal P^*$.
  \item
    For all $\delta > 0$ small enough, we have that $(Z'''_t)_{0\leq t\leq 1}$ is uniformly $(\theta^{2+\beta}\vee \theta^{1+\beta+\delta})$-controllable from $\mathcal P$ to $\mathcal P^*$.
  \end{enumerate}
\end{lem}

\begin{proof}
  (1). According to \eqref{eq: -v has positive real part}, $-U_t$ is an operator from $\mathcal P$ to $\mathcal B(\mathbb R^d, \mathbb C_+)$.
  It follows from \eqref{eq: upper bound for vf} that for all $g\in \mathcal P$, $0\leq t\leq 1$ and $x\in \mathbb R^d$, we have $ |U_t g(x)| \leq \sup_{0\leq u\leq 1}P_u^\alpha |g| (x). $
  We claim that $f\mapsto\sup_{0\leq u\leq 1}P^{\alpha}_u f$ is a map from $\mathcal P^+$ to $\mathcal P^+$. In fact, if $f\in \mathcal P^+$, there exists constant $c>0$ such that
  \[
    0
    \leq \sup_{0\leq u\leq 1}P^{\alpha}_u f
    \leq \sup_{0\leq u\leq 1} P_u (e^{\alpha u} e^{-\kappa_f u} e^{\kappa_f u} f )
    \leq c \sup_{0\leq u\leq 1} (e^{\kappa_fu}P_u f) \leq c Qf \in \mathcal P.
  \]
	It is clear that $f\mapsto\sup_{0\leq u\leq 1}P^{\alpha}_u f$ is a $\theta$-controller.

  (2). Similar to the proof of (1).

  (3). By Lemma \ref{lem:CEP}, there exist $C, \delta >0$ satisfying $\beta+\delta< 1$ such that for all $ f \in \mathcal P^* \cap \mathcal B( \mathbb R^d, \mathbb C_+ )$, it holds that $ |\Psi_0 f| \leq \eta |f|^{1+\beta} + |\Psi_1 f| \leq \eta |f|^{1+\beta} + C|f|^2+ C |f|^{1+\beta + \delta}.$
  Note that the operator
  \[
    f \mapsto \eta f^{1+\beta} + Cf^2+ Cf^{1+\beta + \delta}
    , \quad f\in \mathcal P^+,
  \]
  is a $(\theta^2 \vee \theta^{1+\beta})$-controller.

  (4). From (1)--(3) above and Lemma \ref{lem: property of controllable operators}.(1), we know that the operators
  \[
    f \mapsto P^\alpha_{t-s}\Psi_0(-U_sf)
    , \quad 0\leq s\leq t\leq 1,
  \]
  are uniformly $(\theta^2\vee \theta^{1+\beta})$-controllable.
  Combining this with \eqref{eq:chareq2} and
  Lemma \ref{lem: property of controllable operators}.(5), we get the desired result.

  (5). Notice that from Lemma \ref{lem: Lip of power function},
  \[
    |(-U_t f)^{1+\beta} - (-iP^\alpha_t f)^{1+\beta} |
    \leq  (1+\beta) |U_t f-iP^\alpha_t f|(|U_t f|^{\beta}+|i P^\alpha_t f|^{\beta}).
  \]
  Now using (1), (2) and (4) above, and Lemma \ref{lem: property of controllable operators}, we get that the operators
  \[
    f \mapsto (-U_t f)^{1+\beta} - (-iP^\alpha_t f)^{1+\beta},\quad 0\leq t\leq 1,
  \]
  are uniformly $(\theta^{2+\beta}\vee \theta^{1+2\beta})$-controllable.
  Combining with Lemma \ref{lem: property of controllable operators}, and
  \[
    (Z'_t - Z_t)f
    = \int_0^t P^\alpha_{t-s}\Big( \eta ((-U_s f)^{1+\beta} - (-iT_s^\alpha f)^{1+\beta} )\Big)ds
    , \quad 0\leq t\leq 1, f\in \mathcal P,
  \]
  we get the desired result.

  (6). By  Lemma \ref{lem:CEP}, for all $\delta > 0$ small enough, there exists  $C>0$ such that
  \[
    |\Psi_1(f)|
    \le C(|f|^2+|f|^{1+\beta+ \delta})
    , \quad f\in \mathcal P^*\cap\mathcal B(\mathbb R^d, \mathbb C_+).
  \]
  Note that, for all $\delta, C>0$,
  \[
    f \mapsto C(f^2+f^{1+\beta+\delta})
    , \quad f\in \mathcal P^+
  \]
  is a $(\theta^2 \vee \theta^{1+\beta+\delta})$-controller.
  Therefore, for all $\delta > 0$ small enough, we have that $\Psi_1$ is a $(\theta^2 \vee \theta^{1+\beta+\delta})$-controllable operator from $\mathcal P^*\cap\mathcal B(\mathbb R^d, \mathbb C_+)$ to $\mathcal P^*$.
  Combining  this with  (1)--(2) above,
  and Lemma \ref{lem: property of controllable operators}, we get that, for all $\delta > 0$ small enough, the operators
  \[
    f
    \mapsto Z_t'' f
    = \int_0^t P_{t-s}^\alpha \Psi_1(-U_sf)ds
    , \quad 0\leq t\leq 1,
  \]
  are uniformly $(\theta^2 \vee \theta^{1+\beta+\delta})$-controllable from $\mathcal P$ to $\mathcal P^*$.

  (7). Since $Z'''_t = (Z'_t-Z_t)+Z''_t$, the desired result follows from (5)--(6) above and Lemma \ref{lem: property of controllable operators}.(3).
\end{proof}

\subsection{Stable distributions}
\label{sec: stable distributions}
Recall that the operators $(T_t)_{t\geq 0}$ are defined by \eqref{eq:I:R:1}, and the functionals $(m_{t})_{0\leq t< \infty}$ and $m$ are given by \eqref{eq:I:R:2} and \eqref{eq:I:R:3} respectively.
\begin{lem}
  \label{lem:m}
$(T_t)_{t\geq 0}$,  $(m_{t})_{0\leq t< \infty}$ and  $m$ are well defined.
\end{lem}
\begin{proof}
  \emph{Step 1.} We will show that for each $f \in \mathcal P$, there exists $h \in \mathcal P$ such that $ |T_tf| \leq  e^{- \delta t} h$ for each $t\geq 0$, where
  \begin{align}
    \label{eq:m:1}
    \delta
    := \inf \big\{ |\tilde \beta \alpha - |p|b| : p \in \mathbb Z_+^d, \langle f, \phi_p\rangle_\varphi \neq 0 \big\}
    \geq 0.
  \end{align}
  From this upper bound, it can be verified that $(T_t)_{t\geq 0}$ and $(m_{t})_{0 \leq t < \infty}$ are well defined.
  In fact, we can write $f = f_0 + f_1$ with $f_0\in \mathcal C_s \oplus \mathcal C_c$ and $f_1 \in \mathcal C_l$.
  According to \cite[Lemma 2.7]{MarksMilos2018CLT}, there exists $h_0 \in \mathcal P$ such that for each $t\geq 0$,
  \[
    |T_t f_0|
    = \Big| \sum_{p \in  \mathbb Z_+^d: \tilde \beta \alpha \leq |p|b } e^{- ( |p| b - \tilde \beta \alpha ) t} \langle f, \phi_p \rangle_\varphi \phi_p \Big|
    = e^{\tilde \beta \alpha t} | P_t f_0 |
    \leq e^{- ( \kappa_{(f_0)} b - \tilde \beta \alpha) t} h_0
    \leq e^{- \delta t} h_0.
  \]
  On the other hand
  \[
    |T_t f_1|
    \leq e^{- \delta t}\sum_{p \in \mathbb Z_+^d : \tilde \beta \alpha > |p|b} |\langle f, \phi_p \rangle_\varphi \phi_p|
    =: e^{- \delta t} h_1,
    \quad t\geq 0.
  \]
  So the desired result in this step follows with $h := h_0 + h_1$.

  \emph{Step 2.} We will show that if $f \in \mathcal C_s \oplus \mathcal C_l$, then $m[f]$ is well defined.
  In fact, let $\delta$ be given by \eqref{eq:m:1}, then in this case $\delta > 0$.
  Now, according to Step 1 there exists $h \in \mathcal P$ such that $|T_tf| \leq e^{- \delta t} h$ for each $t\geq 0$.
  This exponential decay implies the desired result in this step.

  \emph{Step 3.} We will show that if $f\in \mathcal P \setminus (\mathcal C_s \oplus \mathcal C_l)$, then $m[f]$ is also well defined.
  In fact, $f$ can be decomposed as $f = f_c + f_{sl}$ where $f \in \mathcal C_c\setminus \{0\}$ and $f_{sl}\in \mathcal C_s \oplus \mathcal C_l$.
  Note that $T_t f_c = f_c$ for each $t\geq 0$.
  Also note that in Step 2, we already have shown that there exist $\delta > 0$ and $h \in \mathcal P^+$ such that for each $t\geq 0$, we have $|T_t f_{sl}| \leq e^{- \delta t}h$.
  Therefore, using Lemma \ref{lem: Lip of power function} we have
  \begin{align}
    &|(-iT_t f)^{1+\beta} - (-i f_c)^{1+\beta}|
      \leq (1+\beta) ( |T_tf|^\beta + |f_c|^\beta) |T_tf_{sl}|
    \\&\leq (1+\beta) ( |f_c + T_t f_{sl}|^\beta + |f_c|^\beta) e^{- \delta t} h
    \leq (1+\beta) ( (|f_c| + |h|)^\beta + |f_c|^\beta) e^{- \delta t} h
    =: e^{- \delta t} g,
  \end{align}
  where $g\in \mathcal P^+$.
  Therefore
  \begin{align}
    \label{eq:P:S:1}
    &\Big| \frac{1}{t} m_t[f] - \langle (-if_c)^{1+\beta}, \varphi\rangle\Big|
      = \Big| \frac{1}{t} \cdot t \int_0^1  \big\langle (-iT_{rt}f)^{1+\beta} - (-if_c)^{1+\beta}, \varphi\big \rangle ~dr \Big|\\
    &\leq \int_0^1  \big\langle |(-iT_{rt}f)^{1+\beta} - (-if_c)^{1+\beta}|, \varphi\big \rangle ~dr 
      \leq \langle g,\varphi\rangle\int_0^1 e^{-\delta rt} ~dr
      \xrightarrow[t\to \infty]{} 0.
      \qedhere
  \end{align}
\end{proof}

\begin{prop}
  \label{prop:PL:S}
  For each $f \in \mathcal P\setminus \{0\}$, there exists a non-degenerate $(1+\beta)$-stable random variable $\zeta^f$ such that $ E[e^{i\theta\zeta^f}] = e^{m[\theta f]}$ for all $\theta \in \mathbb R$.
\end{prop}

The proof of the above proposition relies on the following lemma:
\begin{lem}
  \label{lem: charactreisticfunction}
  Let $q$ be a measure on $\mathbb R^d\setminus\{0\}$ with
  $\int_{\mathbb R^d\setminus\{0\}} |x|^{1+\beta} q(dx) \in (0,\infty)$.
  Then
  \[
    \theta
    \mapsto  \exp\Big\{\int_{\mathbb R^d\setminus\{0\}} (i\theta \cdot x)^{1+\beta} q(dx)\Big\},
    \quad \theta \in \mathbb R^d,
  \]
  is the characteristic function of an $\mathbb R^d$-valued $(1+\beta)$-stable random variable.
\end{lem}
\begin{proof}
  It follows from disintegration that there exist a measure $\lambda$ on $S:= \{\xi\in \mathbb R^d:|\xi| = 1\}$ and a kernel $k(\xi,dt)$ from $S$ to $\mathbb R_+$ such that
  \[
    \int_{\mathbb R^d\setminus \{0\}} f(x)q(dx)
    = \int_S \lambda(d\xi) \int_{\mathbb R_+} f(t \xi)k(\xi,dt)
    , \quad f\in \mathcal B(\mathbb R^d\setminus \{0\}, \mathbb R_+).
  \]
  We define another measure $\lambda_0$ on $S$ by
  \[
    \lambda_0(d\xi)
    := \frac1{\Gamma(-1-\beta)}\int_0^\infty t^{1+\beta}k(\xi,dt) \lambda (d\xi),
  \]
  where $\Gamma$ is the Gamma function.
  Then $\lambda_0$ is a non-zero finite measure, since
  \begin{align}
    \lambda_0(S)
      = &\frac{1}{\Gamma(-1-\beta)} \int_S \lambda (d\xi) \int_0^\infty |t\xi|^{1+\beta}k(\xi,dt) \\
   =  & \frac{1}{\Gamma(-1-\beta)} \int_{\mathbb R^d\setminus\{0\}} |x|^{1+\beta} q(dx) \in (0,\infty).
  \end{align}
  Define a measure $\nu$ on $\mathbb R^d\setminus\{0\}$ by
  \[
    \int_{\mathbb R^d\setminus\{0\}}f(x)\nu(dx)
    = \int_{S} \lambda_0(d\xi) \int_0^\infty f(r\xi) \frac{dr}{r^{2+\beta}}
    ,\quad f\in \mathcal B(\mathbb R^d\setminus \{0\}, \mathbb R_+).
  \]
  Then, according to \cite[Remark 14.4]{Sato2013Levy}, $\nu$ is the L\'evy measure of a $(1+\beta)$-stable distribution on $\mathbb R^d$, say $\mu$, whose characteristic function is
  \[
    \hat \mu(\theta)
    = \exp \Big \{ \int_{\mathbb R^d\setminus\{0\}} (e^{-i\theta \cdot y}-1+i\theta \cdot y) \nu(dy) \Big \}
    , \quad \theta \in \mathbb R.
  \]
  Finally, according to \eqref{eq: stable branching on C+}, we have
  \begin{align}
    & \int_{\mathbb R^d\setminus\{0\}} (e^{-i\theta \cdot y}-1+i\theta \cdot y) \nu(dy)
      = \int_S \lambda_0(d\xi) \int_0^\infty (e^{-ir\theta \cdot \xi}-1+ir\theta \cdot \xi) \frac{dr}{r^{2+\beta}} \\
    & = \int_S \lambda (d\xi) \int_0^\infty (e^{-ir\theta \cdot \xi}-1+ir\theta \cdot \xi) \frac{dr}{\Gamma(-1-\beta)r^{2+\beta}}\int_0^\infty t^{1+\beta} k(\xi,dt) \\
    & = \int_S \lambda (d\xi) \int_0^\infty (i\theta\cdot \xi)^{1+\beta} t^{1+\beta} k(\xi,dt)
      = \int_S \lambda(d\xi) \int_0^\infty (i\theta \cdot t\xi)^{1+\beta} k(\xi,dt) \\
    & = \int_{\mathbb R^d} (i\theta \cdot x)^{1+\beta} q(dx).
      \qedhere
  \end{align}
\end{proof}

\begin{proof}[Proof of Proposition \ref{prop:PL:S}]
	Suppose that $f\in \mathcal C_s \oplus \mathcal C_l$.
Note that $m[\theta f]$ can be written as
  \begin{align}
    \label{eq:PL:S:1}
    m[\theta f]
    = \eta \int_0^{\infty}~ds\int_{\mathbb R^d} (-i\theta T_s f(x))^{1+\beta} \varphi(x)~dx,
    \quad \theta \in \mathbb R.
  \end{align}
	Therefore, according to Lemma \ref{lem: charactreisticfunction}, in order to show that $\zeta^f$ is a $(1+\beta)$-stable random variable we only need to show that
\begin{equation}\label{e:new}
    \int_0^{\infty}~ds\int_{\mathbb R^d} | T_{s} f(x)|^{1+\beta} \varphi(x)~dx
    \in (0, \infty).
\end{equation}
  According to the Step 1 in the proof of Lemma \ref{lem:m}, we know that there exist $\delta> 0$ and $h \in \mathcal P$ such that $|T_sf| \leq e^{- \delta s} h$ for each $s\geq 0$.
The claim \eqref{e:new} then follows.

  If $f \in \mathcal P \setminus (\mathcal C_s \oplus \mathcal C_l)$, then $f$ can be written by $f = f_c +(f - f_c)$ where $f_c \in \mathcal C_c\setminus\{0\}$ and $f - f_c \in \mathcal C_s \oplus \mathcal C_l$.
  In this case, according to \eqref{eq:P:S:1}, $m[\theta f]$ has an integral representation:
  \begin{align}
    \label{eq:PL:S:2}
    m[\theta f]
    = \int_{\mathbb R^d} (-i\theta f_c(x))^{1+\beta} \varphi(x) ~dx,
    \quad \theta \in \mathbb R.
  \end{align}
  Finally, according to Lemma \ref{lem: charactreisticfunction} and the fact that
 $
    \int_{\mathbb R^d} | f_c(x)|^{1+\beta} \varphi(x)~dx
    \in (0, \infty),
  $
  We have that $\zeta^f$ is a non-degenerate $(1+\beta)$-stable random variable.
\end{proof}

\subsection{A refined estimate for the OU semigroup}\label{sc:refined}
It turns out that our proof of the CLT relies on the following refined estimate for the OU semigroup.

\begin{lem}
  \label{lem:P:R}
  Suppose that $g \in \mathcal P$, then there exists $h \in \mathcal P^+$ such that for all $ f \in \mathcal P_g := \{\theta T_n g: n \in \mathbb Z_+, \theta \in [-1,1]\} $ and $t\geq 0$, we have $ | P_t (Z_1 f - \langle Z_1 f, \varphi \rangle )| \leq e^{-bt} h$.
\end{lem}
\begin{proof}
Fix $g \in \mathcal P$.
We write  $g = g_0 + g_1$ with $g_0 \in \mathcal C_s \oplus \mathcal C_c$ and $g_1 \in \mathcal C_l$,  and $q_f:=Z_1f - \langle Z_1f, \varphi \rangle\in \mathcal P^*$ for each $f\in \mathcal P$.
  We need to prove that there exists $h \in \mathcal P^+$ such that for each $f\in \mathcal P_g$, $|P_tq_f| \leq e^{-bt} h$.

  \emph{Step 1.} We claim  that we only need to prove the result for all
  $f \in \widetilde{\mathcal P}_g:= \{T_{n+1} g : n \in \mathbb Z_+\}$.
  In fact, both $\operatorname{Re} q_g$ and $\operatorname{Im} q_g$ are functions in $\mathcal P$ of order $\geq 1$.
  The result is valid for $f = T_0 g = g$ according to \cite[Fact 1.2]{MarksMilos2018CLT}.
  Also, note that if the result is valid for some $f \in \mathcal P$, it is also valid for any $\theta f$ with $\theta \in [-1,1]$.

  \emph{Step 2.} We show that $\{T_s g: s> 0\} \subset C_\infty (\mathbb R^d) \cap \mathcal P$.
  In fact, for each $s > 0$,
  \[
    T_s g
    = T_s (g_0 + g_1)
    = e^{\alpha \tilde \beta s}P_s g_0 + \sum_{p \in \mathbb Z_+^d: \alpha \tilde \beta > |p|b} 
    \langle g_1, \phi_p \rangle_\varphi e^{-(\alpha \tilde \beta - |p|b)s} \phi_p.
  \]
  Notice that the second term is in $C_\infty(\mathbb R^d)\cap \mathcal P$ since it is a finite sum of polynomials, and the first term is also in $C_\infty (\mathbb R^d) \cap \mathcal P$ according to \cite[Fact 1.1]{MarksMilos2018CLT}.

  \emph{Step 3.} We show that there exists $h_3 \in \mathcal P^+$ such that for all $j \in \{1,\dots, d\}$ and $f \in \widetilde {\mathcal P}_g$, it holds that $|\partial_j f| \leq h_3$.
  In fact, it is known that  (see \cite{MetafunePallaraPriola2002Spectrum} for example)
  \begin{align}
    \label{eq:P:R:3:-1}
    P_t f(x)
    = \int_{\mathbb R^d} f\big(x e^{-bt} + y \sqrt{1-e^{-2bt}}\big) \varphi(y)~dy,
    \quad t\geq 0, x\in \mathbb R^d, f\in \mathcal P.
  \end{align}
  For $f \in C_\infty(\mathbb R^d)\cap \mathcal P$ it can be verified from above that
  \begin{align}
    \label{eq:P:R:3:1}
    \partial_j P_t f
    = e^{-bt} P_t \partial_j f,
    \quad t \geq 0, j \in \{1,\dots, d\}.
  \end{align}
  Thanks to Step 2, $T_1 g_0 \in C_\infty(\mathbb R^d)\cap \mathcal P$.
  According to \cite[Fact 1.3]{MarksMilos2018CLT} and the fact that $\alpha \tilde \beta \leq \kappa _{g_0} b$, we have for each $j \in \{1,\dots, d\}$,
  \[
    \kappa_{(\partial_j T_1 g_0)}
    \geq \kappa_{(T_1 g_0)} - 1
    = \kappa_{g_0} - 1
    \geq \frac{\alpha \tilde \beta}{b} - 1.
  \]
  Therefore, there exists  $h'_3\in \mathcal P^+$ such that for all $n \in \mathbb Z_+$ and $j\in \{1,\dots,d\}$,
  \begin{align}
    & | \partial_j T_{n+1}g_0 |
      = | \partial_j e^{\alpha \tilde \beta n}P_n T_1g_0 |
      = e^{\alpha \tilde \beta n-bn} |P_n \partial_j T_1 g_0| \\
    & \leq e^{\alpha \tilde \beta n-bn} \exp\{-\kappa_{(\partial_j T_1 g_0)}bn\}Q \partial_j T_1g_0
      \leq Q\partial_j T_1g_0
      \leq h'_3.
  \end{align}
  On the other hand, there exists $h_3''\in \mathcal P^+$ such that for all $n \in \mathbb Z_+$ and $j\in \{1,\dots,d\}$,
  \begin{align}
    & |\partial_j T_{n+1}g_1 |
      = \Big| \sum_{p\in \mathbb Z_+^d: \alpha \tilde \beta > |p|b} e^{- (\alpha \tilde \beta - |p|b)(n+1)} \langle g_1, \phi_p \rangle_\varphi \partial_j \phi_p \Big| \\
    & \leq \sum_{p\in \mathbb Z_+^d: \alpha \tilde \beta > |p|b} |\langle g_1, \phi_p \rangle_\varphi \partial_j \phi_p |
      \leq h_3''.
  \end{align}
  Then the desired result in this step follows.

  \emph{Step 4.} We show that there exists $h_{4} \in \mathcal P^+$ such that for all $j \in \{1,\dots, d\}, u \in [0, 1]$ and $f \in \widetilde {\mathcal P}_g$, it holds that $ | \partial_j  P_{1-u}^\alpha (- i P_u^\alpha f)^{1+\beta} | \leq h_4$.
  In fact, thanks to Step 2 and \eqref{eq:P:R:3:1}, for all $j \in \{1,\dots, d\}, u \in [0, 1]$ and $f \in \widetilde{\mathcal P}_g$, we have
  \begin{align}
    & \partial_j  P_{1-u}^\alpha (- i P_u^\alpha f)^{1+\beta}
      = e^{-(1-u)b} P_{1-u}^\alpha \partial_j (- i P_u^\alpha f)^{1+\beta}
    \\ & = (1+\beta) e^{-(1-u)b} P_{1-u}^\alpha [ (- i P_u^\alpha f)^\beta \partial_j (- i P_u^\alpha f) ]
    \\ & = -i(1+\beta) e^{-(1-u)b} P_{1-u}^\alpha[ (- i P_u^\alpha f)^\beta e^{-ub} P_u^\alpha \partial_j f]
    \\ & = -i(1+\beta) e^{-b} e^{(1-u)\alpha} e^{u\alpha (1+\beta)} P_{1-u} [ (- i P_u f)^\beta P_u \partial_j f ].
  \end{align}
  Recall from Step 1 in the proof of Lemma \ref{lem:m} there exists $h'_4\in \mathcal P^+$ such that for each $f \in \{T_sg:s\geq 0\}$ it holds that $|f| \leq h'_4$.
  Therefore, using Step 3, we have for all $j \in \{1,\dots, d\}, u \in [0, 1]$ and $f \in \widetilde {\mathcal P}_g$,
  \begin{align}
    & |\partial_j  P_{1-u}^\alpha (- i P_u^\alpha f)^{1+\beta}|
    \leq (1+\beta) e^{\alpha (1+\beta)} P_{1-u} [  (P_u |f|)^\beta P_u |\partial_j f| ]
    \\ & \leq (1+\beta) e^{\alpha (1+\beta)} P_{1-u} [  (P_u h'_4)^\beta P_u h_3 ]
    \leq (1+\beta) e^{\alpha (1+\beta)} Q_0 [  (Q_0 h_4')^\beta Q_0 h_3 ],
  \end{align}
  where $Q_0$ is defined by \eqref{eq:Q}.
  This implies the desired result in this step.

  \emph{Step 5.}
We show that there exists $h_5 \in \mathcal P^+$ such that for each $f \in \widetilde {\mathcal P}_g$, we have $ |\nabla (Z_1f)| \leq h_5$.
  In fact, according to Step 4, for all $j \in \{1,\dots, d\}$, $f \in \widetilde{\mathcal P}_g$ and compact $A \subset \mathbb R^d$, we have
  \[
    \int_0^1 \sup_{x\in A} | (\partial_j  P_{1-u}^\alpha (-i P_u^\alpha f)^{1+\beta}) (x) |~du
    \leq \sup_{x\in A} h_4(x) < \infty.
  \]
  Using this and \cite[Theorem A.5.2]{Durrett2010Probability}, for all $j \in \{1,\dots, d\}$, $f\in \widetilde {\mathcal P}_g$ and $x\in \mathbb R^d$, it holds that
  \[
    | \partial_j Z_1 f(x)|
    = \Big| \int_0^1  ( \partial_jP_{1-u}^\alpha (-iP_u^\alpha f)^{1+\beta} ) (x) ~du  \Big|
    \leq h_4(x).
  \]
  Now, the desired result for this step is valid.

  \emph{Step 6.}
  Let $h_5$ be the function in Step 5.
  There are $c_0, n_0> 0$ such that for all $x\in \mathbb R^d$, $h_5(x) \leq c_0(1+|x|)^{n_0}$.
  Note that for all $x, y \in \mathbb R^d$, $1+|x|+|y|\leq (1+|x|) (1+|y|)$; and that for all $\theta \in [0,1]$, $|\sqrt {1 - \theta} - 1| \leq \theta$.
  Write $D_{x,y} = \{ax+by: a, b \in [0,1]\}$ fo $x, y \in \mathbb R^d$.
  Using \eqref{eq:P:R:3:-1} and Step 5, there exists  $h_6 \in \mathcal P^+$ such that for all $t \geq 0$, $f \in \widetilde {\mathcal P}_g$ and $x \in \mathbb R^d$,
  \begin{align}
    & |P_t q_f(x)|
      = \Big| \int_{\mathbb R^d} ( (Z_1f)(x e^{-bt}+ y \sqrt{1 - e^{-2bt}}) - Z_1 f (y) ) \varphi(y) ~dy \Big| \\
    & \leq \int_{\mathbb R^d} \Big(\sup_{z\in D_{x,y}} |\nabla  (Z_1f) (z) |\Big) | x e^{-bt} + y \sqrt{1 - e^{-2bt}} - y | \varphi(y) ~dy \\
    & \leq e^{-bt} \int_{\mathbb R^d} c_0(1+|x|+|y|)^{n_0} (|x|+|y|) \varphi(y) ~dy \\
    & \leq c_0 e^{-bt}(1+|x|)^{n_0}\Big(|x|\int_{\mathbb R^d} (1+|y|)^{n_0}\varphi(y) ~dy + \int_{\mathbb R^d} (1+ |y|)^{n_0} |y| \varphi(y) ~dy \Big) \\
    & \leq e^{-bt} h_6(x).
      \qedhere
  \end{align}
\end{proof}

\subsection{Small value probability}
\label{sec: Small value probability}
In this subsection, we digress briefly from our super-OU process and consider a (supercritical) \emph{continuous-state branching process (CSBP)} $\{(Y_t)_{t\geq 0}; \mathbf P_x\}$ with branching mechanism $\psi$ given by \eqref{eq: honogeneou branching mechanism}.
Such a process $\{(Y_t)_{t\geq 0}; \mathbf P_x\}$ is defined as an $\mathbb R^+$-valued Hunt process satisfying
\[
  \mathbf P_x[e^{-\lambda Y_t}] = e^{- x v_t(\lambda)},
  \quad x\in \mathbb R^+, t\geq 0, \lambda \in \mathbb R^+,
\]
where for each $\lambda\geq 0$, $t\mapsto v_t(\lambda)$ is the unique positive solution to the equation
\begin{align}
  \label{eq: fkpp equation for CSBP}
  v_t(\lambda) - \int_0^t \psi(v_s(\lambda))~ds = \lambda,
  \quad t\geq 0.
\end{align}
It can be verified that for each $\mu \in \mathcal M(\mathbb R^d)$ with $x = \| \mu \|$, we have $ \{(\|X_t\|)_{t\geq 0}; \mathbb P_\mu\} \overset{\text{law}}{=} \{(Y_t)_{t\geq 0}; \mathbf P_x\}$.

Our goal in this subsection is to determine how fast the probability $\mathbf P_x(0<e^{-\alpha t}Y_t \leq k_t)$ converges to $0$ when $t\mapsto k_t$ is a strictly positive function on $[0,\infty)$ such that $k_t \to 0$ and $k_t e^{\alpha t} \to \infty$ as $t\to \infty$.
Suppose that Grey's condition is satisfied i.e., there exists $z' > 0$ such that $\psi(z) > 0$ for all $z>z'$, and that $\int_{z'}^\infty \psi(z)^{-1}dz < \infty$.
Also suppose that the $L \log L$ condition is satisfied i.e.,
$
  \int_1^\infty y \log y~\pi(dr)
 < \infty.
$
We write $W_t = e^{-\alpha t}Y_t$ for each $t\geq 0$.

\begin{prop}
  \label{lem: control of XT}
  Suppose that $t\mapsto k_t$ is a strictly positive function on $[0,\infty)$ such that $k_t \to 0$ and $k_t e^{\alpha t} \to \infty$ as $t\to \infty$.
  Then, for each $x\geq 0$, there exist $C,\delta>0$ such that
  \[
    \mathbf P_x(0<W_t\leq k_t)
    \leq C(k_t^\delta + e^{-\delta t}), \quad t\geq 0.
  \]
\end{prop}

\begin{proof}
  \emph{Step 1.}
  We recall some known facts about the CSBP $(Y_t)$.
  For each $\lambda \geq 0$, we denote by $t\mapsto v_t(\lambda)$ the unique positive solution of \eqref{eq: fkpp equation for CSBP}.
  Letting $\lambda \to \infty$ in \eqref{eq: fkpp equation for CSBP}, we have by monotonicity that $\bar v_t:= \lim_{\lambda \to \infty}v_t(\lambda)$ exists in $(0,\infty]$ for all $t\geq 0$, and that
  \begin{align}
    \label{eq: svp1}
    \mathbf P_x(Y_t = 0)=e^{-x\bar v_t}, \quad t\geq 0, x\ge 0.
  \end{align}
  It is known, see \cite[Theorems 3.5--3.8]{Li2011Measure-valued} for example, that under Grey's condition $\bar v:= \lim_{t\to \infty} \bar v_t \in [0,\infty)$ exists and is the largest root of $\psi$ on $[0,\infty)$.
  Letting $t \to \infty$ in \eqref{eq: svp1}, we have by monotonicity that
  \[
    \mathbf P_x(\exists t \geq 0, Y_t = 0)
    = e^{-x\bar v}, \quad x\geq 0.
  \]
 Note the derivative of $\psi$, i.e.,
  \[
    \psi'(z)
    = -\alpha + 2\rho z + \int_{(0,\infty)}(1-e^{-zy})y\pi(dy),\quad z\geq 0,
  \]
  is non-decreasing.
  This says that $\psi$ is a convex function.
  Also notice that $\psi'(0+)=-\alpha <0$ and that there exists $z>0$ such that $\psi(z)>0$.
  Therefore we have (i) $\bar v > 0$; (ii) $\psi(z) < 0$ on $z\in (0,\bar v)$; and (iii) $\psi(z) > 0 $ on $z\in (\bar v, \infty)$.
  It is also known, see \cite[Proposition 3.3]{Li2011Measure-valued} for example, that (i) if $\lambda \in (0,\bar v)$, then $0<\lambda \leq v_t(\lambda)<\bar v $; (ii) if $\lambda \in (\bar v, \infty)$, then $\bar v < v_t(\lambda)\leq \lambda< \infty$; and (iii) for each $\lambda \in (0,\infty)\setminus \{\bar v\}$ and $t\geq 0$, we always have
    \(
      \int_{v_t(\lambda)}^\lambda  \psi(z)^{-1}dz = t.
    \)
  Taking $\lambda \to \infty$ and using the monotone convergence theorem, we have that
  \begin{align}
    \label{eq:svp2}
    \int_{\bar v_t}^\infty \frac{dz}{\psi(z)} = t, \quad t\geq 0.
  \end{align}

  \emph{Step 2.} We will show that, for each $x \geq 0$ there exists a constant $c_1>0$ such that
  \[
    \mathbf P_{x}(0< W_t\leq k_t)
    \leq c_1\big(|\bar v- v_t(k_t^{-1}e^{-\alpha t})|+|\bar v_t - \bar v|\big),
    \quad t\geq 0.
  \]
  In fact, for all $x\geq 0$ and $t\geq 0$, we have
  \begin{align}
    & \mathbf P_{x}(0<W_t \leq k_t)
      = \mathbf P_{x}( e^{-k_t^{-1}W_t}\geq e^{-1},W_t > 0) \\
    & \leq e \mathbf P_{x}[e^{-k_t^{-1} W_t};W_t > 0]
      =  e\big(\mathbf P_x[e^{-k_t^{-1} W_t}]-\mathbf P_x(W_t = 0)\big) \\
    & = e\big(e^{-xv_t(k_t^{-1} e^{-\alpha t})}-e^{-x\bar v_t}\big)
      \leq ex \big(|\bar v-v_t(k_t^{-1} e^{-\alpha t})|+ |\bar v_t- \bar v|\big),
  \end{align}
  as desired in this step.

  \emph{Step 3.} We will show that there exist $c_2, \delta_1, t_0 > 0$ such that
  \[
    |\bar v_t-\bar v|
    \leq c_2e^{-\delta_1 t}
    , \quad t\geq t_0.
  \]
  In fact, since $\psi$ is a convex function, we must have $\tau:=\psi'(\bar v)>0$ and that  $\psi(z) \geq (z-\bar v)\tau$ for each $z\geq \bar v$.
  According to Grey's condition, we can find $z_0 >\bar v $ such that $t_0 := \int^\infty_{z_0}\psi(z)^{-1}dz<\infty$.
  For each $t > t_0$, according to \eqref{eq:svp2}, we have
  \begin{align}
    & t - t_0 =
      \int^\infty_{\bar v_t} \frac{dz}{\psi(z)} - \int_{z_0}^\infty \frac{dz}{\psi(z)}
      = \int_{\bar v_t}^{z_0} \frac{dz}{\psi(z)} \\
    & \leq \int_{\bar v_t}^{z_0} \frac{dz}{(z-\bar v)\tau}
      = \frac{1}{\tau} \big(\log (z_0-\bar v) - \log(\bar v_t-\bar v)\big).
  \end{align}
  Rearranging, we get $ \bar v_t - \bar v \leq (z_0 - \bar v)e^{-\tau(t-t_0)}, $ for all $t\geq t_0$.
  This implies the desired result in this step.

  \emph{Step 4.}
  We will show that there exist $c_3, \delta_2, t_1>0$ such that
  \[
    |\bar v - v_t(k_t^{-1} e^{-\alpha t})|\leq
    c_3k_t^{\delta_2}, \quad t\geq t_1.
  \]
  Define $\rho_t := 1+(\log k_t)/(t\alpha)$ for all $t\geq 0$.
  By the fact that $k_t^{-1}e^{-\alpha t} = e^{-\alpha \rho_t t}$ for  all $t\geq 0$ and the condition that $k_t e^{\alpha t} \xrightarrow[t\to \infty]{} \infty$, we have $\rho_t t \xrightarrow[t\to \infty]{} \infty $.
  Since the $L\log L$ condition is satisfied, we have (see \cite{LiuRenSong2009Llog} for example), $W_t \xrightarrow[t\to \infty]{a.s.} W_\infty$, where the martingale limit $W_\infty$ is a non-degenerate positive random variable.
  This implies that
  \[
    v_t(e^{-\alpha t})
    = -\log \mathbf P_1[e^{-W_t}]\xrightarrow[t\to \infty]{} - \log \mathbf P_{1}[e^{-W_\infty}]
    =: z^* \in (0,\infty).
  \]
  The $L \log L$ condition also guarantees that (see again \cite{LiuRenSong2009Llog} for example) $\{W_\infty = 0\} = \{\exists t \geq 0, X_t= 0\}$  a.s. in $\mathbf P_1$. This and the non-degeneracy of $W_\infty$ imply that
  \[
    z^*
    = -\log \mathbf P_1[e^{-W_\infty}]
    < -\log \mathbf P_1(W_\infty = 0) = \bar v.
  \]

  Fix an arbitrary $\epsilon \in (0,\tau)$.
  According to the fact that $\tau=\psi'(\bar v)>0$, there exists $z_0 \in (0,\bar v)$ such that for all $z\in (z_0, \bar v)$, we have $-\psi(z)\geq (\bar v - z)(\tau- \epsilon)$.
  Fix this $z_0$.
  For $t$ large enough, we have $0<k_t^{-1}e^{-\alpha t} < v_t(k_t^{-1}e^{-\alpha t})< \bar v$.
  Then we have for $t>0$ large enough,
  \begin{align}
    t
    & =\int^{v_t(k_t^{-1} e^{-\alpha t})}_{k_t^{-1} e^{-\alpha t}}\frac{dz}{-\psi(z)}
      = \Big(\int^{v_{\rho_t t}(e^{-\alpha \rho_t t})}_{e^{-\alpha \rho_t t}}  + \int^{z_0}_{v_{\rho_t t}(e^{-\alpha \rho_t t})} +\int^{v_t(k_t^{-1}e^{-\alpha  t})}_{z_0}\Big)\frac{dz}{-\psi(z)} \\
    & = \rho_t t + O(1) +\int^{v_t(k_t^{-1}e^{-\alpha t})}_{z_0} \frac{dz}{-\psi(z)},
  \end{align}
  where we used the fact that
  \[
    \int_{v_{\rho_t t}(e^{-\alpha \rho_tt})}^{z_0} \frac{dz}{-\psi(z)}
    \xrightarrow[t\to \infty] {} \int_{z^*}^{z_0} \frac{dz}{-\psi(z)}.
  \]
  Now we have, for $t$ large enough,
  \begin{align}
    & t
      \leq  \rho_t t + O(1) + \int_{z_0}^{v_t(k_t^{-1}e^{-\alpha t})} \frac{dz}{(\bar v-z)(\tau - \epsilon)} \\
    & =  \rho_t t +O(1)- \frac{1}{\tau-\epsilon}\Big( \log \big(\bar v-v_t(e^{-\alpha \rho_t t})\big) - \log(\bar v-z_0)\Big).
  \end{align}
  Rearranging, we get, for $t$ large enough,
  \[
    e^{-t(\tau - \epsilon)}
    \geq e^{-\rho_t t(\tau - \epsilon)+O(1)}(\bar v - v_t(e^{-\alpha \rho_t t})).
  \]
  Therefore, there exist $c_3>0$ and $t_1>0$ such that for all $t\geq t_1$,
  \[
    \bar v - v_t(k_t^{-1} e^{-\alpha t})
    \leq e^{-t(\tau -\epsilon)+ (1+\frac{\log k_t}{t\alpha})t(\tau - \epsilon)+O(1)}
    \leq c_3k_t^{\frac{\tau - \epsilon}{\alpha}}.
  \]
  This implies the desired result in this step.

  Finally, by Steps 2-4, we have for each $x\geq 0$, there exist $c_4, \delta_3, t_2 > 0$ such that
  \[
    \mathbf P_{x}(0< W_t\leq k_t)
    \leq c_4(k_t^{\delta_3}+e^{-\delta_3 t})
    , \quad t\geq t_2.
  \]
  Note that the left hand side is always bounded from above by $1$, so we can take $t_2 =0$ in the above statement.
\end{proof}

\subsection{Moments for super-OU processes}
\label{sec: Moments for super-OU processes}
In this subsection,  we want to find some upper bound for the $(1+\gamma)$-th moment of $X_t(g)$,
where $\gamma \in (0,\beta)$ and $g\in \mathcal P$.

\begin{lem}
  \label{lem: control pair for P(M>lambda)}
  There is a $(\theta^2\vee\theta^{1+\beta})$-controller $R$ such that for all $0\leq t\leq 1$, $g\in \mathcal P$, $\lambda >0$ and $\mu\in \mathcal M_c(\mathbb R^d)$, we have
  \[
    \mathbb P_\mu ( |\mathcal{I}_0^tX_t(g)| > \lambda)
    \leq \frac{\lambda}{2}\int_{-2/\lambda}^{2/\lambda}\mu(R|\theta g|) d\theta.
  \]
\end{lem}

\begin{proof}
  It is elementary calculus (see the proof of \cite[Theorem 3.3.6]{Durrett2010Probability} for example) that
for $u>0$ and $x\neq0$,
  \[\frac{1}{u}\int_{-u}^u (1- e^{i\theta x})~d\theta = 2 - \frac{2\sin ux}{ux} \geq \mathbf 1_{ux>2}.\]
  Denote by $R$ the $(\theta^2\vee\theta^{1+\beta})$-controller in Lemma \ref{lem: upper bound for usgx}.(4).
  Then, using Lemma \ref{lem: estimate of exponential remaining} we get
  \begin{align}
    & |\mathbb P_\mu (|\mathcal{I}_0^tX_t(g)| > \lambda)|
       \leq \Big|\frac{\lambda}{2}\int_{-2/\lambda}^{2/\lambda}(1 - \mathbb P_\mu[e^{i\theta \mathcal{I}_0^tX_t(g)}])d\theta\Big| \\
     & \leq \frac{\lambda}{2}\int_{-2/\lambda}^{2/\lambda}|1-e^{\mu(U_t(\theta g)-iP^\alpha_t (\theta g))}|d\theta
     \leq \frac{\lambda}{2}\int_{-2/\lambda}^{2/\lambda}\mu(|U_t(\theta g) - iP^\alpha_t(\theta g)|) d\theta \\
     & \leq \frac{\lambda}{2}\int_{-2/\lambda}^{2/\lambda}\mu(R|\theta g|) d\theta.
      \qedhere
  \end{align}
\end{proof}

\begin{lem}
  \label{lem: temp}
  For all $h \in \mathcal P^+$ and $\mu \in \mathcal M_c(\mathbb R^d)$, there exists $C > 0$ such that for all $\kappa \in \mathbb Z_+ $, $\lambda > 0$ and $0\leq r\leq s\leq t<\infty$ with $s-r \leq 1$, we have
  \[
     \sup_{g \in \mathcal P: Q_\kappa g\leq h}\mathbb P_{\mu}(|\mathcal I_r^sX_t(g)|>\lambda)
    \leq C e^{\alpha r} \Big(\Big( \frac{e^{(t-s)(\alpha - \kappa b)}}{\lambda}\Big)^{1+\beta} + \Big( \frac{e^{(t-s)(\alpha - \kappa b)}}{\lambda}\Big)^{2} \Big).
  \]
\end{lem}

\begin{proof}
  Denote by $R$ the $(\theta^2\vee\theta^{1+\beta})$-controller in Lemma \ref{lem: control pair for P(M>lambda)}.
  Fix $h \in \mathcal P^+$, $\mu \in \mathcal M_c(\mathbb R^d)$ $\kappa \in \mathbb Z_+ $ and $0\leq r\leq s\leq t < \infty$ with $s-r \leq 1$.
  Suppose that $g\in \mathcal P$ satisfies $Q_\kappa g \leq h$.
  Using the Markov property of $X$, we get
  \begin{align}
   & \mathbb P_{\mu}(|\mathcal I_r^sX_t(g)|>\lambda)
      = \mathbb P_\mu \Big[\mathbb P_\mu [| X_{s}(P_{t-s}^\alpha g) -  X_{r}(P_{t-r}^\alpha g)|> \lambda | \mathscr F_r ]\Big] \\
     & = \mathbb P_\mu \big[\mathbb P_{X_r}(| X_{s-r}(P_{t-s}^\alpha g) -  X_{0}(P_{t-r}^\alpha g)|> \lambda)\big] \\
    & = \mathbb P_\mu \big[\mathbb P_{X_r}(|\mathcal I_0^{s-r} X_{s-r}( P_{t-s}^\alpha g) |> \lambda)\big]
      \leq \mathbb P_\mu \Big[ \frac{\lambda}{2}\int_{-2/\lambda}^{2/\lambda}X_r(R|\theta P^\alpha_{t-s}g|) d\theta \Big] \\
     & \leq \mathbb P_\mu \Big[ \frac{\lambda}{2}\int_{-2/\lambda}^{2/\lambda}X_r(R|\theta e^{(t-s)(\alpha- \kappa b)}h|) d\theta \Big] \\
     & \leq \mathbb P_\mu [X_r(Rh) ]
\frac{\lambda}{2}\int_{-2/\lambda}^{2/\lambda}(|\theta e^{(t-s)(\alpha- \kappa b)}|^{1+\beta} + |\theta e^{(t-s)(\alpha- \kappa b)}|^{2})d\theta
     \\ & =  \mu(P_r^\alpha Rh)
\Big(  \frac{2^{2+\beta}}{2+\beta}\Big(\frac{e^{(t-s)(\alpha- \kappa b)}}{\lambda}\Big)^{1+\beta} + \frac{2^{3}}{3}\Big(\frac{e^{(t-s)(\alpha- \kappa b)}}{\lambda}\Big)^2\Big)
    \\ & \leq C e^{\alpha r} \Big(\Big( \frac{e^{(t-s)(\alpha - \kappa b)}}{\lambda}\Big)^{1+\beta} + \Big( \frac{e^{(t-s)(\alpha - \kappa b)}}{\lambda}\Big)^{2} \Big),
  \end{align}
where $C := \Big(\frac{2^{2+\beta}}{2+\beta} + \frac{2^{3}}{3} \Big) \mu(Q_0Rh)>0$.
\end{proof}

For each random variable $\{Y; \mathbb P\}$ and $p \in [1,\infty)$, we write $ \|Y\|_{\mathbb P;p} := \mathbb P[|Y|^p]^{1/p}$.
Recall that we write $\tilde u = \frac{u}{1+u}$ for each $u\neq -1$.
\begin{lem}
  \label{lem: control of mgtrs}
  For all $h \in \mathcal P$, $\mu \in \mathcal M_c(\mathbb R^d)$ and $\gamma\in (0, \beta)$, there exists $C > 0$ such that for all $\kappa \in \mathbb Z_+$ and $0\leq r \leq s\leq t<\infty$ with $s-r \leq 1$, we have
  \[
     \sup_{g \in \mathcal P: Q_\kappa g \leq h} \|\mathcal I_r^s X_t(g) \|_{\mathbb P_\mu;1+\gamma}
    \leq C e^{t\alpha (1- \tilde \gamma)+(t-s) (\alpha \tilde \gamma - \kappa b)}.
  \]
\end{lem}

\begin{proof}
  Fix $h \in \mathcal P$ and $\mu \in \mathcal M_c(\mathbb R^d)$. Let $C_0$ be the constant in the Lemma \ref{lem: temp}.
  For all $\kappa \in \mathbb Z_+$,  $0\leq r\leq s\leq t$ with $s-r \leq 1$,  $g\in \mathcal P$ with $Q_{\kappa} g \leq h$, and $c>0$, we have
  \begin{align}
     & \mathbb P_\mu[|\mathcal I_r^sX_t(g)|^{1+\gamma}]
       = (1+\gamma)\int_0^\infty \lambda^{\gamma} \mathbb P_{\mu}(|\mathcal I_r^sX_t(g)|>\lambda) d\lambda \\
     & \leq (1+\gamma)\int_0^c \lambda^{\gamma} d\lambda +(1+\gamma)\int_c^\infty \lambda^{\gamma}\mathbb P_\mu(|\mathcal I_r^sX_t(g)|> \lambda) d\lambda \\
    & \leq c^{1+\gamma} + C_0  e^{\alpha r}(1+\gamma)\int_c^\infty \bigg(\Big(\frac{e^{(t-s)(\alpha - \kappa b)}}{\lambda}\Big)^{1+\beta}+\Big(\frac{e^{(t-s)(\alpha - \kappa b)}}{\lambda}\Big)^{2}\bigg)\lambda^{\gamma}d\lambda \\
    & \leq c^{1+\gamma} e^{\alpha r} + C_0e^{\alpha r}(1+\gamma)\Big(  \frac{e^{(1+\beta)(t-s)(\alpha- \kappa b)}}{(\beta - \gamma)c^{\beta - \gamma}}  + \frac{e^{2(t-s)(\alpha- \kappa b)}}{(1 - \gamma)c^{1 - \gamma}} \Big).
  \end{align}
  Taking $c = e^{(t-s)(\alpha- \kappa b)}$, we get
  \begin{align}
     & \mathbb P_\mu\big[|\mathcal I_r^s X_t(g)|^{1+\gamma}\big]
      \leq e^{(1+\gamma)(t-s)(\alpha- \kappa b)} e^{\alpha r}\Big(1+ C_0 \frac{1+\gamma}{\beta - \gamma}+ C_0 \frac{1+\gamma}{1 - \gamma}\Big).
  \end{align}
  Note that
  \begin{align}
    & (1+\gamma) (t-s) (\alpha- \kappa b) + \alpha r
      = (t-s)\alpha+(t-s) (\gamma \alpha- (1+\gamma )\kappa b) \\
    & \leq t\alpha+(t-s) (\gamma \alpha- (1+\gamma)\kappa b).
  \end{align}
  So the desired result is true.
\end{proof}

\begin{lem}
  \label{lem:P:M:uc}
  For all $h \in \mathcal P$, $\mu \in \mathcal M_c(\mathbb R^d)$, $\gamma\in (0, \beta)$ and $\kappa \in \mathbb Z_+$, there exists a constant $C > 0$ such that for all $t\geq 0$, we have
  \begin{enumerate}
  \item
    \label{item:P:M:uc:1}
    $\sup_{g\in \mathcal P: Q_\kappa g \leq h}\|X_t(g)\|_{\mathbb{P}_{\mu};1+\gamma}\leq C e^{(\alpha-\kappa b)t}$ provided $\alpha \tilde \gamma > \kappa b$;
  \item
    \label{item:P:M:uc:2}
    $\sup_{g\in \mathcal P: Q_\kappa g \leq h}\|X_t(g)\|_{\mathbb{P}_{\mu};1+\gamma}\leq C te^{\frac{\alpha}{1+\gamma}t}$ provided $\alpha \tilde \gamma = \kappa b$;
  \item
    \label{item:P:M:uc:3}
    $\sup_{g\in \mathcal P: Q_\kappa g \leq h} \|X_t(g)\|_{\mathbb{P}_{\mu};1+\gamma}\leq C e^{\frac{\alpha}{1+\gamma}t}$ provided $\alpha \tilde \gamma < \kappa b$.
  \end{enumerate}
\end{lem}

\begin{proof}
  Fix $\gamma \in (0,\beta)$ and $\mu \in \mathcal M_c(\mathbb R^d)$.
  Let $C$ be the constant in Lemma \ref{lem: control of mgtrs}.
  Using the triangle inequality, for all $\kappa\in \mathbb Z_+$, $g \in \mathcal P$ with $Q_\kappa g \leq h$ and $t\geq 0$, we have
  \begin{align}
    & \|X_t(g)\|_{\mathbb P_\mu;1+\gamma}
      \leq \sum_{l=0}^{\lfloor t\rfloor - 1}\big\| \mathcal{I}_{t-l-1}^{t-l}X_t(g) \big\|_{\mathbb P_\mu;1+\gamma}+\big\| \mathcal{I}_{0}^{t-\lfloor t \rfloor}X_t(g)  \big\|_{\mathbb P_\mu;1+\gamma} + |\mu(P^\alpha_t g)| \\
    & \leq C^{\frac{1}{1+\gamma}} e^{\frac{\alpha}{1+\gamma}t} \sum_{l=0}^{\lfloor t\rfloor} e^{\frac{\gamma\alpha-\kappa (1+\gamma)b}{1+\gamma} l} + e^{(\alpha - \kappa b)t} \mu(h).
  \end{align}
  By calculating the sum on the right, we get the desired result.
\end{proof}

\section{Proofs of main results}
\label{proofs of main results}
In this section, we will prove the main results of this paper.
For simplicity, we will write $\mathbb{\widetilde{P}}_{\mu}=\mathbb{P}_{\mu}(\cdot|D^c)$ in this section.

\subsection{Law of large numbers}
\label{sec: large rate lln}

In this subsection, we prove Theorem \ref{thm: law of large number}.
For this purpose, we first prove the almost sure and $L^{1+\gamma}(\mathbb{P}_{\mu})$ convergence of a family of martingales for $\gamma\in (0, \beta)$. Recall that $L$ is the infinitesimal generator of the OU-process.  For $f\in \mathcal{P}\cap C^2(\mathbb R^d)$ and  $a\in \mathbb R$, we define
\begin{align}
  \label{defmartingale}
  M_t^{f,a}
   :=e^{-(\alpha-ab)t}X_t(f)-\int_0^t e^{-(\alpha-ab)s} X_s((L+ab)f)~ ds.
\end{align}
Let $(\mathscr{F}_t)_{t\geq 0}$ be the natural filtration of $X$.  The following lemma says that $\{M_t^{f,a}: t\geq 0\}$ is a martingale with respect to $(\mathscr{F}_t)_{t\geq 0}$.

\begin{lem}
  \label{lemma25}
  For all $f\in \mathcal{P}\cap C^2(\mathbb R^d)$, $a\in \mathbb R$ and $\mu\in \mathcal M_c(\mathbb R^d)$, the process $(M_t^{f,a})_{t\geq 0}$ is a $\mathbb P_\mu$-martingale with respect to $(\mathscr F_t)_{t\geq 0}$.
\end{lem}

\begin{proof}
  Put$\bar{f} :=(L+ab)f$.
  It follows easily from Ito's formula that
  \begin{align}
    \label{Theorem55}
    P_t^{ab}f(x)
    = f(x)+\int_0^t P_s^{ab}\bar{f}(x)~ds,\quad t\geq 0,x\in \mathbb R^d,
  \end{align}
  where $P_t^{ab} := e^{abt}P_t$.
  For $0\leq s\leq t$, we have
  \begin{align}
    \label{martingale1}
    & \quad\mathbb{P}_{\mu}[M_t^{f,a}|\mathscr{F}_s]
    =e^{-(\alpha-ab)t}\mathbb{P}_{\mu}\left[X_t(f)|\mathscr{F}_s\right]-\mathbb{P}_{\mu}\Big[\int_0^t e^{-(\alpha-ab)u}X_u(\bar{f})~ du\Big|\mathscr{F}_s\big] \\
    & =e^{-(\alpha-ab)t} X_s(P_{t-s}^{\alpha}f)-\int_0^s e^{-(\alpha-ab)u} X_u(\bar{f})~ du - \int_s^t e^{-(\alpha-ab)u}X_s(P_{u-s}^{\alpha} \bar{f})~ du.
  \end{align}
  Using \eqref{Theorem55} and Fubini's theorem, we have
  \begin{align}
    & \int_s^t e^{-(\alpha-ab)u}X_s(P_{u-s}^{\alpha} \bar{f})~ du=e^{-(\alpha-ab)s}\int_s^tX_s(P_{u-s}^{ab}\bar{f})~du\\
    & = e^{ - ( \alpha - ab ) s } X_s\left( \int_0^{t-s} P_{u}^{ab} \bar{f}~ du\right)
      = e^{-(\alpha-ab)s}\left(X_s(P_{t-s}^{ab}f) - X_s(f) \right) \\
    & = e^{-(\alpha-ab)t} X_s( P_{t-s}^{\alpha}f) - e^{ - ( \alpha - ab ) s} X_s(f).
  \end{align}
  Using this and \eqref{martingale1}, we get the desired result.
\end{proof}

Recall that, for $p\in \mathbb Z_+^d$,  $\phi_p$ is an eigenfunction of $L$ corresponding to the eigenvalue $-|p|b$ and $ H_t^p =e^{-(\alpha-|p|b)t}X_t(\phi_p)$ for each $t\geq 0$.

\begin{lem}
  \label{lem:M:L:ML}
   For all $\mu\in \mathcal M_c(\mathbb R^d)$ and $p \in \mathbb Z_+^d$, $(H^p_t)_{t\geq 0}$ is a $\mathbb P_{\mu}$-martingale with respect to $(\mathscr F_t)_{t\geq 0}$.
    Moreover if $\alpha\tilde \beta>|p|b$, the martingale is bounded in $L^{1+\gamma}(\mathbb P_\mu)$ for each $\gamma\in (0, \beta)$.
  Thus the limit $ H_{\infty}^p := \lim_{t\rightarrow \infty}H_t^p $  exists $\mathbb{P}_{\mu}$-a.s. and in $L^{1+\gamma}(\mathbb{P}_{\mu})$ for each $\gamma \in (0,\beta)$.
\end{lem}

\begin{proof}
  Fix a $\mu \in \mathcal M_c(\mathbb R^d)$ and a $p \in \mathbb Z_+^d$.
  It follows from Lemma \ref{lemma25} that $(H_t^p)_{t\geq 0}$ is a $\mathbb P_\mu$-martingale.
  Further suppose that $\alpha \tilde \beta > |p| b$.
  Then there exists a $\gamma_0 \in (0,\beta)$ which is close enough to $\beta$ so that $\alpha\tilde \gamma>|p|b$ for each $\gamma\in [\gamma_0, \beta)$.
  Using  Lemma \ref{lem:P:M:uc} and the fact $\kappa_{\phi_p}=|p|$, we get that, for each $\gamma\in [\gamma_0, \beta)$, there exists a constant $C>0$ such that
  \[
    \|H_t^p\|_{\mathbb P_\mu;1+\gamma}
    \leq C e^{-(\alpha-|p|b)t}e^{(\alpha-|p|b)t}
    = C
    , \quad t\geq 0.
  \]
  For each $\gamma\in (0, \gamma_0)$ there exists a constant $C'>0$ such that
  \[
    \| H_t^p \|_{\mathbb P_\mu;1+\gamma}
    \leq \| H_t^p \|_{\mathbb P_\mu;1+\gamma_0}
    \leq C',
    \quad t\geq 0.
  \]
  Therefore, for each $\gamma \in (0,\beta)$, the martingale $(H_t^p)_{t\geq 0}$ is bounded in $L^{1+\gamma}(\mathbb{P}_{\mu})$.
\end{proof}

\begin{lem}
  \label{lem: control of wt}
  Suppose that $\mu\in \mathcal M_c(\mathbb R^d)$ and that $p \in \mathbb Z_+^d$ satisfies $\alpha \tilde \beta > |p|b$.
  Then for each $\gamma \in (0,\beta)$ satisfying $\alpha \tilde \gamma > |p|b$, there exists a constant $C> 0$ such that,
  \[
    \|H^p_t-H^p_s\|_{\mathbb{P}_{\mu};1+\gamma}
    \leq C e^{-(\alpha \tilde \gamma-|p|b)s},
    \quad 0 \leq s < t \leq \infty.
  \]
\end{lem}

\begin{proof}
  Thanks to Lemma \ref{lem:M:L:ML}, we only  need to prove the inequality  when $0\leq s < t<\infty$.
  Suppose $p\in \mathbb{Z}_+^d$, $\mu\in \mathcal M_c(\mathbb R^d)$ and  $\gamma \in (0,\beta)$ with $\alpha \tilde \gamma > |p|b$ are fixed.
  Using Lemma \ref{lem: control of mgtrs} with $g=\phi_p$ and $k=|p|$,  we know that there exists a constant $C_1>0$ such that for all $0\leq r\leq s $ with $s-r\leq1$,
  \begin{align}
    \|H^p_s-H^p_r\|_{\mathbb P_\mu; 1+\gamma}
    \leq  C_1 e^{-(\alpha\tilde \gamma-|p|b)s}.
  \end{align}
  Thus there exists $C_2>0$ such that for all $0\leq s<t$,
  \begin{align}
    & \|H^p_t-H^p_s\|_{\mathbb{P}_{\mu};1+\gamma} \\
    & \leq \|H^p_{\lfloor s \rfloor+1}-H^p_s\|_{\mathbb{P}_{\mu};1+\gamma}+\sum_{k=\lfloor s \rfloor+1}^{\lfloor t \rfloor}\|H^p_{k+1}-H^p_{k}\|_{\mathbb{P}_{\mu};1+\gamma}+\|H^p_t-H^p_{\lfloor t \rfloor+1}\|_{\mathbb{P}_{\mu};1+\gamma} \\
    & \leq C_1 \Big(e^{-(\alpha \tilde \gamma- |p|b) s}+\sum_{k=\lfloor s \rfloor+1}^{\lfloor t \rfloor} e^{-(\alpha \tilde \gamma- |p|b) k} + e^{-(\alpha \tilde \gamma-|p|b) t}\Big)
      \leq C_2e^{-(\alpha \tilde \gamma-|p|b)s}.
      \qedhere
  \end{align}	
\end{proof}

\begin{proof}[Proof of Theorem \ref{thm: law of large number}]
	Fix $f \in \mathcal P$ such that $\alpha \beta > \kappa_f b (1+\beta)$ and $\mu \in \mathcal M_c(\mathbb R^d)$.
	Write
  \[
    f
    = \sum_{p\in \mathbb Z_+^d:|p|\geq \kappa_f}\langle f,\phi_p\rangle_\varphi \phi_p
    =: \sum_{p\in \mathbb Z_+^d:|p|= \kappa_f}\langle f,\phi_p\rangle_\varphi \phi_p+\widetilde{f}.
  \]
	Then
  \begin{align}
    & e^{-(\alpha-\kappa_fb)t}X_t(f)=
      \sum_{p\in \mathbb Z_+^d:|p|= \kappa_f}\langle f,\phi_p\rangle_\varphi H_t^p+e^{-(\alpha-\kappa_fb)t} X_t(\widetilde{f}),
      \quad t\geq 0.
  \end{align}
	According to Lemma \ref{lem:M:L:ML}, we have
  \begin{align}
    \label{as convergence}
    \sum_{p\in \mathbb{Z}_+^d:|p|= \kappa_f}\langle f,\phi_p\rangle_\varphi H_t^p
    \xrightarrow[t\to \infty]{} \sum_{p\in \mathbb{Z}_+^d:|p|=\kappa_f}\langle f, \phi_p\rangle_{\varphi} H_{\infty}^p,
  \end{align}
  $\mathbb{P}_{\mu}$-a.s. and in $L^{1+\gamma}(\mathbb{P}_{\mu})$ for each $\gamma\in(0,\beta)$.
  Therefore, it suffices to show that
  \begin{align}
    J_t
    :=e^{-(\alpha-\kappa_fb)t}X_t( \widetilde{f}),
    \quad t\geq 0,
  \end{align}
  converges to $0$ in $L^{1+\gamma}(\mathbb{P}_{\mu})$ for all $\gamma\in(0,\beta)$, and converges almost surely provided $f$ is twice differentiable and all its second order partial derivatives are in $\mathcal{P}$.

  \emph{Step 1.} Let $g\in \mathcal P$.
  Let $\kappa > 0$ be such that $\kappa < \kappa_g$ and $\kappa b < \alpha \tilde \beta$.
  We will show that for each $\gamma \in (0,\beta)$ there exist $C_1,\delta_1 > 0$ such that
  \[
    \|e^{-(\alpha - \kappa b)t} X_t(g)\|_{\mathbb P_\mu;1+\gamma}
    \leq C_1 e^{-\delta_1 t},
    \quad t\geq 0.
  \]
  In order to do this, we choose a $\gamma_0 \in (0,\beta)$ close enough to $\beta$ such that $\kappa b< \alpha \tilde \gamma$ for each $\gamma \in [\gamma_0, \beta)$.
  According to Lemma \ref{lem:P:M:uc}, we have for each $\gamma \in (0,\beta)$,
  \begin{enumerate}
  \item
    if $\gamma \in [\gamma_0, \beta)$ and $\alpha\tilde \gamma> \kappa_g b$, then there exists $C_2>0$ such that
    \[
      \|e^{-(\alpha - \kappa b)t}  X_t(g)\|_{\mathbb P_\mu;1+\gamma}
      \leq C_2 e^{-(\alpha-\kappa b)t}e^{(\alpha-\kappa_g b)t}
      \leq C_2  e^{-(\kappa_g - \kappa )bt},
      \quad t\geq 0;
    \]
  \item
    if $\gamma \in [\gamma_0, \beta)$ and $\alpha\tilde \gamma=\kappa_g b$, then there exists $C_3>0$ such that
    \[
      \|e^{-(\alpha - \kappa b)t}  X_t(g)\|_{\mathbb P_\mu;1+\gamma}
      \leq C_3 t e^{-(\alpha - \kappa b)t}e^{\frac{\alpha}{1+\gamma}t}
      = C_3 t e^{-(\alpha \tilde \gamma - \kappa b)t},
      \quad t\geq 0;
    \]
  \item
    if $\gamma \in [\gamma_0, \beta)$ and $\alpha\tilde \gamma < \kappa_g b$, then there exists $C_4>0$ such that
    \[
      \|e^{-(\alpha - \kappa b)t}  X_t(g)\|_{\mathbb{P}_{\mu};1+\gamma}
      \leq C_4  e^{-(\alpha - \kappa b)t}e^{\frac{\alpha}{1+\gamma}t}
      = C_4  e^{-(\alpha \tilde \gamma - \kappa b)t},
      \quad t\geq 0;
    \]
  \item
  if $\gamma \in (0,\gamma_0)$,  then
    thanks to (1)--(3) above and the fact that \[\|e^{-(\alpha - \kappa b)t} X_t(g)\|_{\mathbb{P}_{\mu};1+\gamma}
      \leq \|e^{-(\alpha - \kappa b)t} X_t(g)\|_{\mathbb{P}_{\mu};1+\gamma_0},\] there exist $C_5, \delta_2 >0$ such that
    \[
      \|e^{-(\alpha - \kappa b)t} X_t(g)\|_{\mathbb{P}_{\mu};1+\gamma}
      \leq C_5e^{-\delta_2 t},
      \quad t\geq 0.
    \]
  \end{enumerate}
  Thus, the desired conclusion in this step is valid.
  In particular, by taking $g = \widetilde f$ and $\kappa = \kappa_f$, we get that $J_t$ converges to $0$ in $L^{1+\gamma}(\mathbb{P}_{\mu})$ for any $\gamma\in(0,\beta)$.

  \emph{Step 2.}
  We further assume that $f\in C^2(\mathbb R^d)$ and $D^2f \in \mathcal{P}$.
  We will show that $J_t$ converges to $0$ almost surely.
  For $a \geq 0$, $ t\geq 0$, and $g\in \mathcal{P}\cap C^2(\mathbb{R}^d)$ satisfying $D^2g\in \mathcal{P}$, we define
  \begin{align}
    L_t^{g,a}
    & :=\int_0^t e^{-(\alpha-ab)s}X_s((L+ab)g) ds,
    \quad
    Y_t^{g,a}
    :=\int_0^t e^{-(\alpha-ab)s}|X_s((L+ab)g)|ds.
  \end{align}
  Now choose $a_0 \in (\kappa_{f}, \kappa_f + 1)$ close enough to $\kappa_f$ so that $a_0 b < \alpha \tilde \beta$.
  According to \eqref{defmartingale},
  \begin{align}
    J_t
    = e^{-(a_0-\kappa_f)bt} (M_t^{\widetilde{f}, a_0}+L_t^{\widetilde{f}, a_0}),
    \quad t\geq 0.
  \end{align}
  So we only need to show that
  \begin{align}
    e^{-(a_0-\kappa_f)b t}M_t^{\widetilde{f},a_0}
    \xrightarrow[t\to \infty]{} 0,
    \quad e^{-(a_0-\kappa_f)b t}L_t^{\widetilde{f},a_0}
    \xrightarrow[t\to \infty]{} 0,
    \quad \mathbb{P}_{\mu}\text{-a.s.}
  \end{align}
  Notice that $\kappa_{(L+a_0 b)\widetilde{f}}\geq \kappa_{\widetilde{f}}\geq \kappa_f+1 > a_0$.
  By Step 1, for any fixed $\gamma\in (0,\beta)$, there exist $C_6, \delta_3>0$ such that for each $t\geq 0$,
  \[
    \| e^{-(\alpha-a_0 b)t}X_t(\widetilde{f}))\|_{\mathbb{P}_{\mu};1+\gamma}
    \leq C_6 e^{-\delta_3 t},
    \quad \|e^{-(\alpha-a_0 b)t}X_t(L\widetilde{f}+a_0 b\widetilde{f})\|_{\mathbb{P}_{\mu};1+\gamma}
    \leq C_6 e^{-\delta_3 t}.
  \]
  Now, by the triangle inequality, for each $t\geq 0$,
  \begin{align}
    & \|L_t^{\widetilde{f},a_0}\|_{\mathbb{P}_{\mu};1+\gamma}
      \leq\|Y_t^{\widetilde{f},a_0}\|_{\mathbb{P}_{\mu};1+\gamma} \\
    & \leq \int_0^t \|e^{-(\alpha-a_0 b)s}X_s( L\widetilde{f}+a_0 b\widetilde{f})\|_{\mathbb{P}_{\mu};1+\gamma}ds\leq C_6 \int_0^t e^{-\delta_3 s}ds\leq\frac{C_6}{\delta_3}.
  \end{align}
  Since $Y_t^{\widetilde{f},a_0}$ is increasing in $t$, it converges to some finite random variable $Y_{\infty}^{\widetilde{f},a_0}$ almost surely and in $L^{1+\gamma}(\mathbb{P}_{\mu})$.
  Consequently,  we have
  \begin{align}
    \lim_{t\rightarrow \infty}e^{-(a_0 - \kappa_f)bt}|L_t^{\widetilde{f},a_0}|
    \leq  \lim_{t\rightarrow \infty}e^{-(a_0 - \kappa_f)bt}|Y_t^{\widetilde{f},a_0}|=0,
    \quad \mathbb P_\mu\text{-a.s.}
  \end{align}
  On the other hand, the martingale $M_t^{\widetilde{f},a_0}$ satisfies
  \begin{align}
    \|M_t^{\widetilde{f},a_0}\|_{\mathbb{P}_{\mu};1+\gamma}
    \leq \|e^{-(\alpha-a_0 b)t}X_t(\widetilde{f})\|_{\mathbb{P}_{\mu};1+\gamma}+\|L_t^{\widetilde{f},a_0}\|_{\mathbb{P}_{\mu};1+\gamma}
    \leq C_6(e^{-\delta_3 t}+\frac{1}{\delta_3}),
    \quad t\geq 0.
  \end{align}
  This implies that the martingale converges almost surely.
  Consequently,
  \[
    \lim_{t\rightarrow\infty} e^{-(a_0-\kappa_f)bt}M_t^{\widetilde{f},a_0}
    = 0,
    \quad \mathbb P_\mu\text{-a.s.}.
    \qedhere
  \]
\end{proof}

\subsection{Central limit theorems for unit time intervals}
\label{sec:critical}
In this subsection, we will establish the following CLT.
\begin{thm}
  \label{lem:PR:LC}
If  $\mu \in \mathcal M_c(\mathbb R^d)$ and $f\in \mathcal{P}\setminus \{0\}$,
then  under $\mathbb{P}_{\mu}(\cdot | D ^c)$, we have
  \begin{align}
    \label{eq:PR:LC:1}
    \Upsilon^f_t
    := \frac{X_{t+1} (f) - X_t(P_1^\alpha f)}{\| X_t\|^{1-\tilde \beta}}
    \xrightarrow[t\to \infty]{d}\zeta^f_0,
  \end{align}
  where $\zeta^f_0$ is a $(1+\beta)$-stable random variable with characteristic function $\theta\mapsto e^{\langle Z_1(\theta f), \varphi\rangle}$.
\end{thm}

In fact, we prove a stronger result:

\begin{prop}
  \label{thm:Key}
  For all $\mu \in \mathcal M_c(\mathbb R^d)$ and $g \in \mathcal P \setminus \{0\}$, there exist $C,\delta>0$ such that
  for all $t\geq 1$ and $f \in \mathcal P_g:= \{\theta T_ng:n \in \mathbb Z_+, \theta \in [-1,1]\}$, we have
  \[
    \mathbb P_\mu
    \Big[  |\mathbb P_\mu [e^{i\Upsilon^f_t} - e^{\langle Z_1f, \varphi\rangle}; D^c | \mathscr F_t ]  |\Big]
    \leq C e^{- \delta t}.
  \]
\end{prop}

\begin{proof}
  Fix  $\mu \in \mathcal M_c(\mathbb R^d)$ and $g \in \mathcal P\setminus \{0\}$.

  \emph{Step 1.} Write $ A_t(\epsilon) :=\{ \|X_t\| \geq e^{(\alpha - \epsilon)t} \} $ for $t\geq 0$ and $\epsilon > 0$.
  We will show that for all $f\in \mathcal P \setminus \{0\}$, $\epsilon > 0$ and $t\geq 0$, it holds that
  \[
    \mathbb P_\mu \Big[ | \mathbb P_\mu [e^{i\Upsilon^f_t} - e^{\langle Z_1(\theta f), \varphi\rangle}; D^c | \mathscr F_t ]| \Big]
    \leq J^f_1(t,\epsilon)+J^f_2(t,\epsilon)+J^f_3(t,\epsilon),
  \]
where
\begin{align}
\label{eq: Def of Ji}
  &J^f_1(t,\epsilon):= \mathbb{P}_{\mu} [ | X_t(Z'''_1(\theta_t f)) |; A_t(\epsilon) ],
 \quad
  J^f_2(t,\epsilon):= \mathbb{P}_{\mu}[|X_t( Z_1(\theta_t f))-\langle Z_1f, \varphi\rangle |; A_t(\epsilon)],
  \\ & J_3(t,\epsilon):=2\mathbb{P}_{\mu}(A_t (\epsilon)\Delta D^c),
 \quad
   \theta_t := \|X_t\|^{-(1 - \tilde \beta)}.
\end{align}
In fact, it follows from \eqref{eq: key equality}, the definitions of $U_1$, $Z'''_1$ and $Z_1$, that for all $t\geq 0$,
\begin{align}
  \label{eq: need1}
  & \mathbb{P}_{\mu}[e^{i\Upsilon^f_t}|\mathscr{F}_t]
    = \mathbb{P}_{\mu}[\exp\{i\theta_t X_{t+1} (f) - i \theta_t X_t(P_1^\alpha f)\} |\mathscr{F}_{t}] \\
  & = \exp\{X_t((U_1 - iP^\alpha_1 ) (\theta_t f))\}
    = \exp\{X_t((Z_1 + Z'''_1) (\theta_t f))\}.
\end{align}
From Lemma \ref{lem: charactreisticfunction}, we  get that $\theta\mapsto \langle Z_1(\theta f),\varphi\rangle$ is the characteristic function of some $(1+\beta)$-stable random variable, and then  $\operatorname{Re} \langle Z_1f, \varphi\rangle \leq 0$.
Using this, \eqref{eq: need1}, \eqref{eq: -v has positive real part} and the fact $|e^{-x} - e^{-y}| \leq |x-y|$ for all $x,y \in \mathbb C_+$, we get for each $t\geq 0$ and $\epsilon> 0$,
\begin{align}
  \label{eq: inequality that will used later}
  & \mathbb{P}_\mu \Big[ |  \mathbb{P}_\mu [ e^{i\Upsilon^f_t} - e^{\langle Z_1f,\varphi \rangle} ; D^c | \mathscr F_{t}]   |\Big]  \\
  &  \leq \mathbb{P}_\mu   \Big[ |    \mathbb{P}_\mu [ e^{i \Upsilon^f_t }-e^{\langle Z_1f, \varphi\rangle}; A_{t}(\epsilon) | \mathscr F_{t}] |  + 2\mathbb P_\mu ( A_{t}(\epsilon) \Delta D^c | \mathscr F_{t}) \Big] \\
  & = \mathbb{P}_{\mu}\Big[ |\mathbb{P}_\mu [e^{i\Upsilon^f_t}| \mathscr F_{t}]-e^{\langle Z_1f, \varphi\rangle}| ; A_{t}(\epsilon) \Big] + J_3(t,\epsilon) \\
  & \leq \mathbb{P}_\mu \Big[ |e^{X_{t}((Z_1+Z'''_1) (\theta_t f))}-e^{\langle Z_1f, \varphi\rangle} | ; A_{t}(\epsilon) \Big]+  J_3(t,\epsilon) \\
  & \leq \mathbb{P}_\mu \Big[ | X_{t} ( (Z_1+Z'''_1)(\theta_t f)) - \langle Z_1f, \varphi\rangle | ;A_{t}(\epsilon)\Big]+  J_3(t,\epsilon) \\
  & \leq J^f_1(t,\epsilon) + J^f_2(t,\epsilon)+ J_3(t,\epsilon).
\end{align}

\emph{Step 2.} We will show that for $\epsilon>0$ small enough, there exist  $C_2, \delta_2>0$ such that for all $t\geq 1$ and
$f \in \mathcal P_g$, we have $ J^f_1(t,\epsilon) \leq C_2e^{-\delta_2 t}$.

In fact, let $\delta_0 >0$ be the constant in Lemma \ref{lem: upper bound for usgx}.(7) and let $R$ be the corresponding $(\theta^{2+\beta}\vee \theta^{1+\beta+\delta_0})$-controller.
Acording to Step 1 in the proof of Lemma \ref{lem:m}, there exists $h_{2} \in \mathcal P^+$ such that for each $f \in \mathcal P_g$ it holds that $|f| \leq h_{2}$.
Then, we have for all $t\geq 0$, $\epsilon> 0$ and $f\in \mathcal P_g$,
\begin{align}
  & |Z'''_1(\theta_t f)|\mathbf{1}_{A_{t}(\epsilon)}
    \leq R(|\theta_{t} f|)\mathbf{1}_{A_{t}(\epsilon)}
    \leq R \Big(\frac{h_{2}}{e^{(\alpha-\epsilon)t(1-\tilde \beta)}}\Big)
    \leq \sum_{\rho \in \{\delta_0,1\}}e^{-\frac{1+\beta+\rho}{1+\beta}(\alpha-\epsilon)t}Rh_{2}.
\end{align}
Thus for all $t\geq 0$, $\epsilon> 0$ and $f\in \mathcal P_g$,
\begin{align}
  \label{eq: estimate of J1}
  J^f_1(t,\epsilon)
& \leq \sum_{\rho \in \{\delta_0,1\}}e^{-\frac{1+\beta+\rho}{1+\beta}(\alpha-\epsilon)t}\mathbb{P}_{\mu}[X_{t}(Rh_2)]
  \leq \sum_{\rho \in \{\delta_0,1\}} \mu(Q_0 R h_{2}) e^{-(\alpha\frac{\rho}{1+\beta}-\epsilon\frac{1+\beta+\rho}{1+\beta})t},
\end{align}
where $Q_0$ is defined by \eqref{eq:Q}.
By taking $\epsilon>0$ small enough, we get the desired result in this step.

\emph{Step 3.}
We will show that for $\epsilon>0$ small enough there exist $C_3, \delta_3 > 0$ such that for all $t \geq 0$ and $f\in \mathcal P_g$,  we have $ J^f_2(t,\epsilon) \leq C_3 e^{-\delta_3 t}$.
In fact, for all $t\geq 0$, and $f\in \mathcal P_g$,
\begin{align}
  & X_{t}(Z_1(\theta_t f))- \langle Z_1f, \varphi\rangle
     = \theta_t^{1+\beta} X_t(Z_1 f) - \langle  Z_1 f,\varphi \rangle
     = \frac{1}{\|X_{t}\|}X_t(Z_1f - \langle  Z_1 f ,\varphi \rangle),
\end{align}
and therefore,
\begin{align}
  \label{eq: prevJ2}
  J^f_2(t,\epsilon)
  & = \mathbb P_\mu\Big[\Big|  \frac{1}{\|X_{t}\|}X_t(Z_1f - \langle  Z_1 f ,\varphi \rangle) \Big|;A_t(\epsilon)\Big]
    \leq e^{-(\alpha-\epsilon)t} \mathbb{P}_{\mu}[|X_t (q_f) |].
\end{align}
where $ q_f = Z_1 f-\langle  Z_1 f,\varphi\rangle \in \mathcal P^*$.
It follows from Lemma \ref{lem:P:R} that there exists $h_{3}\in \mathcal{P}$ such that for each $ f\in \mathcal P_g$, we have $Q_1 (\operatorname{Re} q_f) \leq h_{3} \text{ and } Q_1 (\operatorname{Im} q_f)\leq h_3$, where $Q_1$ is given by \eqref{eq:Q} with $\kappa=1$.
In the rest of this step, we  fix a $\gamma\in(0,\beta)$ small enough such that $\alpha \gamma < b < (1+\gamma)b$.
According to Lemma \ref{lem:P:M:uc}.(3) (with $\kappa=1$), there exists $C_{3}>0$ such that for all $t\geq 0$ and $f\in \mathcal P_g$,
\begin{align}
  & \mathbb{P}_{\mu}\left[\left|X_{t}(q_f)\right|\right]
    \leq \| X_{t}( \operatorname{Re} q_f)\|_{\mathbb{P}_{\mu,1+\gamma}} + \| X_{t}(\operatorname{Im} q_f)\|_{\mathbb{P}_{\mu,1+\gamma}} \\
 & \leq 2\sup_{q\in \mathcal P: Q_1 q\leq h_{3}} \|X_t(q)\|_{\mathbb P_\mu; 1+\gamma} \leq C_{3} e^{\frac{\alpha t}{1+\gamma}}.
\end{align}
Therefore, for all $t\geq 0, \epsilon > 0$ and $f \in \mathcal P_g$, we have
\begin{align}
  \label{eq: right bound for J2}
   J^f_2(t, \epsilon)
    \leq  C_3 e^{-(\alpha-\epsilon)t}e^{\frac{\alpha t}{1+\gamma}}
   \leq C_{3} e^{-(\alpha\tilde \gamma -\epsilon)t}.
\end{align}
By taking $\epsilon >0$ small enough, we get the required result in this step.

\emph{Step 4.}
We will show that, for each $\epsilon\in (0,  \alpha)$, there exist $C_4,\delta_4>0$ such that for all $t\geq 1$, $J_3(t,\epsilon)\leq C_4e^{-\delta_4 t}.$
In fact, we have  for all $t\geq 0, \epsilon >0$,
\[
  \mathbb P_{\mu}(A_{t}(\epsilon), D)
  = \mathbb P_{\mu}[\mathbb P_{\mu}(D|\mathscr F_t);A_t(\epsilon)]
  = \mathbb P_\mu[e^{-\bar v\|X_t\|};A_t(\epsilon)]
  \leq \exp({-\bar v \|\mu\|e^{(\alpha - \epsilon)t}}).
\]
On the other hand, by Proposition \ref{lem: control of XT}, for each $\epsilon \in (0, \alpha)$, there exists  $C_{4}, \delta_{4}>0$ such that for all $t\geq 0$,
\begin{align}
  \mathbb P_\mu(A_t(\epsilon)^c,D^c)
  \leq \mathbb P_\mu(0 < e^{-\alpha t}\|X_t\|
  \leq e^{ - \epsilon t}) \leq C_{4} (e^{-\epsilon \delta_{4} t}+e^{-\delta_{4} t}).
\end{align}
 Combining these results, we get the desired result in this step.

 \emph{Step 5.} Combining the results in Steps 1--4, we immediately get the desired result.
\end{proof}

The following corollary  will be used later in the proof of Theorem \ref{thm:M}.
\begin{cor}
  \label{cor:MI}
If  $g\in \mathcal{P}\setminus\{0\}$ and $\mu\in \mathcal M_c(\mathbb R^d)$, then there exist $C,\delta>0$ such that for
all $l\leq n$ in $\mathbb Z_+$ and  $(f_j)_{j=l}^n\subset \mathcal P_g$,
\begin{align}
  \label{32corollary}
  \Big|\mathbb{\widetilde{P}}_{\mu}\Big[\prod_{k=l}^ne^{i \Upsilon^{f_k}_{k} }-\prod_{k=l}^n e^{\langle Z_1f_k, \varphi\rangle}\Big]\Big|\leq C e^{-\delta l}.
\end{align}
\end{cor}
\begin{proof}
  For  $l\leq n$ in $\mathbb Z_+$, $k \in \{l,\dots,n\}$ and $(f_j)_{j=l}^n\subset \mathcal P_g$, define
  \[
    a_k
    :=  \mathbb{\widetilde{P}}_{\mu}\Big[\prod_{j=l}^{k} e^{i\Upsilon_j^{f_j}}\Big] \times \Big(\prod_{j=k+1}^{n} e^{ \langle Z_1f_j,\varphi \rangle} \Big).
  \]
  Then for all $l\leq n$ in $\mathbb Z_+$, $k \in \{l,\dots,n\}$ and $(f_j)_{j=l}^n\subset \mathcal P_g$, we have
  \begin{align}
    & a_{k} - a_{k-1}
      =\mathbb{P}_{\mu}(D^c)^{-1} \mathbb{P}_{\mu}\Big[(e^{i\Upsilon^{f_k}_k}-e^{\langle Z_1f_k, \varphi\rangle})\prod_{j=l}^{k-1} e^{i\Upsilon_j^{f_j}};D^c\Big] \Big(\prod_{j=k+1}^n e^{\langle Z_1f_j, \varphi\rangle}\Big)\\
    & =\mathbb{P}_{\mu}(D^c)^{-1} \mathbb{P}_{\mu}\Big[\mathbb P_\mu[e^{i\Upsilon_k^{f_k}}-e^{\langle Z_1f_k, \varphi \rangle}; D^c|\mathscr F_k] \prod_{j=l}^{k-1} e^{i\Upsilon_j^{f_j}}\Big] \Big(\prod_{j=k+1}^{n}e^{\langle Z_1f_j, \varphi\rangle}\Big).
  \end{align}
  By Lemma \ref{thm:Key}, there exist $C_0,\delta_0 >0$ such that for all $l\leq n$ in $\mathbb Z_+$,  $k \in \{l,\dots , n\}$, and $(f_j)_{j=l}^n\subset \mathcal P_g$, we have
  \begin{align}
    | a_{k} - a_{k-1}|
    & \leq \mathbb{P}_{\mu}(D^c)^{-1}\mathbb{P}_{\mu}\Big[\big|\mathbb P_\mu[e^{i\Upsilon_k^{f_k}}-e^{\langle Z_1f_k, \varphi \rangle}; D^c | \mathscr{F}_k]\big|\Big]
    \leq C_0 e^{-\delta_0 k}.
  \end{align}
  Therefore, there exist $C,\delta >0$ such that for all $l\leq n$ in $\mathbb Z_+$ and each $(f_j)_{j=l}^n\subset \mathcal P_g$, we have
  \begin{align}
    \text{LHS of \eqref{32corollary}}
    & = \left|a_n-a_{l-1}\right|
      \leq \sum_{k=l}^n\left|a_{k}-a_{k-1}\right|
      \leq \sum_{k=l}^n C_0 e^{-\delta_0 k}
      \leq C e^{- \delta l}.
      \qedhere
  \end{align}
\end{proof}

\subsection{Central limit theorem for $f\in \mathcal C_s$}
\label{sec: small rate}

\begin{proof}[Proof of Theorem \ref{thm:M}.(\ref{thm:M:1})]
	Fix $\mu\in \mathcal M_c(\mathbb R^d)$, $f\in \mathcal C_s$ and $t_0 > 1$  large enough so that $ \lceil t - \ln t\rceil \leq \lfloor t \rfloor - 1$ for all $t\geq t_0$.
  For each $t\geq t_0$, in this proof we write $\theta_t = \|X_t\|^{\tilde \beta - 1}$,
  \begin{multline}
    \label{eq:PM:CLTS:1}
    \theta_t  X_t(f)
    = I^f_1(t) + I^f_2(t) + I^f_3(t)
    := \Big(\sum_{k=0}^{\lfloor t-\ln t \rfloor} \theta_t \mathcal I_{t-k-1}^{t-k} X_t(f) \Big)\\
    + \Big( \theta_t \mathcal I_0^{t-\lfloor t \rfloor} X_t(f)   + \sum_{k=\lceil t-\ln t \rceil}^{\lfloor t \rfloor-1} \theta_t \mathcal I_{t-k-1}^{t-k} X_t(f) \Big) + (\theta_t X_0(P_t^\alpha f) ),
  \end{multline}
  and $ I^f_0(t) := \sum_{k=0}^{\lfloor t-\ln t \rfloor} \Upsilon_{t-k-1}^{T_k \tilde f},$ where $\tilde f:= e^{\alpha(\tilde \beta - 1)} f$.

  \emph{Step 1.} We show that $I^f_0(t) \xrightarrow[t\to \infty]{d} \zeta^f$.
  In fact, for each $k \in \mathbb Z_+$, we have  $T_{k} \tilde f \in \mathcal P_{\tilde f}:=\{\theta T_n \tilde f: n \in \mathbb Z_+, \theta \in [-1,1]\}$.
  Therefore from Corollary \ref{cor:MI} we get that there exist $C_1,\delta_1 > 0$ such that
  \begin{align}
    \Big|\mathbb{\widetilde{P}}_{\mu} [e^{i I^f_0(t)} ]-\exp\Big(\sum_{k=0}^{\lfloor t-\ln t \rfloor} \langle Z_1T_{k}\tilde f, \varphi\rangle \Big)\Big|
    \leq C_1 e^{-\delta_1(t - \lfloor t - \ln t\rfloor)},
    \quad t\geq t_0.
  \end{align}
  On the other hand, using \eqref{eq:PL:S:1} and the fact that $\varphi(x)dx$ is the invariant probability of the semigroup $(P_t)_{t\geq 0}$, we have
  \begin{align}
    \label{eq:PM:CLTS:2}
    & \sum_{k=0}^\infty \langle Z_1 T_{k} \tilde f, \varphi \rangle
    = \sum_{k=0}^\infty \int_0^1 \langle P_u^\alpha ((-iP_{1 - u}^\alpha T_k \tilde f)^{1+\beta}), \varphi\rangle ~du
    \\& = \sum_{k=0}^\infty \int_0^1 e^{\alpha u} \langle  (-iP_{1 - u}^\alpha T_{k}\tilde f)^{1+\beta}, \varphi \rangle ~du
    \\& = \sum_{k=0}^\infty \int_0^1 \langle  (-iT_{k+1 - u} f)^{1+\beta}, \varphi\rangle~du
    = \int_0^\infty \langle  (-iT_{u} f)^{1+\beta}, \varphi\rangle~du = m[f].
  \end{align}
  Therefore, we have $\mathbb{\widetilde{P}}_{\mu} [e^{i I^f_0(t)} ] \xrightarrow[t\to \infty]{} e^{m[f]}$. Since $I_0^f(t)$ is linear in $f$, we can replace $f$ with $\theta f$, $\theta \in \mathbb R$, and then the desired result in this step follows.

  \emph{Step 2.} We show that $I^f_1(t) - I^f_0(t) \xrightarrow[t\to \infty]{d} 0$.
  In fact,  by \cite[Lemma 3.4.3]{Durrett2010Probability} we have that for each $t\geq t_0$,
  \begin{align}
    \label{eq:PM:S:1}
    |\mathbb{\widetilde{P}}_{\mu}[e^{i (I^f_1(t) - I^f_0(t) ) }] - 1|
    \leq \sum_{k=0}^{\lfloor t-\ln t \rfloor}\mathbb{\widetilde{P}}_{\mu}\big[|Y_{t,k}|\big],
  \end{align}
  where $ Y_{t,k} := \exp(i \Upsilon_{t-k-1}^{T_{k} \tilde f} - i\theta_t \mathcal I_{t-k-1}^{t-k} X_t(f)) - 1. $
  We claim that there exist $C_2, \delta_2>0$ such that \(\widetilde {\mathbb P}_\mu [|Y_{t,k}|] \leq C_2 e^{-\delta_2 (t-k-1)}\) for all $k\in \mathbb Z_+$ and $t\geq k+1$.
  Then  there exists $C_2'>0$ such that for each  $t \geq t_0$, $|\mathbb{\widetilde{P}}_{\mu}[e^{i (I^f_1(t)- I^f_0(t))}]-1| \leq C_2't^{-\delta_1}$  which, combined with the fact that $I^f_1(t) - I^f_0(t)$ is linear in $f$, completes this step.

 We will show the claim above  in the following substeps 2.1 and 2.2.
  First we choose $\gamma \in (0,\beta)$ close enough to $\beta$ so that there exist $\eta,\eta'>0$ with $ \alpha \tilde \gamma > \eta > \eta - 3\eta' > \alpha \tilde \beta - \alpha \tilde \gamma > 0;$ and define for  $k \in \mathbb Z_+$ and $t\geq k+1$,
  \[
    \mathcal{D}_{t,k}
    :=\{|H_t-H_{t-k-1}|\leq  e^{-\eta (t-k-1)}, H_{t-k-1}> 2e^{-\eta' (t-k-1)}\},
  \]
  where $H_t := e^{-\alpha t}\|X_t\|$.

  \emph{Substep 2.1.} We show that there exist $C_{2.1},\delta_{2.1} >0$ such that for all $k \in \mathbb Z_+$ and $t\geq k+1$, $ \mathbb{\widetilde{P}}_{\mu} \big[ |Y_{t,k}| ;\mathcal{D}^c_{t,k} \big] \leq C_{2.1} e^{-\delta_{2.1} (t-k)}.$
  In fact, it follows from Proposition \ref{lem: control of XT}, Lemma \ref{lem: control of wt} with $|p|=0$ and Chebyshev's inequality that there exist $C_{2.1}', \delta_{2.1}'>0$ such that for all $k \geq 0$ and $t\geq k+1$,
  \begin{align}
    \label{eq: prob of Dtkc}
    & \mathbb{\widetilde{P}}_{\mu}(\mathcal{D}_{t,k}^c)
    \leq \mathbb{\widetilde{P}}_{\mu}(|H_t-H_{t-k-1}| > e^{-\eta (t-k-1)})+\mathbb{\widetilde{P}}_{\mu}(H_{t-k-1}\leq 2e^{-\eta'(t-k-1)}) \\
    & \leq \mathbb{P}_{\mu}(D^c)^{-1}e^{\eta(t-k-1)}\mathbb{P}_{\mu}[|H_t-H_{t-k-1}|] +  \mathbb{P}_{\mu}(D^c)^{-1} \mathbb P_\mu(H_{t-k-1}\leq 2e^{-\eta'(t-k-1)}; D^c) \\
    & \leq \mathbb{P}_{\mu}(D^c)^{-1}  e^{\eta(t-k-1)}\|H_t - H_{t-k-1}\|_{\mathbb P_\mu; 1+\gamma} + \mathbb{P}_{\mu}(D^c)^{-1} \mathbb P_\mu(0<H_{t-k-1}\leq 2e^{-\eta'(t-k-1)}) \\
    & \leq C'_{2.1} e^{-(\alpha \tilde \gamma - \eta)(t-k-1)}+C'_{2.1} e^{-\delta'_{2.1}(t-k-1)}.
  \end{align}
  This implies the desired result in this substep, since $|Y_{t,k}| \leq 2$ a.s..

  \emph{Substep 2.2.} We will show that there exist $C_{2.2},\delta_{2.2} > 0$ such that for all $k\in \mathbb Z_+$ and $t\geq k+1$, it holds that $ \mathbb{\widetilde{P}}_{\mu} [|Y_{t,k}|;\mathcal{D}_{t,k}] \leq  C_{2.2} e^{-\delta_{2.2} (t-k)}$.
  In fact, noticing that for $f\in \mathcal C_s$ and $k\in \mathbb Z_+$, we have $T_kf = e^{\alpha (\tilde \beta - 1 )k}P_k^\alpha f $; and therefore for all $k\in \mathbb Z_+$ and $t \geq k + 1$,
  \begin{align}
    \label{eq:gammafunction11}
    \Upsilon_{t-k-1}^{T_{k} \tilde f}
    = \frac{X_{t-k}(T_{k} \tilde  f) - X_{t -k-1}(P_1^\alpha T_{k} \tilde f)}{\|X_{t-k-1}\|^{1-\tilde \beta}}
    = \frac{\mathcal I_{t - k - 1}^{t - k} X_t(f)}{\|e^{\alpha (k+1)}X_{t-k-1} \|^{1 -\tilde \beta}}.
  \end{align}
  Since $|e^{ix}-e^{iy}|\leq|x-y|$ for all $x,y\in \mathbb R$, we have for all $k \in \mathbb Z_+$ and $t\geq k+1$,
  \begin{align}
    \label{eq: control of Ykt}
    \mathbb{\widetilde{P}}_{\mu}[|Y_{t,k}|;\mathcal{D}_{t,k}]
    & \leq \mathbb{\widetilde{P}}_{\mu}\Big[|\mathcal I_{t-k-1}^{t-k} X_t(f) | \cdot \Big| \| e^{\alpha(k+1)}X_{t-k-1}\| ^{ \tilde \beta - 1} - \|X_t\|^{ \tilde \beta - 1}\Big|; \mathcal D_{t,k}\Big] \\
    & \leq  e^{\alpha(\tilde \beta - 1)t}\mathbb{\widetilde{P}}_{\mu}\big[|\mathcal I_{t-k-1}^{t-k}X_t(f)|\cdot K_{t,k}\big],
  \end{align}
  where
  \[
    K_{t,k}
    := \Big| \frac {H_t^{1- \tilde \beta} - H_{t-k-1}^{1 - \tilde \beta}} {H_t^{1 - \tilde \beta} H_{t-k-1}^{ 1- \tilde \beta }} \Big| \mathbf{1}_{\mathcal{D}_{t,k}}.
  \]
  Note that, since $\eta' < \eta$, we have almost surely on $\mathcal D_{t,k}$,
  \begin{align}
    H_t
    & \geq H_{t-k-1}- e^{-\eta (t-k-1)}
      \geq 2e^{-\eta'(t-k-1)}-e^{-\eta(t-k-1)}
      \geq e^{-\eta'(t-k-1)}.
  \end{align}
  Therefore, for all $k \in \mathbb Z_+$ and $t\geq k+1$, almost surely  on $\mathcal D_{t,k}$,
  \begin{align}
    & \Big|H_t^{1- \tilde \beta}-H_{t-k-1}^{1- \tilde \beta}\Big|
      \leq (1- \tilde \beta) \max \{ H_t^{-\tilde \beta }, H_{t-k-1}^{ -\tilde \beta} \} | H_t - H_{t-k-1} | \\
    & \leq (1- \tilde \beta ) \max\{e^{\eta' (t-k-1)}, \frac{1}{2}e^{\eta'(t-k-1)}\}^{\tilde \beta} e^{-\eta(t-k-1)}  \leq (1- \tilde \beta) e^{-(\eta - \eta') (t-k-1)}
  \end{align}
  and $ |H_t^{1 - \tilde \beta} H_{t-k-1}^{ 1 - \tilde \beta}| \geq 2^{\frac{1}{1+\beta}} e^{-2\eta'(t-k-1)}$.
  Thus, there exists $C_{2.2}'> 0$ such that for all $k \geq 0, t\geq k+1$, almost surely
  \begin{align}
    K_{t,k}
    \leq C_{2.2}' e^{-(\eta - 3\eta')(t-k-1)}.
  \end{align}
  Now, by Lemma \ref{lem: control of mgtrs}, there exists $C''_{2.2}>0$ such that for all $k\geq 0$ and $t\geq k+1$,
  \begin{align}
    \label{eq: Y in D}
    & \mathbb{\widetilde{P}}_{\mu} [|Y_{t,k}| ; \mathcal{D}_{t,k} ]
    \leq C_{2.2}' e^{\alpha (\tilde \beta - 1)t} \mathbb{\widetilde{P}}_{\mu} [ | \mathcal{I}_{t-k-1}^{t-k}X_t(f)| ] e^{-(\eta - 3\eta')(t-k-1)} \\
    & \leq \frac{C_{2.2}' } {\mathbb{P}_{\mu}(D^c)} e^{ \alpha (\tilde \beta - 1)t} \|\mathcal{I}_{t-k-1}^{t-k} X_t(f)\|_{\mathbb P_\mu; 1+\gamma} e^{-(\eta - 3\eta')(t-k - 1)} \\
    & \leq C_{2.2}'' e^{\alpha(\tilde \beta - \tilde \gamma)t} e^{ (\alpha \tilde \gamma - \kappa_f b)k} e^{-(\eta - 3\eta')(t-k)}
     \leq C_{2.2}'' e^{\alpha(\tilde \beta - \tilde \gamma)(t-k)} e^{-(\eta - 3\eta')(t-k)},
  \end{align}
  as desired in this step.
  In the last inequality, we used the fact that $f\in \mathcal C_s$ and therefore $\alpha \tilde \beta < \kappa_f b$.

  \emph{Step 3.}
  We show that $I^f_2(t)\xrightarrow[t\to \infty]{d} 0$.
  First fix a $\gamma \in (0,\beta)$ in this step.
  From the fact that $\kappa_f b -\alpha \tilde \gamma > \alpha (\tilde \beta - \tilde \gamma)$, we can choose $\epsilon >0$ small enough so that $q:=\kappa_fb- \alpha \tilde \gamma  > \alpha (\tilde \beta - \tilde \gamma) + 2\epsilon (1 - \tilde \beta)$.
  Now writing $\mathcal{E}_t:=\{\|X_t\|>e^{(\alpha-\epsilon) t}\}$, according to Proposition \ref{lem: control of XT}, there exist $C_3, \delta_3>0$ such that
  \begin{align}
    \mathbb{\widetilde{P}}_{\mu}(\mathcal{E}^c_t)
    \leq \frac{1}{\mathbb{P}_{\mu}(D^c)}\mathbb{P}_{\mu}(0<e^{-\alpha t}\|X_t\|\leq e^{-\epsilon t})\leq C_3e^{-\delta_3 t}
    , \quad t\geq0.
  \end{align}
  Therefore,
  \begin{align}
    \label{Theorem123}
    |\mathbb{\widetilde{P}}_{\mu}[e^{i I^f_2(t)}-1;\mathcal{E}^c_t]|
    \leq 2\mathbb{\widetilde{P}}_{\mu}(\mathcal{E}^c_t)
    \leq 2C_3e^{-\delta_3 t},
    \quad t\geq t_0.
  \end{align}
	According to Lemma \ref{lem: control of mgtrs}, there exist $C_3',C_3'',C_3'''>0$ such that for each $t\geq t_0 >1$,
  \begin{align}
    & |\mathbb{\widetilde{P}}_{\mu} [ (e^{i I^f_2(t)}-1);\mathcal{E}_t]|
      \leq  \mathbb{\widetilde{P}}_{\mu} [ |I^f_2(t)|;\mathcal{E}_t] \\
    & \leq  ( e^{(\alpha-\epsilon) t} )^{\tilde \beta - 1}\Big(\sum_{k=\lceil t-\ln t \rceil}^{\lfloor t \rfloor - 1}\mathbb{\widetilde{P}}_{\mu} [| \mathcal{I}_{t-k-1}^{t-k} X_t(f) |] + \mathbb{\widetilde{P}}_{\mu}[| \mathcal{I}_{0}^{t-\lfloor t\rfloor} X_t(f)|]\Big) \\
    & \leq ( e^{(\alpha-\epsilon) t} )^{\tilde \beta - 1}\Big(\sum_{k=\lceil t-\ln t \rceil}^{\lfloor t \rfloor - 1}\|\mathcal{I}_{t-k-1}^{t-k} X_t(f) \|_{\mathbb P_\mu; 1+\gamma} + \|\mathcal I_0^{t-\lfloor t \rfloor} X_t(f)\|_{\mathbb P_\mu;1+\gamma}\Big) \\
    & \leq C_3' e^{\alpha (\tilde \beta - \tilde \gamma)t} e ^{\epsilon (1-\tilde \beta) t}\sum_{k=\lceil t-\ln t \rceil}^{\lfloor t \rfloor}  e^{(\alpha\tilde \gamma-\kappa_f b)k}
      \leq C_3' e^{q t}e^{-\epsilon ( 1 - \tilde \beta)t}\sum_{k=\lceil t-\ln t \rceil}^{\lfloor t \rfloor}  e^{-q k}
    \\ & \leq C_3'' e^{q(t - \lceil t - \ln t\rceil)}e^{-\epsilon(1 - \tilde \beta) t}
         \leq C_3'' t^q e^{- \epsilon(1 - \tilde \beta) t}.
  \end{align}
  From this and \eqref{Theorem123}, we get that $\widetilde {\mathbb P}_\mu[e^{i I^f_2(t)}] \xrightarrow[t\to \infty]{} 1$.
  Note that $I^f_2(t)$ is linear in $f$ so we can replace $f$ with $\theta f$ for $\theta \in \mathbb R$ and get the desired result in this step.

	\emph{Step 4.} We will show that $I^f_3(t) \xrightarrow[t\to \infty]{\widetilde {\mathbb P}_\mu \text{-} a.s.} 0$.
  In fact, we have
  \begin{align}
    & |I^f_3(t)|
      \leq \frac{X_0(|P^\alpha_tf|)}{\|X_t\|^{1 - \tilde \beta }}
      \leq \frac{e^{\alpha t - \kappa_f b t}X_0(Qf)}{(e^{\alpha t} H_t)^{1 - \tilde \beta}}
      = e^{(\alpha \tilde \beta - k_fb)t} H_t^{\tilde \beta - 1} X_0(Qf)
      \xrightarrow[t\to \infty]{\widetilde {\mathbb P}_\mu \text{-} a.s.} 0.
  \end{align}

  \emph{Step 5.} Combining Steps 1--4, we complete the proof of Theorem  \ref{thm:M}.\eqref{thm:M:1}.
\end{proof}

\subsection{Central limit theorem for $f \in \mathcal C_c$}
\begin{proof}[Proof of Theorem \ref{thm:M}.(\ref{thm:M:2})]
  Fix $\mu\in \mathcal M_c(\mathbb R^d)$, $f\in \mathcal C_c$ and $t_0 > 1$ large enough so that $ \lceil t - \ln t\rceil \leq \lfloor t \rfloor - 1$ for each $t\geq t_0$.
  For each $t\geq t_0$, in this proof we write $\theta_t = \|t X_t\|^{\tilde \beta - 1}$,  define $I_i^f(t)$ using \eqref{eq:PM:CLTS:1}  for $i = 1,2,3$, and set $ I^f_0(t) := t^{\tilde \beta - 1}\sum_{k=0}^{\lfloor t-\ln t \rfloor} \Upsilon_{t-k-1}^{T_{k} \tilde f}$, where $\tilde f = e^{\alpha(\tilde \beta - 1)} f$.

  \emph{Step 1.} We show that $I^f_0(t) \xrightarrow[t\to \infty]{d} \zeta^f$.
  In fact, for each $t \geq  t_0( > 1)$ we have $t^{\tilde \beta - 1} < 1$; and therefore, for each $k \in \mathbb Z_+$, we have $t^{\tilde \beta - 1} T_{k+1} f \in \mathcal P_f:=\{\theta T_n f: n \in \mathbb Z_+, \theta \in [-1,1]\}$.
  Therefore from Proposition \ref{cor:MI} and that $\tilde \beta - 1 = -\frac{1}{1+\beta}$ we get that there exist $C_1,\delta_1 > 0$ such that
  \begin{align}
    \Big|\mathbb{\widetilde{P}}_{\mu} [e^{i I^f_0(t)} ]-\exp\Big(\frac{1}{t}\sum_{k=0}^{\lfloor t-\ln t \rfloor} \langle Z_1T_{k}\tilde f, \varphi\rangle \Big)\Big|
    \leq C_1 e^{-\delta_1(t - \lfloor t - \ln t\rfloor)}
    \leq \frac{C_1}{t^{\delta_1}},
    \quad t\geq t_0.
  \end{align}
  Since $f \in \mathcal C_c\setminus \{0\}$, we have $T_k \tilde f = \tilde f$ for each $k \in \mathbb Z_+$.
  Similar to the argument in \eqref{eq:PM:CLTS:2} we have
  \begin{align}
    \label{CLT:C:eq:m}
    \lim_{t\to \infty} \frac{1}{t}\sum_{k=0}^{\lfloor t-\ln t \rfloor} \langle Z_1 T_{k}\tilde f, \varphi\rangle
    = \langle Z_1 \tilde f,\varphi \rangle
    = \langle (-if)^{1+\beta}, \varphi \rangle
    = m[f].
  \end{align}
  Therefore $\mathbb {\widetilde P}_\mu[e^{i I^f_0(t)}] \xrightarrow[t\to \infty]{} e^{m[f]}$.
  The desired result in this step follows.

  \emph{Step 2.} We show that $ I^f_1(t) - I^f_0 (t) \xrightarrow[t\to \infty]{d} 0$.
  In fact, similar to Step 2 in the proof of Theorem \ref{thm:M}.(\ref{thm:M:1}), we  have \eqref{eq:PM:S:1} is valid with $ Y_{t,k} := \exp(i t^{\tilde \beta - 1} \Upsilon_{t-k-1}^{T_{k}\tilde f} - i\theta_t \mathcal I_{t-k-1}^{t-k} X_t(f)) - 1$.
  Similarly, we claim that there exist $C_2, \delta_2>0$ such that $\widetilde {\mathbb P}_\mu [|Y_{t,k}|] \leq C_2 e^{-\delta_2 (t-k-1)}$ for all $k\in \mathbb N$ and $t\geq k+1$, and  then the desired result in this step follows.

  We will show the claim above in the following substeps 2.1 and 2.2.
  First we choose $\gamma \in (0,\beta)$ close enough to $\beta$ so that there exist $\eta,\eta'>0$ with $ \alpha \tilde \gamma > \eta > \eta - 3\eta' > \alpha \tilde \beta - \alpha \tilde \gamma > 0$; and define, for  $k \in \mathbb N$ and $t\geq k+1$, $ \mathcal{D}_{t,k} := \{|H_t-H_{t-k-1}| \leq  e^{-\eta (t-k-1)}, H_{t-k-1}> 2e^{-\eta' (t-k-1)}\}$.

  \emph{Substep 2.1.}
  Similar to Substep 2.1 in the proof of Theorem \ref{thm:M}.(\ref{thm:M:1}), there exist $C_{2.1},\delta_{2.1} >0$ such that for all $k \in \mathbb N$ and $t\geq k+1$, $\mathbb{\widetilde{P}}_{\mu}[|Y_{t,k}|;\mathcal{D}^c_{t,k}] \leq C_{2.1} e^{-\delta_{2.1} (t-k)}$.
  We omit the details.

  \emph{Substep 2.2.} We will show that there exist $C_{2.2},\delta_{2.2} > 0$ such that for all $k\in \mathbb N$ and $t\geq k+1$, $ \mathbb{\widetilde{P}}_{\mu}[|Y_{t,k}|;\mathcal{D}_{t,k}] \leq  C_{2.2} e^{-\delta_{2.2} (t-k)}.$
  In fact, noticing that for $f\in \mathcal C_c$ and $k\in \mathbb Z_+$, we have $T_kf = e^{\alpha (\tilde \beta - 1 )k}P_k^\alpha $; and therefore for all $k\in \mathbb Z_+$ and $t \geq k + 1$,
  \[
    t^{\tilde \beta - 1} \Upsilon_{t-k-1}^{T_{k} \tilde f}
    = \frac{X_{t-k}(T_{k} \tilde f) - X_{t -k-1}(P_1^\alpha T_{k} \tilde f)}{\|t X_{t-k-1}\|^{1-\tilde \beta}}
    = \frac{\mathcal I_{t - k - 1}^{t - k} X_t(f)}{\|te^{\alpha (k+1)}X_{t-k-1} \|^{1 -\tilde \beta}}.
  \]
  The rest is similar to Substep 2.2 in the proof of Theorem \ref{thm:M}.(\ref{thm:M:2}).
  We omit the details.

  \emph{Step 3.}
  We show that $ I^f_2(t)\xrightarrow[t\to \infty]{d} 0$.
  In fact, writing $\mathcal{E}_t:=\{\|X_t\|>t^{-1/2}e^{\alpha t}\}$, according to Proposition \ref{lem: control of XT}, there exist $C_3, \delta_3>0$ such that
  \[
    \mathbb{\widetilde{P}}_{\mu}(\mathcal{E}^c_t)
    \leq \frac{1}{\mathbb{P}_{\mu}(D^c)}\mathbb{P}_{\mu}(0<e^{-\alpha t}\|X_t\|\leq t^{-1/2})\leq C_3( t^{-\delta_3}+e^{-\delta_3 t})
    , \quad t\geq0.
  \]
  Therefore,
  \begin{align}
   \label{Theorem123a}
    |\mathbb{\widetilde{P}}_{\mu}[e^{i I^f_2(t)}-1;\mathcal{E}^c_t]|
    \leq 2\mathbb{\widetilde{P}}_{\mu}(\mathcal{E}^c_t)
    \leq C_3(t^{-\delta_3}+e^{-\delta_3 t}),
    \quad t\geq t_0.
  \end{align}
  Choose a $\gamma\in (0,\beta)$ close enough to $\beta$ so that $\alpha(\tilde \beta - \tilde \gamma) \leq \frac{1}{2}(1- \tilde \beta)$.
	According to Lemma \ref{lem: control of mgtrs}, there exist $C_3',C_3'',C_3'''>0$ such that for each $t\geq t_0 (>1)$,
  \begin{align}
    & |\mathbb{\widetilde{P}}_{\mu} [ (e^{i I^f_2(t)}-1)\mathbf{1}_{\mathcal{E}_t}]|
      \leq  \mathbb{\widetilde{P}}_{\mu} [ |I^f_2(t)|\mathbf{1}_{\mathcal{E}_t}] \\
    & \leq  (t^{\frac{1}{2}} e^{\alpha t} )^{\tilde \beta - 1}\Big(\sum_{k=\lceil t-\ln t \rceil}^{\lfloor t \rfloor - 1}\mathbb{\widetilde{P}}_{\mu} [| \mathcal{I}_{t-k-1}^{t-k} X_t(f) |] + \mathbb{\widetilde{P}}_{\mu}[| \mathcal{I}_{0}^{t-\lfloor t\rfloor} X_t(f)|]\Big) \\
    & \leq C_3' t^{\frac{1}{2}(\tilde \beta - 1)} e^{\alpha(\tilde \beta - 1)t}\Big(\sum_{k=\lceil t-\ln t \rceil}^{\lfloor t \rfloor - 1}\|\mathcal{I}_{t-k-1}^{t-k} X_t(f) \|_{\mathbb P_\mu; 1+\gamma} + \|\mathcal I_0^{t-\lfloor t \rfloor} X_t(f)\|_{\mathbb P_\mu;1+\gamma}\Big) \\
    & \leq C_3' t^{\frac{1}{2}(\tilde \beta - 1)} e^{\alpha (\tilde \beta - \tilde \gamma)t}\sum_{k=\lceil t-\ln t \rceil}^{\lfloor t \rfloor}  e^{(\alpha\tilde \gamma-\kappa_f b)k}
      = C_3' t^{\frac{1}{2}(\tilde \beta - 1)} e^{\alpha(\tilde \beta - \tilde \gamma) t}\sum_{k=\lceil t-\ln t \rceil}^{\lfloor t \rfloor}  e^{-\alpha (\tilde \beta -\tilde \gamma) k}
    \\ & \leq C_3'' t^{\frac{1}{2}(\tilde \beta - 1)} e^{\alpha (\tilde \beta - \tilde \gamma)(t - \lceil t - \ln t\rceil)}
         \leq C_3'' t^{\frac{1}{2}(\tilde \beta - 1)} t^{\alpha (\tilde \beta - \tilde \gamma)}.
  \end{align}
  From this and \eqref{Theorem123a}, we get the desired result in this step.

  \emph{Step 4.} Similar to Step 4 in the proof of Theorem \ref{thm:M}.(\ref{thm:M:1}), we can verify that $I_3(t) \xrightarrow[t\to \infty]{\widetilde {\mathbb P}_\mu \text{-} a.s.} 0$.
  We omit the details.

 	\emph{Step 5.} Combining Steps 1--4, we complete the proof of Theorem \ref{thm:M}.\eqref{thm:M:2}.
\end{proof}

\subsection{Central limit theorem for $f\in \mathcal C_l$}
\label{sec: large rate clt}
\begin{proof}[Proof of Theorem \ref{thm:M}.(\ref{thm:M:3})]
  Fix $\mu \in \mathcal M_c(\mathbb R^d)$ and $f \in \mathcal C_l$.
  Define $\mathcal N:= \{p\in \mathbb Z_+^d: \alpha \tilde \beta > |p|b\}$.
  In this proof we write for each $t\geq 0$,
  \begin{align}
    & \frac{X_t(f) - \sum_{p\in \mathbb Z_+^d: \alpha \tilde \beta \geq |p|b} \langle f,\phi_p\rangle_\varphi e^{(\alpha - |p|b)t}H_\infty^p}{\|X_t\|^{1- \tilde \beta}}
      = \sum_{p\in \mathcal N}\frac{ \langle f,\phi_p\rangle_\varphi [X_t(\phi_p) - e^{(\alpha - |p|b)t}H_\infty^p]}{\|X_t\|^{1- \tilde \beta}}
    \\& = \sum_{p \in \mathcal N} \frac{\langle f,\phi_p\rangle_\varphi e^{(\alpha - |p|b)t}(H_t^p - H_\infty^p)}{\|X_t\|^{1- \tilde \beta}}
    = \sum_{k=0}^\infty \sum_{p \in \mathcal N}  \langle f,\phi_p\rangle_\varphi e^{(\alpha - |p|b)t}\frac{ H_{t+k}^p - H_{t+k+1}^p}{\|X_t\|^{1- \tilde \beta}}
    \\ & =: \sum_{k=0}^\infty \widetilde \Upsilon_{t,k}
         = \Big(\sum_{k = 0}^{\lfloor t^2 \rfloor}  \widetilde \Upsilon_{t,k} \Big) + \Big(\sum_{k = \lceil t^2 \rceil}^\infty  \widetilde \Upsilon_{t,k}\Big)
         = : I^f_1(t) + I^f_2(t),
  \end{align}
  and $I^f_0(t):= \sum_{k = 0}^{\lfloor t^2 \rfloor} \Upsilon_{t+k}^{- T_k \tilde f}$ where $\tilde f := \sum_{p\in \mathcal N} e^{-(\alpha - |p|b)}\langle f, \phi_p \rangle_\varphi \phi_p$.

  \emph{Step 1.} We show that $I^f_0(t) \xrightarrow [t\to \infty]{d} \zeta^{-f}$.
  In fact, since $T_k\tilde f \in \mathcal P_{\tilde f}$ for each $k\in \mathbb Z_+$, from Corollary \ref{cor:MI} we have $\widetilde{ \mathbb P}_\mu[e^{i I_0^f(t)}]\xrightarrow[t\to \infty]{}\exp\{\sum_{k=0}^\infty \langle Z_1T_k(-\tilde f),\varphi\rangle\}$.
  Using \eqref{eq:PL:S:1} and the fact that $\varphi(x)dx$ is the invariant probability of the semigroup $(P_t)_{t\geq 0}$ we have
  \begin{align}
    \label{eq:PM:CLTS:2a}
    & \sum_{k=0}^\infty \langle Z_1 T_{k} (-\tilde f), \varphi \rangle
      = \sum_{k=0}^\infty \int_0^1 \langle P_u^\alpha ((iP_{1 - u}^\alpha T_k \tilde f)^{1+\beta}), \varphi\rangle ~du
    \\& = \sum_{k=0}^\infty \int_0^1 e^{\alpha u} \langle  (iP_{1 - u}^\alpha T_{k}\tilde f)^{1+\beta}, \varphi \rangle ~du
    \\& = \sum_{k=0}^\infty \int_0^1 \langle  (iT_{k+ u} f)^{1+\beta}, \varphi\rangle~du
    = \int_0^\infty \langle  (iT_{u} f)^{1+\beta}, \varphi\rangle~du = m[-f].
  \end{align}
  The result in this step follows.

  \emph{Step 2.} We show that $I^f_1(t) - I^f_0(t) \xrightarrow[t\to \infty]{d} 0$.
  In fact, by \cite[Lemma 3.4.3]{Durrett2010Probability} we have, for each $t\geq 0$, that $|\widetilde {\mathbb P}_{\mu}[e^{i(I_{1}^{f}(t) - I_0^f(t))} - 1]| \leq \sum_{k=0}^{\lfloor t^2 \rfloor} \widetilde {\mathbb {P}}_\mu[|Y_{t,k}|]$ where $Y_{t,k} := e^{i(\widetilde {\Upsilon}_{t,k} - \Upsilon_{t+k}^{-T_{k}\widetilde {f}})} - 1. $
  We claim that there exist $C_2, \delta_2>0$ such that $\widetilde {\mathbb {P}}_\mu[|Y_{t,k}|] \leq C_2e^{-\delta_2 t}$ for all $t\geq 0$ and $k \in \mathbb Z_+$.
  Then  $|\widetilde {\mathbb P}_{\mu}[e^{i(I_{1}^{f}(t) - I_0^f(t))} - 1]| \leq (t^2+1)C_2e^{-\delta_2 t}$ which completes this step.

  We will show the claim above in the following substeps 2.1 and 2.2.
  First we choose $\gamma\in(0, \beta)$ close enough to $\beta$ so that $\alpha \tilde \gamma > |p|b$ for each $p\in \mathcal N$;
  and even closer so that there exist $\eta,\eta'>0$ satisfying $\alpha \tilde \gamma > \eta>\eta - 3\eta'> \alpha (\tilde \beta - \tilde \gamma)>0$. We also define $\mathcal{D}_{t,k} :=\{|H_t-H_{t+k}|\leq  e^{-\eta t}, H_{t}> 2e^{-\eta' t}\}$.

  \emph{Substep 2.1.} Similar to Substep 2.1 in the proof of Theorem \ref{thm:M}.(\ref{thm:M:1}), we have that there exist $C_{2.1},\delta_{2.1} >0$ such that for all $k \in \mathbb Z_+$ and $t\geq 0$, $ \mathbb{\widetilde{P}}_{\mu}[|Y_{t,k}|;\mathcal{D}^c_{t,k}] \leq C_{2.1} e^{-\delta_{2.1} t}$.
  We omit the details.

  \emph{Substep 2.2.} We show that there exist $C_{2.2}, \delta_{2.2}>0$ such that for all $k \in \mathbb Z_+$ and $t\geq 0$, we have $\widetilde {\mathbb {P}}_\mu[|Y_{t,k}|; \mathcal D_{t,k}]\leq C_{2.2}e^{- \delta_{2,2} t}$.
  In fact, it can be verified that for all $k \in \mathbb Z_+$ and $t\geq 0$,
  \begin{align}
    & \Upsilon_{t+k}^{-T_k\tilde f}
      = \frac{X_{t+k}(P^\alpha_1T_k\tilde f) - X_{t+k+1}(T_k \tilde f)}{\|X_{t+k}\|^{1 - \tilde \beta}}
    \\& = \sum_{p\in \mathcal N}
    \langle\tilde f,\phi_p\rangle_\varphi e^{-(\alpha \tilde \beta - |pb|)k}\frac{ X_{t+k}(P_1^\alpha \phi_p) - X_{t+k+1}(\phi_p)}{\|X_{t+k}\|^{1 - \tilde \beta}}
    \\& = \sum_{p\in \mathcal N}
    \langle f,\phi_p\rangle_\varphi  e^{(\alpha  -|p|b)t}\frac{H_{t+k}^p-H_{t+k+1}^p }{\|e^{-\alpha k}X_{t+k}\|^{1 - \tilde \beta}}.
  \end{align}
  Therefore for all $k\in \mathbb Z_+$ and $t\geq 0$,
  \begin{align}
    &|Y_{t,k}| \mathbf 1_{\mathcal D_{t,k}}
      \leq \Big( \sum_{p\in \mathcal N}|\langle f,\phi_p\rangle_\varphi|  e^{(\alpha  -|p|b)t} | H_{t+k}^p-H_{t+k+1}^p |\Big) \Big( \frac{1}{\|X_t\|^{1 - \tilde \beta}} - \frac{1}{\|e^{-\alpha k}X_{t+k}\|^{1 - \tilde \beta}} \Big)\mathbf 1_{\mathcal D_{t,k}}.
    \\ &= \Big( \sum_{p\in \mathcal N}|\langle f,\phi_p\rangle_\varphi|  e^{(\alpha  -|p|b)t} | H_{t+k}^p-H_{t+k+1}^p |\Big)e^{\alpha (\tilde \beta - 1)t} K_{t,k}
    \\ &= \Big( \sum_{p\in \mathcal N}|\langle f,\phi_p\rangle_\varphi|  e^{(\alpha \tilde \beta  -|p|b)t} | H_{t+k}^p-H_{t+k+1}^p |\Big) K_{t,k},
  \end{align}
  where
  \[
    K_{t,k}
    := \Big| \frac {H_t^{1- \tilde \beta} - H_{t+k}^{1 - \tilde \beta}} {H_t^{1 - \tilde \beta} H_{t+k}^{ 1- \tilde \beta }} \Big| \mathbf{1}_{\mathcal{D}_{t,k}}.
  \]
  Similar to Substep 2.2 in the proof of Theorem \ref{thm:M}.(\ref{thm:M:1}), we can verify that for all $k\in \mathbb Z_+$ and $t\geq 0$, almost surely $K_{t,k} \leq C_{2.2}'' e^{- (\eta - 3\eta')t}$.
  From this and Lemma \ref{lem: control of wt} we know that there exists $C'''_{2.2}$ such that for all $k\in \mathbb Z_+$ and $t\geq 0$,
  \begin{align}
    & \widetilde{\mathbb P}_\mu[|Y_{t,k}|; \mathcal D_{t,k}]
      \leq \mathbb P_\mu(D)^{-1}\mathbb P_\mu[ |Y_{t,k}| ;\mathcal D_{t,k} ]
    \\ & \leq \mathbb P_{\mu}(D)^{-1} C_{2.2}'' e^{- (\eta - 3\eta') t}\sum_{p\in \mathcal {N}} |\langle f,\phi_p\rangle_\varphi|  e^{(\alpha \tilde \beta  -|p|b)t} \mathbb P_\mu[| H_{t+k}^p-H_{t+k+1}^p |]
    \\ & \leq \mathbb P_{\mu}(D)^{-1} C_{2.2}'' e^{- (\eta - 3\eta') t}\sum_{p\in \mathcal {N}} |\langle f,\phi_p\rangle_\varphi|  e^{(\alpha \tilde \beta  -|p|b)t} \| H_{t+k}^p-H_{t+k+1}^p \|_{\mathbb P_\mu; 1+\gamma}
    \\&\leq  \mathbb P_{\mu}(D)^{-1} C_{2.2}'' e^{- (\eta - 3\eta') t}\sum_{p\in \mathcal N} |\langle f,\phi_p\rangle_\varphi|  e^{(\alpha \tilde \beta  -|p|b)t} e^{-(\alpha \tilde \gamma - |p|b)(t+k)} \\
    &  \leq  C_{2.2}''' e^{- (\eta - 3\eta') t} e^{(\alpha \tilde \beta - \alpha \tilde \gamma)t},
  \end{align}
  as desired in this substep.

  \emph{Step 3.} We show that $I^f_2(t) \xrightarrow[t\to \infty]{d} 0$.
  In order to do this, choose an $\epsilon \in (0,\alpha)$ and a $\gamma \in (0,\beta)$ close enough to $\beta$ so that for each $p\in \mathcal N$, it holds that $\alpha \tilde \gamma > |p|b$.
  Define $\mathcal E_t:= \{\|X_t\| > e^{(\alpha - \epsilon)t}\}$.
  According to Proposition \ref{lem: control of XT}, there exist $C_3, \delta_3 > 0$ such that for each $t\geq 0$, $|\widetilde {\mathbb {P}}_\mu[e^{i I_2^f(t)} - 1; \mathcal E_t^c]|\leq 2\widetilde {\mathbb {P}}_\mu(\mathcal E_t^c) \leq C_3 e^{- \delta_3 t}$.
  On the other hand, according to Lemma \ref{lem: control of wt}, we know that there exist $C_3',C_3''>0$ and $\delta_3'>0$ such that
  \begin{align}
    & |\widetilde {\mathbb {P}}_\mu[e^{i I_{2}^{f}(t)} - 1; \mathcal {E}_t]|
      \leq  \widetilde {\mathbb {P}}_\mu[ | I_{2}^{f}(t)|; \mathcal {E}_t]
      \leq \sum_{k = \lceil t^2\rceil}^\infty \widetilde {\mathbb {P}}_\mu[ |\widetilde {\Upsilon}_{t,k}|; \mathcal {E}_t]
    \\ & \leq \mathbb P_\mu(D^c)^{-1} \sum_{k = \lceil t^2\rceil}^\infty \sum_{p \in \mathcal N} |\langle f,\phi_p\rangle_\varphi| e^{(\alpha - |p|b)t}\mathbb {P}_\mu\Big[\frac{ |H_{t+k}^p - H_{t+k+1}^p|}{\|X_t\|^{1- \tilde \beta}}; \mathcal E_t\Big]
    \\ & \leq \mathbb P_\mu(D^c)^{-1} e^{(\alpha - \epsilon) (\tilde \beta - 1) t} \sum_{k = \lceil t^2\rceil}^\infty \sum_{p \in \mathcal N} |\langle f,\phi_p\rangle_\varphi| e^{(\alpha - |p|b)t}\|H_{t+k}^p - H_{t+k+1}^p\|_{\mathbb P_\mu; 1+\gamma}
    \\ & \leq C_3' e^{(\alpha - \epsilon) (\tilde \beta - 1) t} \sum_{k = \lceil t^2\rceil}^\infty \sum_{p \in \mathcal N} |\langle f,\phi_p\rangle_\varphi| e^{(\alpha - |p|b)t} e^{- (\alpha \tilde \gamma - |p|b)(t+k)}
    \\ & = C_3'' e^{ \alpha (\tilde \beta - \tilde \gamma) t } e^{ \epsilon (1 - \tilde \beta) t}e^{- \delta'_3 t^2}.
  \end{align}
  To sum up we have that $\widetilde {\mathbb P}_\mu[e^{iI_2^f(t)}] \xrightarrow[t\to \infty]{} 1$, which completes this step.

  \emph{Step 4.} Combining Steps 1--3, we complete the proof of Theorem \ref{thm:M}.\eqref{thm:M:3}.
\end{proof}

\appendix
\section{ }
\subsection{Analytic facts}
In this subsection, we collect some useful analytic facts.
\begin{lem}
  \label{lem: estimate of exponential remaining}
  For $z\in \mathbb C_+$,  we have
  \begin{align}
    \label{eq: estimate of exponential remaining}
    \Big|e^{-z} - \sum_{k=0}^n \frac{(-z)^k}{k!} \Big|
    \leq \frac{|z|^{n+1}}{(n+1)!} \wedge \frac{2|z|^{n}}{n!}, \quad n\in \mathbb Z_+.
  \end{align}
\end{lem}
\begin{proof}
  Notice that $|e^{-z}| = e^{- \operatorname{Re} z} \leq 1$.
  Therefore, $ |e^{-z} - 1| = \Big| \int_0^1 e^{-\theta z} z d\theta\Big| \leq |z|. $
  Also, notice that $|e^{-z} - 1| \leq |e^{-z}|+1 \leq 2$.
  Thus \eqref{eq: estimate of exponential remaining} is true when $n = 0$.
  Now, suppose that \eqref{eq: estimate of exponential remaining} is true when $n = m$ for some $m \in \mathbb Z_+$.
  Then
  \begin{align}
    &\Big|e^{-z} - \sum_{k=0}^{m+1} \frac{(-z)^k}{k!}\Big|
      = \Big| \int_0^1\Big(e^{-\theta z} - \sum_{k=0}^m \frac{(-\theta z)^k}{k!} \Big) z d\theta \Big| \\
    & \quad \leq  \Big(\int_0^1 \frac{|\theta z|^{m+1}}{(m+1)!} |z| d\theta\Big) \wedge \Big(\int_0^1 \frac{2|\theta z|^{m}}{m!} |z| d\theta\Big)
      = \frac{|z|^{m+2}}{(m+2)!} \wedge \frac{2|z|^{m+1}}{(m+1)!},
  \end{align}
  which says that \eqref{eq: estimate of exponential remaining} is true for $n = m + 1$.
\end{proof}

\begin{lem}
  \label{lem: extension lemma for branching mechanism}
  Suppose that  $\pi$ is a measure on $(0,\infty)$ with $\int_{(0,\infty)} (y \wedge y^2) \pi(dy)< \infty$.
  Then the functions
  \begin{align}
    & h (z)
      = \int_{(0,\infty)} (e^{-zy} - 1 + zy) \pi(dy), \quad z \in \mathbb C_+, \\
    & h'(z)
      = \int_{(0,\infty)}(1- e^{-zy})y \pi(dy), \quad z \in \mathbb C_+
  \end{align}
  are well defined, continuous on $\mathbb C_+$ and holomorphic on $\mathbb C_+^0$.
  Moreover,
  \[
    \frac{h(z)-h(z_0)}{z-z_0}
    \xrightarrow[\mathbb C_+\ni z \to z_0]{} h'(z_0),\quad z_0 \in \mathbb C_+.
  \]
\end{lem}
\begin{proof}
  It follows from Lemma \ref{lem: estimate of exponential remaining} that $h$ and $h'$ are well defined on $\mathbb C_+$.
  According to \cite[Theorems 3.2. \& Proposition 3.6]{SchillingSongVondravcek2010Bernstein}, $h'$ is continuous on $\mathbb C_+$ and holomorphic on $\mathbb C_+^0$.

  It follows from Lemma \ref{lem: estimate of exponential remaining} that, for each $z_0 \in \mathbb C_+$,  there exists $C>0$ such that for $z \in \mathbb C_+$ close enough to $z_0$ and any $y>0$,
  \begin{align}
    & \Big| \frac{e^{-zy} - e^{-z_0 y}+(z-z_0) y}{z-z_0} \Big|
      = \frac{1}{|z-z_0|}\Big| \int_0^1 (-y e^{-(\theta z+(1-\theta)z_0)y}+y)(z-z_0)d\theta\Big| \\
    & \leq y\int_0^1 |1-e^{-(\theta z +(1-\theta)z_0)y}| d\theta
      \leq (2y) \wedge\Big( y^2\int_0^1|\theta z+(1-\theta)z_0|d\theta\Big)
      \leq C(y\wedge y^2).
  \end{align}
  Using this and the dominated convergence theorem, we have
  \begin{align}
    & \frac{h(z)-h(z_0)}{z-z_0} = \int_{(0,\infty)} \frac{e^{-zy}+zy -(e^{-z_0 y}+z_0 y)}{z-z_0}  \pi(dy) \\
    & \xrightarrow[\mathbb C_+\ni z\to z_0]{} \int_{(0,\infty)}(1 - e^{-z_0 y} )y\pi(dy) = h'(z_0),
  \end{align}
  which says that $h$ is continuous on $\mathbb C_+$ and holomorphic on $\mathbb C_+^0$.
\end{proof}

For each $z\in \mathbb C\setminus (-\infty,0]$, we define $ \log z := \log |z| + i \arg z$ where $\arg z \in (-\pi,\pi)$ is uniquely determined by $ z = |z|e^{i \arg z}$. 	
For all $z\in \mathbb C\setminus (-\infty,0]$ and $\gamma \in \mathbb C$, we define $ z^\gamma := e^{\gamma \log z}. $
Then it is known, see \cite[Theorem 6.1]{SteinShakarchi2003Complex} for example, that $z\mapsto \log z$ is holomorphic in $\mathbb C\setminus (-\infty,0]$.
Therefore, for each $\gamma \in \mathbb C$, $z\mapsto z^\gamma$ is holomorphic in $\mathbb C\setminus (-\infty,0]$. (We use the convention that  $0^\gamma := \mathbf 1_{\gamma = 0}$.)
Using the definition above we can easily show that $(z_1z_0)^\gamma = z_1^\gamma z_0^\gamma$ provided $\arg (z_1z_0)=\arg (z_1) + \arg(z_0)$.

It is known, see, for instance, \cite[Theorem 6.1.3]{SteinShakarchi2003Complex} and the remark following it, that the Gamma function $\Gamma$ has an unique analytic extension in $\mathbb C\setminus\{0, -1,-2,\dots\}$ and that
\[
	\Gamma(z+1)
  = z \Gamma(z),\quad z\in \mathbb C\setminus\{0, -1,-2,\dots\}.
\]
Using this recursively, one gets that
\begin{align}
  \label{eq: definition of Gamma function}
  \Gamma(x)
  := \int_0^\infty t^{x-1} \Big(e^{-t} - \sum_{k=0}^{n-1} \frac{(-t)^k}{k!}\Big) dt,
  \quad -n< x< -n+1, n\in \mathbb N.
\end{align}

Fix a $\beta \in (0,1)$.
Using the uniqueness of holomorphic extension and Lemma \ref{lem: extension lemma for branching mechanism}, we get that
\begin{align}
  z^{\beta}
	= \int_0^\infty (e^{-zy}-1) \frac{dy}{\Gamma(-\beta)y^{1+\beta}},
  \quad z\in \mathbb C_+,
\end{align}
and similarly that
\begin{align}
  \label{eq: stable branching on C+}
  z^{1+\beta}
  = \int_0^\infty (e^{-zy}-1+zy)\frac{dy}{\Gamma(-1-\beta)y^{2+\beta}},
  \quad z\in \mathbb C_+.
\end{align}
Lemma \ref{lem: extension lemma for branching mechanism} also says that the derivative of $z^{1+\beta}$ is $(1+\beta)z^{\beta}$ on $\mathbb C^0_+$.
\begin{lem}
  \label{lem: Lip of power function}
  For all $z_0,z_1 \in \mathbb C_+$, we have
\begin{align}
  \label{eq: Lip of power function}
  |z_0^{1+\beta} - z_1^{1+\beta}|
  \leq (1+\beta)(|z_0|^{\beta}+|z_1|^{\beta})|z_0 - z_1|.
\end{align}
\end{lem}
\begin{proof}
  Since $z^{1+\beta}$ is continuous on $\mathbb C_+$, we only need to prove the lemma assuming $z_0,z_1 \in \mathbb C^0_+$.
  Notice that
  \begin{align}
    \label{eq: upper bound for beta power of z}
    |z^\beta|
    = |e^{\beta \log |z| +i\beta \operatorname {arg}z}| = e^{\beta \log |z|} = |z|^\beta,
    \quad z \in \mathbb C\setminus (-\infty, 0].
  \end{align}
  Define a path $\gamma: [0,1] \to \mathbb C^0_+$ such that
  \[
    \gamma(\theta)
    = z_0 (1-\theta) + \theta z_1,
    \quad \theta \in [0,1].
  \]
  Then, we have
  \begin{align}
    |z_0^{1+\beta} - z_1^{1+\beta}|
    & \leq (1+\beta) \int_0^1 |\gamma(\theta)^{\beta}|\cdot |\gamma'(\theta)|d\theta
      \leq (1+\beta)  \sup_{\theta \in [0,1]} |\gamma(\theta)|^{\beta} \cdot |z_1-z_0| \\
    & \leq (1+\beta)  ( |z_1|^{\beta}+|z_0|^{\beta} ) |z_1-z_0|.
      \qedhere
  \end{align}
\end{proof}

Suppose that $\varphi(\theta)$ is a continuous function from $\mathbb R$ into $\mathbb C$ such that $\varphi(0) = 1$ and $\varphi(\theta) \neq 0$ for all $\theta \in \mathbb R$.
Then according to \cite[Lemma 7.6]{Sato2013Levy}, there is a unique continuous function $f(\theta)$ from $\mathbb R$ into $\mathbb C$ such that $f(0) = 0$ and $e^{f(\theta)} = \varphi(\theta)$.
Such a function $f$ is called the distinguished logarithm of the function $\varphi$ and is denoted as $\operatorname{Log} \varphi(\theta)$.
In particular, when $\varphi$ is the characteristic function of an infinitely divisible random variable $Y$,  $\operatorname{Log} \varphi(\theta)$ is called the L\'evy exponent of $Y$.
This distinguished logarithm should not be confused with the $\log$ function defined on $\mathbb C\setminus (-\infty, 0]$.
See the paragraph immediately after \cite[Lemma 7.6]{Sato2013Levy}.

\subsection{Feynman-Kac formula with complex valued functions}
\label{seq: complex Feynman-Kac transform}
In this subsection we give a version of the Feynman-Kac formula with complex valued functions.
Suppose that $\{(\xi_t)_{t \in [r,\infty)}; (\Pi_{r,x})_{r\in [0,\infty), x\in E}\}$ is a (possibly non-homogeneous) Hunt process in a locally compact separable metric space $E$.
We write
\begin{align}
  H^{(h)}_{(s,t)}
  := \exp\Big\{\int_s^t h(u,\xi_u) du\Big\},
  \quad 0 \leq s \leq t, h \in \mathcal B_b([0,t] \times E,\mathbb C).
\end{align}

\begin{lem}
  \label{eq: complex FK}
  Let $t \geq 0$. Suppose that $\rho_1, \rho_2\in \mathcal B_b([0,t] \times E, \mathbb C)$ and $f\in \mathcal B_b(E, \mathbb C)$.
  Then
  \begin{align}
    \label{eq: expresion of g}
    g(r,x)
    := \Pi_{r,x}[ H_{(r,t)}^{(\rho_1+\rho_2)} f(\xi_t)],\quad r \in [0,t], x\in E,
  \end{align}
  is the unique locally bounded solution to the equation
  \[
    g(r,x)
    = \Pi_{r,x} [ H_{(r,t)}^{(\rho_1)} f(\xi_t)]+\Pi_{r,x} \Big[ \int_r^tH_{(r,s)}^{(\rho_1)}\rho_2(s,\xi_s) g(s,\xi_s)~ds \Big],\quad r \in [0,t], x\in E.
  \]
\end{lem}

\begin{proof}
  The proof is similar to that of \cite[Lemma A.1.5]{Dynkin1993Superprocesses}. We include it here for the sake of completeness.
  We first verify that \eqref{eq: expresion of g} is a solution.
  Notice that
  \begin{align}
    \Pi_{r,x} \Big[ \int_r^t | H_{(r,t)}^{(\rho_1)}\rho_2(s,\xi_s) H_{(s,t)}^{(\rho_2)} f(\xi_t)| ~ds \Big]
    \leq  \int_r^t e^{(t-r)\|\rho_1\|_\infty}e^{(t-s)\|\rho_2\|_\infty}\|\rho_2\|_\infty\|f\|_\infty ~ds
    < \infty.
  \end{align}
  Also notice that
  \begin{align}
    \label{eq: crucial for Feynman-Kac}
    \frac{\partial}{\partial s} H^{(\rho_2)}_{(s,t)}= -H^{(\rho_2)}_{(s,t)}\rho_2(s,\xi_s),
    \quad s\in (0,t).
  \end{align}
  Therefore, from the Markov property of $\xi$ and Fubini's theorem we get that
  \begin{align}
    & \Pi_{r,x} \Big[ \int_r^tH_{(r,s)}^{(\rho_1)}~(\rho_2 g)(s,\xi_s)~ds \Big]
      =\Pi_{r,x} \Big[ \int_r^t H_{(r,s)}^{(\rho_1)}\rho_2(s,\xi_s) \Pi_{s,\xi_s}[ H_{(s,t)}^{(\rho_1+\rho_2)} f(\xi_t)]~ds \Big] \\
    & = \Pi_{r,x} \Big[ \int_r^t H_{(r,t)}^{(\rho_1)}\rho_2(s,\xi_s) H_{(s,t)}^{(\rho_2)} f(\xi_t) ~ds \Big]
      = \Pi_{r,x} [ H_{(r,t)}^{(\rho_1)}f(\xi_t)(H_{(r,t)}^{(\rho_2)} - 1)] \\
    & = g(r,x) - \Pi_{r,x} [ H_{(r,t)}^{(\rho_2)} f(\xi_t)].
  \end{align}
  For uniqueness, suppose  $\widetilde g$ is another solution. Put $h(r) = \sup_{x\in E}|g(r,x) - \widetilde g(r,x)|$.
  Then
  \[
    h(r)
    \leq e^{t\|\rho_1\|_\infty}\|\rho_2\|_\infty \int_r^t h(s)ds,
    \quad r\le t.
  \]
  Applying Gronwall's inequality, we get that $h(r) =  0$ for $r\in [0,t]$.
\end{proof}

\subsection{Superprocesses}
\label{sec: definition of superprocess}
In this subsection, we will give the definition of a general superprocess.
Let $E$ be locally compact separable metric space. Denote by $\mathcal M(E)$ the collection of all the finite measures on $E$ equipped with the topology of weak convergence.
For each function $F(x,z)$ on $E\times \mathbb R_+$ and each $\mathbb R_+$-valued function $f$ on $E$,
we use the  convention:
$
  F(x,f)
  := F(x,f(x)),
  x\in E.
$
A process $X=\{(X_t)_{t\geq 0}; (\mathbf P_\mu)_{\mu \in \mathcal M(E)}\}$ is said to be a $(\xi,\psi)$-superprocess if
\begin{itemize}
\item
  the spatial motion $\xi=\{(\xi_t)_{t\geq 0};(\Pi_x)_{x\in E}\}$ is an $E$-valued Hunt process with its lifetime denoted by $\zeta$;
\item
  the branching mechanism $\psi: E\times[0,\infty) \to \mathbb R$ is given by
\begin{align}
  \label{eq: branching mechanism}
  \psi(x,z)=
  -\rho_1(x) z + \rho_2 (x) z^2 + \int_{(0,\infty)} (e^{-zy} - 1 + zy) \pi(x,dy).
\end{align}
where $\rho_1 \in \mathcal B_b(E)$, $\rho_2 \in \mathcal B_b(E, \mathbb R_+)$ and $\pi(x,dy)$ is a kernel from $E$ to $(0,\infty)$ such that $\sup_{x\in E} \int_{(0,\infty)} (y\wedge y^2) \pi(x,dy) < \infty$;
\item
  $X=\{(X_t)_{t\geq 0}; (\mathbf P_\mu)_{\mu \in \mathcal M(E)}\}$ is an $\mathcal M(E)$-valued Hunt process with transition probability determined by
  \begin{align}
    \mathbf P_\mu [e^{-X_t(f)}] = e^{-\mu(V_tf)},
    \quad t\geq 0, \mu \in \mathcal M(E), f\in \mathcal B^+_b(E),
  \end{align}
  where for each $f\in \mathcal B_b(E)$, the function $(t,x)\mapsto V_tf(x)$ on $[0,\infty) \times E$ is the unique locally bounded non-negative solution to the equation
  \begin{align}
    \label{eq:FKPP_in_definition}
    V_tf(x) + \Pi_x \Big[  \int_0^{t\wedge \zeta} \psi(\xi_s,V_{t-s}f)ds \Big]
    = \Pi_x [ f(\xi_t)\mathbf 1_{t<\zeta} ],
    \quad t \geq 0, x \in E.
  \end{align}
\end{itemize}
We refer our readers to \cite{Li2011Measure-valued} for more discussions about the definition and the existence of superprocesses.
To avoid triviality, we assume that $\psi(x,z)$ is not identically equal to $-\rho_1(x)z$.

Notice that the branching mechanism $\psi$ can be extended into a map from $E \times \mathbb C_+$ to $\mathbb C$ using \eqref{eq: branching mechanism}.
Define
\begin{align}
  \psi'(x,z)
  := - \rho_1(x) + 2\rho_2(x) z + \int_{(0,\infty)} (1-e^{-zy})y\pi(x,dy),
  \quad x\in E, z\in \mathbb C_+.
\end{align}
Then according to Lemma \ref{lem: extension lemma for branching mechanism}, for each $x \in E$, $z \mapsto \psi(x,z)$ is a holomorphic function on $\mathbb C_+^0$ with derivative $z \mapsto \psi'(x,z)$.
Define $\psi_0(x,z) := \psi(x,z)+ \rho_1(x)z $ and $\psi'_0(x,z) := \psi'(x,z) + \rho_1(x)$.

Denote by $\mathbb W$ the space of $\mathcal M(E)$-valued c\`{a}dl\`{a}g paths with its canonical path denoted by $(W_t)_{t\geq 0}$.
We say $X$ is \emph{non-persistent} if $\mathbf P_{\delta_x}(\|X_t\|= 0) > 0$ for all $x\in E$ and $t> 0$.
Suppose that $(X_t)_{t\geq 0}$ is non-persistent, then according to \cite[Section 8.4]{Li2011Measure-valued}, there is a unique family of measures $(\mathbb N_x)_{x\in E}$ on $\mathbb W$ such that
(i) $\mathbb N_x (\forall t > 0, \|W_t\|=0) =0$;
(ii) $\mathbb N_x(\|W_0 \|\neq 0) = 0$;
and (iii) if $\mathcal N$ is a Poisson random measure defined on some probability space with intensity $\mathbb N_\mu(\cdot):= \int_E \mathbb N_x(\cdot )\mu(dx)$, then the superprocess $\{X;\mathbf P_\mu\}$ can be realized by $\widetilde X_0 := \mu$ and $\widetilde X_t(\cdot) := \mathcal N[W_t(\cdot)]$ for each $t>0$.
We refer to $(\mathbb N_x)_{x\in E}$ as the \emph{Kuznetsov measures} of $X$.

\subsection{Semigroups for superprocesses}
\label{sec: definition of vf}
Let $X$ be a non-persistent superprocess with its Kuznetsov measure denoted by $(\mathbb N_x)_{x\in E}$.
We define the mean semigroup
\begin{align}
  P_t^{\rho_1} f(x)
  := \Pi_{x}[e^{\int_0^t \rho_1(\xi_s)ds}f(\xi_t) \mathbf 1_{t< \zeta}],
  \quad t\geq 0, x\in E, f\in \mathcal B_b(E,\mathbb R_+).
\end{align}
It is known from \cite[Proposition 2.27]{Li2011Measure-valued} and \cite[Theorem 2.7]{Kyprianou2014Fluctuations} that for all $t > 0$, $\mu \in \mathcal M(E)$ and $f\in \mathcal B_b(E,\mathbb R_+)$,
\begin{align}
  \label{eq: mean formula for superprocesses}
  \mathbb N_{\mu}[W_t(f)]
  =\mathbf P_{\mu}[X_t(f)]
  =\mu(P^{\rho_1}_t f).
\end{align}

Define
\begin{align}
  L_1(\xi)
  &:= \{f\in \mathcal B(E): \forall x\in E, t\geq 0, \quad \Pi_x[|f(\xi_t)|]< \infty\}, \\
  L_2(\xi)
  &:= \{f \in \mathcal B(E): |f|^2 \in L_1(\xi)\}.
\end{align}
Using monotonicity and linearity, we get from \eqref{eq: mean formula for superprocesses}  that
\begin{align}
  \mathbb N_x[W_t(f)]
  = \mathbf P_{\delta_x}[X_t(f)]
  = P^{\rho_1}_t f(x) \in \mathbb R,
  \quad f\in L_1(\xi), t > 0,x\in E.
\end{align}
This says that the random variable $\langle X_t, f\rangle$ is well defined under probability $\mathbf P_{\delta_x}$ provided $f\in L_1(\xi)$.
By the branching property of the superprocess, $X_t(f)$ is an infinitely divisible random variable.
Therefore, we can write
\[
  U_t(\theta f)(x)
  := \operatorname{Log} \mathbf P_{\delta_x}[e^{i \theta X_t( f)}],
  \quad t\geq 0, f\in L_1(\xi), \theta \in \mathbb R, x\in E,
\]
as its characteristic exponent.
According to Campbell's formula, see \cite[Theorem 2.7]{Kyprianou2014Fluctuations} for example, we have
\[
  \mathbf P_{\delta_x} [e^{i\theta X_t(f)}]
  = \exp(\mathbb N_x[ e^{i\theta W_t(f)} - 1]),
  \quad t>0, f\in L_1(\xi), \theta \in \mathbb R, x\in E.
\]
Noticing that $\theta \mapsto \mathbb N_x[e^{i\theta W_t(f)} - 1]$ is a continuous function on $\mathbb R$ and that $\mathbb N_x[e^{i\theta W_t(f)} - 1] = 0$ if $\theta = 0$, according to \cite[Lemma 7.6]{Sato2013Levy}, we have
\begin{align}
  \label{eq: N and characteristic exponent}
  U_t(\theta f)(x)
  = \mathbb N_x[e^{i W_t(\theta f)} - 1],
  \quad t>0, f\in L_1(\xi), \theta \in \mathbb R, x\in E.
\end{align}

\begin{lem}
  There exists a constant $C\geq 0$ such that
  for all $f \in L_1(\xi),x\in E$ and $t\geq 0$, we have
  \begin{align}
    \label{eq: upper bound of psi(v)}
    |\psi (x,-U_tf)|
    \leq C P^{\rho_1}_t |f|(x) + C (P^{\rho_1}_t |f| (x))^2.
  \end{align}
\end{lem}

\begin{proof}
  Noticing that
 $
     e^{\operatorname{Re} U_tf(x)}
    = |e^{U_tf(x)}|
    = |\mathbf P_{\delta_x}[e^{i X_t(f)}]|
    \leq 1,
$
  we have
  \begin{align}
    \label{eq: -v has positive real part}
    \operatorname{Re} U_tf(x)
    \leq 0.
  \end{align}
  Therefore, we can speak of $\psi(x,-U_tf)$ since $z\mapsto \psi(x,z)$ is well defined on $\mathbb C_+$.
  According to Lemma \ref{lem: estimate of exponential remaining}, we have that
  \begin{align}
    \label{eq: upper bound for vf}
    |U_tf(x)|
    \leq \mathbb N_x[|e^{i W_t(f)} - 1|]
    \leq \mathbb N_x[|i W_t(f)|]
    \leq (P^{\rho_1}_t |f|)(x).
  \end{align}
  Notice that, for any compact $K \subset \mathbb R$,
  \begin{align}
    \label{eq: estimate of deriavetive of v(theta)}
    \mathbb N_x \Big[\sup_{\theta \in K} \Big|\frac{\partial}{\partial \theta} (e^{i\theta W_t(f)} - 1) \Big|\Big]
    \leq \mathbb N_x[|W_t(f)|] \sup_{\theta \in K}|\theta|
    \leq (P^{\rho_1}_t |f|)(x) \sup_{\theta \in K}|\theta| < \infty.
  \end{align}
  Therefore, according to \cite[Theorem A.5.2]{Durrett2010Probability} and \eqref{eq: N and characteristic exponent}, $ U_t( \theta f)( x )$ is differentiable in $\theta \in \mathbb R$ with
  \[
    \frac{\partial}{\partial \theta} U_t(\theta f)(x)
    = i\mathbb N_x[W_t(f) e^{i\theta W_t(f)}],
    \quad \theta \in \mathbb R.
  \]
  Moreover, from the above, it is clear that
  \begin{align}
    \label{eq: upper bounded for derivative of v(theta)}
    \sup_{\theta \in \mathbb R}\Big| \frac{\partial}{\partial \theta}U_t(\theta f)(x)\Big|
    \leq ( P^{\rho_1}_t |f|)(x).
  \end{align}
  It follows from the dominated convergence theorem that $(\partial/\partial \theta)U_t(\theta f)(x)$ is continuous in $\theta$.
  In other words, $\theta \mapsto -U_t(\theta f)(x)$ is a $C^1$ map from $\mathbb R$ to $\mathbb C_+$.
  Thus,
  \begin{align}
    \label{eq: path integration representation of psi(v)}
    \psi(x,-U_tf)
    = -\int_0^1 \psi' (x,-U_t(\theta f) ) \frac{\partial}{\partial \theta} U_t(\theta f)(x)~d\theta.
  \end{align}
  Notice that
  \begin{align}
    & |\psi'(x, -U_tf)| \\
    & = \Big| -\rho_1(x)- 2\rho_2(x) U_tf(x)+ \int_{(0,\infty)} y (1- e^{y U_tf(x)} ) \pi(x,dy)\Big| \\
    & = \Big| - \rho_1(x)- 2\rho_2(x)\mathbb N_x[e^{i W_t(f)} - 1]  + \int_{(0,\infty)} y \mathbf P_{y \delta_x}[1-e^{i X_t(f)}] \pi(x,dy) \Big| \\
    & \leq \|\rho_1\|_\infty + 2\rho_2(x)\mathbb N_x[W_t(|f|)]+ \int_{(0,\infty)} y\mathbf P_{y\delta_x}[2\wedge X_t(|f|)] \pi(x,dy) \\
    & \leq \|\rho_1\|_\infty + 2\|\rho_2\|_\infty P^{\rho_1}_t |f|(x) + \Big(\sup_{x\in E}\int_{(0,1]}y^2 \pi(x,dy)\Big)~P^{\rho_1}_t |f|(x) + 2\sup_{x\in E}\int_{(1,\infty)} y \pi(x,dy) \\
    & =: C_1 + C_2(P^{\rho_1}_t |f|)(x), \label{eq: upper bound of psi'(v)}
  \end{align}
where $C_1, C_2$ are constants independent of $f,x$ and $t$.
  Now, combining the display above with \eqref{eq: path integration representation of psi(v)} and \eqref{eq: upper bounded for derivative of v(theta)} we get the desired result.
\end{proof}

This lemma also says that if $f\in L^2(\xi)$, then
$
  \Pi_x\Big[\int_0^t \psi(\xi_s,- U_{t-s}f)ds\Big]
  \in \mathbb C,
 x\in E, t\geq 0,$
is well defined.
In fact, using Jensen's inequality and the Markov property, we have
\begin{align}
  \label{eq: domination of psi(v)}
  & \Pi_x\Big[\int_0^t |\psi (\xi_s,-U_{t-s}f )|ds\Big]
  \leq \Pi_x\Big[\int_0^t (C_1 P_{t-s}^{\rho_1}|f|(\xi_s)+C_2 P_{t-s}^{\rho_1}|f|(\xi_s)^2 )ds\Big] \\
  & \leq \int_0^t (C_1 e^{t\|\rho_1\|}\Pi_x [ \Pi_{\xi_s}[|f(\xi_{t-s})|] ]+C_2 e^{2t\|\rho_1\|}\Pi_x [ \Pi_{\xi_s}[|f (\xi_{t-s})|]^2 ] )~ds \\
  & \leq \int_0^t (C_1 e^{t\|\rho_1\|}\Pi_x [ |f(\xi_{t})|]+C_2e^{2t\|\rho_1\|}\Pi_x [ |f (\xi_{t})|^2 ])~ds < \infty.
\end{align}

\subsection{A complex-valued non-linear integral equation}
Let $X$ be a non-persistent superprocess.
In this subsection, we will prove the following:

\begin{prop}
  \label{prop: complex FKPP-equation}
  If $f\in L_2(\xi)$,  then for all $t\geq 0$ and $x\in E$,
\begin{align}
  \label{eq: complex FKPP-equation}
  U_tf(x) - \Pi_x \Big[\int_0^t \psi (\xi_s, - U_{t-s}f ) ds \Big]
  = i \Pi_x [f(\xi_t)].
\end{align}
\begin{align}
  \label{eq: complex FKPP-equation with FK-transform}
  U_tf(x) -  \int_0^t P_{t-s}^{\rho_1} \psi_0(\cdot,-U_sf) (x)~ds
  = iP_t^{\rho_1} f(x).
\end{align}
\end{prop}

To prove this, we will need the generalized spine decomposition theorem from \cite{RenSongSun2017Spine}.
Let $f\in \mathcal B_b(E,\mathbb R_+)$, $T >0$ and $x\in E$.
Suppose that $\mathbf P_{\delta_x}[X_T(f)] = \mathbb N_x[ W_T(f)] = P^{\rho_1}_T f(x) \in (0,\infty)$, then we can define the following probability transforms:
\begin{align}
  d\mathbf P_{\delta_x}^{ X_T(f)}
  := \frac{X_T(f)}{P_T^{\rho_1} f(x)} d\mathbf P_{\delta_x};
  \quad d\mathbb N_x^{W_T(f)}
  :=  \frac{W_T(f)}{P_T^{\rho_1} f(x)} d\mathbb N_x.
\end{align}
Following the definition in \cite{RenSongSun2017Spine}, we say that $\{\xi, \mathbf n;\mathbf Q_{x}^{(f,T)}\}$ is a spine representation of $\mathbb N_x^{\langle W_T, f\rangle}$ if
\begin{itemize}
\item
  the spine process $\{(\xi_t)_{0\leq t\leq T}; \mathbf Q^{(f,T)}_x\}$ is a copy of $\{(\xi_t)_{0\leq t\leq T}; \Pi^{(f,T)}_{x}\}$, where
  \begin{align}
    d\Pi_x^{(f,T)}
    := \frac{f(\xi_T)e^{\int_0^T \rho_1(\xi_s)ds}}{P^{\rho_1}_T f(x)} d \Pi_x;
  \end{align}
\item
  given $\{(\xi_t)_{0\leq t\leq T}; \mathbf Q^{(f,T)}_x\}$, the immigration measure
$
  \{\mathbf n(\xi,ds,dw); \mathbf Q^{(f,T)}_x[\cdot |(\xi_t)_{0\leq t\leq T}]\}
$
is a Poisson random measure on $[0,T] \times \mathbb W$ with intensity
\begin{align}
  \label{eq: conditional intensity}
  \mathbf m(\xi,ds,dw)
  := 2 \rho_2(\xi_s) ds \cdot \mathbb N_{\xi_s}(dw) + ds \cdot \int_{y\in (0,\infty)} y \mathbf P_{y\delta_{\xi_s}}(X\in dw) \pi(\xi_s,dy);
\end{align}
\item
  $\{(Y_t)_{0\leq t\leq T}; \mathbf Q^{(f,T)}_x\}$ is an $\mathcal M(E)$-valued process defined by
  \begin{align}
    Y_t
    := \int_{(0,t] \times \mathbb W} w_{t-s} \mathbf n(\xi,ds,dw),
    \quad 0 \leq t\leq T.
  \end{align}
\end{itemize}
According to the spine decomposition theorem in \cite{RenSongSun2017Spine}, we have that
\begin{align}
  \label{eq: Spine decomposition 1}
  \{(X_s)_{s \geq 0};\mathbf P_{\delta_x}^{X_T(f)}\}
  \overset{f.d.d.}{=} \{(X_s + W_s)_{s \geq 0};\mathbf P_{\delta_x} \otimes \mathbb N_x^{W_T(f)} \},
\end{align}
\begin{align}
  \label{eq: Spine decomposition 2}
  \{(W_s)_{0\leq s\leq T};\mathbb N_x^{W_T(f)}\}
  \overset{f.d.d.}{=} \{(Y_s)_{s \geq 0};\mathbf Q_x^{(f,T)}\}.
\end{align}

\begin{proof}[Proof of Proposition \ref{prop: complex FKPP-equation}]
  Assume that $f\in \mathcal B_b(E)$.
  Fix $t>0, r\in [0,t), x\in E$ and a strictly positive $g\in \mathcal B_b(E)$.
  Denote by $\{\xi, \mathbf n; \mathbf Q_x^{(g,t)}\}$ the spine representation of $\mathbb N_x^{W_t(g)}$.
  Conditioned on $\{\xi; \mathbf Q_x^{(g,t)}\}$, denote by $\mathbf m(\xi, ds,dw)$ the conditional intensity of $\mathbf n$ in \eqref{eq: conditional intensity}.
  Denote by $\Pi_{r,x}$ the probability of Hunt process $\{\xi; \Pi\}$ initiated at time $r$ and position $x$.
  From Lemma \ref{lem: estimate of exponential remaining}, we have $\mathbf Q^{(g,t)}_{x}$-almost surely
  \begin{align}
    & \int_{[0,t]\times \mathbb W}|e^{i w_{t-s}(f)} - 1| \mathbf m(\xi, ds,dw)
      \leq \int_{[0,t]\times \mathbb W} (| w_{t-s}(f)| \wedge 2 ) \mathbf m(\xi, ds,dw) \\
    & \leq \int_0^t \Big(2\rho_2(\xi_s)\mathbb N_{\xi_s} ( W_{t-s}(|f|) )  + \int_{(0,1]} y \mathbf P_{y \delta_{\xi_s}}[X_{t-s}(|f|)] \pi(\xi_s,dy)
   + 2\int_{(1,\infty)}y\pi(\xi_s,dy)\Big) ds
    \\ & \leq \int_0^t (P_{t-s}^{\rho_1} |f|)(\xi_s)\Big(2\rho_2(\xi_s)  + \int_{(0,1]} y^2 \pi(\xi_s,dy)\Big) ds + 2t \sup_{x\in E}\int_{(1,\infty)}y\pi(x,dy)
    \\ & \leq \Big(2\|\rho_2\|_\infty +\sup_{x\in E}\int_{(0,1]} y^2 \pi(x,dy)\Big) t e^{t\|\rho_1\|_\infty}\|f\|_\infty + 2t \sup_{x\in E}\int_{(1,\infty)}y\pi(x,dy)
         < \infty.
  \end{align}
  Using this, Fubini's theorem, \eqref{eq: N and characteristic exponent} and \eqref{eq: -v has positive real part} we have $\mathbf Q^{(g,t)}_{x}$-almost surely,
  \begin{align}
    & \int_{[0,t]\times \mathbb N}(e^{i  w_{t-s}(f)} - 1) \mathbf m(\xi, ds,dw)
    \\ & =\int_0^t \Big(2\rho_2(\xi_s)\mathbb N_{\xi_s}(e^{i W_{t-s}(f)} - 1)  + \int_{(0,\infty)} y \mathbf P_{y \delta_{\xi_s}}[e^{i X_{t-s}(f)} - 1] \pi(\xi_s,dy)\Big) ds
    \\ & =\int_0^t \Big( 2\rho_2(\xi_s) U_{t-s} f(\xi_s) + \int_{(0,\infty)} y (e^{y U_{t-s}f(\xi_s)} - 1) \pi(\xi_s,dy) \Big) ds
    \\ & = -\int_0^t \psi'_0 (\xi_s, -U_{t-s}f )ds.
  \end{align}
  Therefore, according to \eqref{eq: Spine decomposition 2}, Campbell's formula and above, we have that
  \begin{align}
    \label{eq: N to Pi}
    & \mathbb N_x^{ W_t( g)}[e^{i W_t(f)}]
      = \mathbf Q_x^{(g,t)} \Big[\exp\Big\{\int_{[0,t]\times \mathbb N}(e^{i w_{t-s}(f)} - 1) \mathbf m(\xi, ds,dw)\Big\}\Big]
    \\ & = \Pi_x^{(g,t)} [e^{-\int_0^t \psi'_0(\xi_s, -U_{t-s}f)ds}]
    = \frac{1}{P_t^{\rho_1} g (x)} \Pi_x[ g(\xi_t) e^{-\int_0^t \psi'(\xi_s, -U_{t-s}f)ds} ].
  \end{align}
  Let $\epsilon >0$.
  Define $f^+ = (f \vee 0) + \epsilon$ and $f^- = (-f) \vee 0 + \epsilon$, then $f^\pm$ are strictly positive and $f = f^+ - f^-$.
  According to \eqref{eq: Spine decomposition 1}, we have that
  \begin{align}
    \frac{\mathbf P_{\delta_x}[X_t(f^{\pm}) e^{i X_t(f)}]}{\mathbf P_{\delta_x}[X_t(f^{\pm}) ]}
    = \mathbf P_{\delta_x}[e^{i X_t(f)}] \mathbb N_x^{ W_t(f^{\pm})}[e^{i X_t(f)}].
  \end{align}
  Using \eqref{eq: N to Pi} and the above, we have
  \begin{align}
    \frac{\mathbf P_{\delta_x}[X_t(f) e^{i X_t(f)}] }{\mathbf P_{\delta_x}[e^{i X_t(f)}]}
    & = \mathbf P_{\delta_x}[X_t(f^+)] \mathbb N_x^{W_t(f^+)} [e^{i X_t(f)}] - \mathbf P_{\delta_x}[X_t(f^-)]\mathbb N_x^{W_t(f^-)}[e^{i X_t(f)}]
    \\ & = \Pi_x[ f(\xi_t) e^{- \int_0^t \psi'(\xi_s, -U_{t-s}f) ds}  ].
  \end{align}
  Therefore, we have
  \begin{align}
    \frac{\partial}{\partial \theta} {U_t(\theta f)(x)}
    = \frac{\mathbf P_{\delta_x}[i X_t(f) e^{i X_t(f)}] }{\mathbf P_{\delta_x}[e^{i X_t(f)}]}
    = \Pi_x[ if(\xi_t) e^{ - \int_0^t \psi'(\xi_s, -U_{t-s}(\theta f)) ds} ].
  \end{align}
  Since $\{(\xi_{r+t})_{t \geq 0}; \Pi_{r,x}\} \overset{d}{=} \{(\xi_{t})_{t\geq 0}; \Pi_{x}\} $, we have
  \begin{align}
    & \frac{\partial}{\partial \theta} U_{t-r}(\theta f)( x)
      = \Pi_x[ i f(\xi_{t-r}) e^{-\int_0^{t-r} \psi'(\xi_s, -U_{t-r-s}(\theta f)) ds} ] \\
    & = \Pi_{r,x}[i f(\xi_t)e^{-\int_0^{t-r} \psi'(\xi_{r+s}, -U_{t-r-s}(\theta f)) ds} ]
      = \Pi_{r,x}[if(\xi_t)e^{-\int_r^t \psi'(\xi_{s}, -U_{t-s}(\theta f)) ds} ].
  \end{align}

  From \eqref{eq: upper bound of psi'(v)}, we know that for each $\theta\in \mathbb R$, $(t,x) \mapsto |\psi'(x,-U_tf(x))|$ is locally bounded (i.e. bounded on $[0,T]\times E$ for each $T \geq 0$).
  Therefore, we can apply Lemma \ref{eq: complex FK} and get that
  \[
    \frac{\partial}{\partial \theta} U_{t-r}(\theta f)(x) + \Pi_{r,x} \Big[\int_r^t \psi' (\xi_s,- U_{t-s}(\theta f) )\frac{\partial}{\partial \theta} U_{t-s}(\theta f)(\xi_s)~ds\Big]
    = \Pi_{r,x} [i f(\xi_t)]
  \]
  and
  \begin{align}
    & \frac{\partial}{\partial \theta} U_{t-r}(\theta f)(x) + \Pi_{r,x} \Big[\int_r^t e^{\int_r^s \rho_1(\xi_u)du}\psi_0' (\xi_s,- U_{t-s}(\theta f) )\frac{\partial}{\partial \theta} U_{t-s}(\theta f)(\xi_s)~ds\Big]\\
    & = \Pi_{r,x} [i e^{\int_r^t \rho_1(\xi_s)ds}f(\xi_t)].
  \end{align}
  Integrating the two displays above with respect to $\theta$  on [0,1], using
  Fubini's theorem, \eqref{eq: upper bounded for derivative of v(theta)}, \eqref{eq: path integration representation of psi(v)} and \eqref{eq: upper bound of psi'(v)}, we get
  \begin{align}
    U_{t-r}f(x) - \Pi_{r,x} \Big[\int_r^t \psi (\xi_s,-U_{t-s}f ) ~ds\Big]
    = i \theta \Pi_{r,x} [f(\xi_t)]
  \end{align}
  and
  \begin{align}
    U_{t-r}f(x) - \Pi_{r,x} \Big[\int_r^t e^{\int_r^s \rho_1(\xi_u)du} \psi_0 (\xi_s,- U_{t-s}f ) ~ds\Big]
    = i \Pi_{r,x} [e^{\int_r^t\rho_1(\xi_u)du}f(\xi_t)].
  \end{align}
  Taking $r = 0$, we get that \eqref{eq: complex FKPP-equation} and \eqref{eq: complex FKPP-equation with FK-transform} are true if $f\in \mathcal B_b(E)$.

  The rest of the proof is to evaluate \eqref{eq: complex FKPP-equation} and \eqref{eq: complex FKPP-equation with FK-transform} for all $f\in L_2(\xi)$. We only do this for \eqref{eq: complex FKPP-equation} since the argument for \eqref{eq: complex FKPP-equation with FK-transform} is similar.
  Let $n \in \mathbb N$.
  Writing $f_n := (f^+ \wedge n) - (f^- \wedge n)$, then $f_n \xrightarrow[n\to \infty]{} f$ pointwise.
  From what we have proved, we have
  \begin{align}
    \label{eq: complex FKPP-equation for fn}
    U_tf_n(x) - \Pi_{x} \Big[\int_0^t \psi (\xi_s, - U_{t-s}f_n ) ~ds\Big]
    = i \Pi_{x} [f_n(\xi_t)].
  \end{align}
  Note that
  (i) $\Pi_{x}[f_n(\xi_t)] \xrightarrow[n\to \infty]{} \Pi_{x}[f(\xi_t)]$;
  (ii) by \eqref{eq: N and characteristic exponent}, the dominated convergence theorem and the fact that
  \[
     |e^{i W_t(f_n)} - 1| \leq W_t(|f|);
    \quad \mathbb N_x[W_t(|f|)] = (P_t^{\rho_1} |f|)(x) < \infty,
 \]
  we have $U_tf_n(x) \xrightarrow[n\to \infty]{} U_tf(x)$, and (iii) by the dominated convergence theorem, \eqref{eq: domination of psi(v)} and the fact (see \eqref{eq: upper bound of psi(v)}) that
  \[
    |\psi(\xi_s,- U_{t-s}f_n) |
    \leq C_1 P_{t-s}^{\rho_1}|f|(\xi_s)+C_2 P_{t-s}^{\rho_1}|f|(\xi_s)^2,
  \]
  we get that $\Pi_{x} [\int_0^t \psi(\xi_s,- U_{t-s}f_n)ds] \xrightarrow[n\to \infty]{} \Pi_{x} [\int_0^t \psi(\xi_s,- U_{t-s}f)ds]$.
  Using these, letting $n \to \infty$ in \eqref{eq: complex FKPP-equation for fn}, we get the desired result.
\end{proof}

\subsection*{Acknowledgment}
We thank Zenghu Li and Rui Zhang for helpful conversations.
We also thank the referee for very helpful comments.

\end{document}